\documentclass[12pt, a4paper]{amsart}

\usepackage{graphicx}
\usepackage{wrapfig}
\usepackage{comment}
\usepackage{lipsum}

\usepackage{pifont}

\usepackage[
unicode=true,
plainpages = false,
pdfpagelabels,
bookmarks=true,
bookmarksnumbered=true,
bookmarksopen=true,
breaklinks=true,
backref=false,
colorlinks=true,
linkcolor = DarkBlue,
urlcolor  = DarkBlue,
citecolor = DarkRed,
anchorcolor = green,
hyperindex = true,
hyperfigures
]{hyperref}

\hypersetup{
pdftitle={An explicit algorithm for the Higman Embedding Theorem},
pdfauthor={V.H. Mikaelian},
pdfsubject={An explicit algorithm for the Higman Embedding Theorem}
}

\usepackage[dvipsnames]{xcolor}
\definecolor{DarkRed}{HTML}{cc0000} 
\definecolor{DarkBlue}{HTML}{0000cc}

\usepackage[top=2.5cm, bottom=2.5cm, outer=2.5cm, inner=2.5cm, heightrounded
]{geometry}

\renewcommand{\S}{{\mathcal S}}
\newcommand{\Zz}{{\mathcal Z}}
\newcommand{\E}{{\mathcal E}}

\newcommand{\X}{{\mathfrak X}}
\newcommand{\Y}{{\mathfrak Y}}

\newcommand{\x}{{\mathbf x}}
\newcommand{\y}{{\mathbf y}}
\newcommand{\RR}{{\mathbf R}}

\newcommand{\N}{{\mathbb N}}
\newcommand{\Z}{{\mathbb Z}}
\newcommand{\Q}{{\mathbb Q}}
\newcommand{\Co}{{\mathbb C}}

\newcommand{\1}{\{1\}}

\newcommand{\bigast}{\textrm{\footnotesize\ding{91}}}

\usepackage[bitstream-charter]{mathdesign}

\DeclareSymbolFont{cmsymbols}{OMS}{cmsy}{m}{n}
\SetSymbolFont{cmsymbols}{bold}{OMS}{cmsy}{b}{n}
\DeclareSymbolFontAlphabet{\mathcal}{cmsymbols}

\theoremstyle{plain}
\newtheorem{Theorem}{Theorem}[section]
\newtheorem{Lemma}[Theorem]{Lemma}
  
\newtheorem{Corollary}[Theorem]{Corollary}

\theoremstyle{definition}

\newtheorem{Algorithm}[Theorem]{Algorithm}

\theoremstyle{remark}
\newtheorem{Remark}[Theorem]{Remark} 
\newtheorem{Example}[Theorem]{Example} 

\numberwithin{equation}{section}

\theoremstyle{plain}
\newtheorem{innercustomgeneric}{\customgenericname}
\providecommand{\customgenericname}{}
\newcommand{\newcustomtheorem}[2]{%
\newenvironment{#1}[1]
{%
\renewcommand\customgenericname{#2}%
\renewcommand\theinnercustomgeneric{##1}%
\innercustomgeneric
}
{\endinnercustomgeneric}
}

\newcustomtheorem{TheoremABC}{Theorem}

\usepackage{enumitem} 
\setlist[itemize]{leftmargin=*}
\setlist[enumerate]{labelsep=*, leftmargin=1.5pc}
\setlist[enumerate,1]{label=\bf \arabic*., ref=\arabic*}
\setlist[enumerate,2]{label=\rm \emph{\alph*}),ref=\theenumi.\emph{\alph*}}

\begin{document}



\title[An explicit algorithm for the Higman Embedding Theorem]{An explicit algorithm for the Higman Embedding Theorem}

\author{V.\,H. Mikaelian}
 
\begin{abstract}
We propose an algorithm which, for any  recursive group $G$, given by its effectively enumerable generators and recursively enumerable relations, outputs an explicit embedding of $G$ into a finitely presented group directly written by its generators and defining relations.
This is the explicit analogue of the celebrated Higman Embedding Theorem stating that a fi\-nitely generated group $G$ is embeddable into a finitely presented group if and only if $G$ is recursive.
The constructed finitely presented group can even be chosen to be $2$-generator. 
This algorithm has already been applied, for example, to the additive group of rational numbers $\Q$, which is clearly recursive. 
The question of explicitly embedding $\mathbb{Q}$ into a finitely presented group has been highlighted in the literature by Johnson, de la Harpe, Bridson, and others.
The suggested method can also be used to solve the embedding problem for various other recursive groups.
The embedding algorithm is developed using conventional free constructions, including free products with amalgamation, HNN-extensions, and also their modifications,  such as the auxiliary $\bigast$-constructions. 
We also analyze the steps of the original Higman embedding to precisely indicate which of its components are non-explicit. 
\end{abstract}

%

\date{\today}

\maketitle

\begin{flushright}
\vskip-4mm
\small
\textit{``We do fundamental research, not only to acquire\\  results solely, but because the process is an ennobling one''}\\
{\footnotesize 
Graham Higman, 1987}
\end{flushright}

\makeatletter
\renewcommand\tocchapter[3]{%
\indentlabel{\@ifnotempty{#2}{\ignorespaces#2.\quad}}#3\dotfill%
}
\renewcommand{\tocsection}[3]{%
\indentlabel{\@ifnotempty{#2}{\ignorespaces#1 #2.\quad}}#3\dotfill%
}
\renewcommand{\tocsubsection}[3]{%
\indentlabel{\@ifnotempty{#2}{\ignorespaces#1 #2.\quad}}#3\dotfill%
}
\renewcommand{\tocsubsubsection}[3]{%
\indentlabel{\@ifnotempty{#2}{\ignorespaces#1 #2.\quad}}#3\dotfill%
}
\makeatother

\let\oldtocsection=\tocsection
\let\oldtocsubsection=\tocsubsection
\let\oldtocsubsubsection=\tocsubsubsection
\renewcommand{\tocsection}[2]{\small \hspace{-12pt}\oldtocsection{#1}{#2}}
\renewcommand{\tocsubsection}[2]{\footnotesize \hspace{4.5pt} 
\oldtocsubsection{#1}{#2}}
\renewcommand{\tocsubsubsection}[2]{\footnotesize \hspace{34.0pt}\oldtocsubsubsection{#1}{#2}}

\setcounter{tocdepth}{2}

\tableofcontents

\section{Introduction}
\label{SE Introduction}


\noindent
The Higman Embedding Theorem proven in the fundamental article \cite{Higman Subgroups of fP groups} implies that a recursive group $G = \langle\, X \mathrel{|} R \,\rangle$ given by its finite or effectively enumerable countable set of generators $X$, and by its recursively enumerable set of defining relations $R$, can be embedded into a finitely presented group $\mathcal G$, see Theorem 1 and Corollary on page 456 in \cite{Higman Subgroups of fP groups}.

Our objective is to establish an \textit{explicit} analogue of this theorem, that is, to construct an algorithm which, given a recursive group $G$, outputs an explicit embedding of $G$ into a group $\mathcal G$ recorded by a finite presentation, see Algorithm~\ref{AL Algorithm for explicit embedding of a recursive group} in Section~\ref{SU The explicit embedding algorithm} below.

Higman proved the \textit{existence} of such an embedding for every recursive group $G$; his argument, however, was not designed for writing down a finite presentation of $\mathcal{G}$.
In Chapter~\ref{SE Higman's original construction} we study Higman’s original construction in detail and identify the steps that are not carried out explicitly, see Section~\ref{SU Is Higman's construction explicit} including Point~\ref{SU Conclusion}.

The reader interested only in the reasoning behind Algorithm~\ref{AL Algorithm for explicit embedding of a recursive group}, and not in the analysis of Higman’s original method \cite{Higman Subgroups of fP groups}, may skip to chapters~\ref{SE The modified construction of the current work}\,--\ref{SE Theorem B and the final embedding}.

The relevant references for the terms effective, recursive, explicit, etc., are collected in Section~\ref{SU Necessary terms on computability}. There we also indicate the sense in which Higman used these terms.

Integer-valued functions $f$, the Higman operations \eqref{EQ Higman operations}, benign subgroups, etc., are introduced in sections~\ref{SU Integer functions f}--\ref{SU Benign subgroups and Higman operations}; they are needed both in the preliminary Chapter~\ref{SE Higman's original construction} and in the main proof in chapters~\ref{SE The modified construction of the current work}\,--\ref{SE Theorem B and the final embedding}.

The basic terminology and references for combinatorial group theory and free constructions are collected in Sections~\ref{SU Free constructions}\,--\ref{SU The conjugates collecting process}. The proofs below depend on constructions that we developed in 
\cite{Embeddings using universal words, 
The Higman operations and embeddings, 
Auxiliary free constructions for explicit embeddings, 
On explicit embeddings of Q}. 
Since we want to avoid any  repeating material from these earlier articles, we state some necessary results without proofs in sections~\ref{SU The embedding alpha into a 2-generator group},
\ref{SU The *-construction},
\ref{SU Construction of the group A},
\ref{SU Computing the conjugation of by d j}.

\subsection{The explicit embedding algorithm}
\label{SU The explicit embedding algorithm} 
A detailed description of the explicit embedding method, with examples and references to the relevant proofs, is given in Chapter~\ref{SE The modified construction of the current work}. The embedding process can be outlined by the following pseudocode:

\begin{Algorithm}[Explicit embedding of a recursive group into a finitely presented group]
\label{AL Algorithm for explicit embedding of a recursive group}
We are given a recursive group $G = \langle\, X \mathrel{|} R \,\rangle$ defined on an effectively enumerable alphabet $X=\{a_1, a_2,\ldots \}$ by a set of recursively enumerable relations $R$.

\vskip1mm
\noindent
Output an explicit embedding of $G$ into a finitely presented group $\mathcal G$ given by its generators and defining relations explicitly.
\begin{enumerate}
\item
\label{Step 1 AL Algorithm for explicit embedding of a recursive group}
Mapping each generator $a_i$ to the ``universal word'' $a_i(x,y)$ from \eqref{EQ definition of a_i(x,y)} 
in Section~\ref{SU The embedding alpha into a 2-generator group}, 
write for each relation $w\in R$  a new relation $w'(x,y)$ in $F_2=\langle x,y \rangle$, and denote by $R'$ the set of all such new $w' (x,y)$.
Set $T_{\!G} = F_2 / \langle\, R' \,\rangle^{\! F_2}$ to be the $2$-generator group explicitly given via $\langle\, x,y \mathrel{|} R' \,\rangle$. 
The group $T_{\!G}$ is recursive by 
Corollary~\ref{CO if G is recursive then TG is also recirsive}. 

\item  
\label{Step 2 AL Algorithm for explicit embedding of a recursive group}
Construct the injective embedding $\alpha: G \to  T_G$ by mapping each $a_i$ to its image $a_i(x,y)$ in $T_{\!G}$, 
see Theorem~\ref{TH universal embedding}.

\item
\label{Step 3 AL Algorithm for explicit embedding of a recursive group}
For each relation \eqref{EQ random relation on x and y} on two variables $x,y$ in $R'$, output the sequence of integers 
\eqref{EQ f_i occurs first}, 
see Section~\ref{SU Using the 2-generator group to get the set X}. Denote $\mathcal X$ to be the set of such sequences ``coding'' the set of words $R'$. 

\item
\label{Step 4 AL Algorithm for explicit embedding of a recursive group}
Construct $\mathcal X$ 
from $\Zz$ and $\S$ via the Higman operations \eqref{EQ Higman operations} in one of \textit{two} ways. Either by the \textit{long} process following Higman's original method involving the  functions $f(n,r)$, $a(r)$, $b(r)$ in \cite{Higman Subgroups of fP groups}, see sections \ref{SU Is Higman's construction explicit} and \ref{SU Writing X via Higman operations} below. Or, alternatively, when $\mathcal X$ satisfies certain conditions, use the much \textit{shorter} process from \cite{The Higman operations and  embeddings}, see Section~\ref{SU Writing X via Higman operations} below.

\item 
\label{Step 5 AL Algorithm for explicit embedding of a recursive group}
The subgroup $A_{\!\mathcal X}$ corresponding to $\mathcal X$, see Section~\ref{SU Defining subgroups by integer sequences}, is benign in $F_3=\langle a, b, c\rangle$, see Section~\ref{SU Benign subgroups and Higman operations}: there is a finitely presented overgroup  $K_{\!\mathcal X}$ of $F_3$ and a finitely generated subgroup $L_{\!\mathcal X} \le K_{\!\mathcal X}$ such that $F_3 \cap L_{\!\mathcal X} = A_{\!\mathcal X}$. 
Using the proof steps for Theorem~\ref{TH Theorem A} in Chapter~\ref{SE Theorem A and its proof steps}, explicitly output $K_{\!\mathcal X}$ by its generators and defining relations, and indicate $L_{\!\mathcal X}$ by its generators. Namely, write the groups  
$K_{\!\mathcal X}$, $L_{\!\mathcal X}$ for two sets 
$\mathcal X = \Zz, \S$, see Section~\ref{SU The proof for the case of Z and S}, and then for each of operations \eqref{EQ Higman operations} (used to construct $\mathcal X$ from from $\Zz$ and $\S$ in Step~\ref{Step 4 AL Algorithm for explicit embedding of a recursive group}) apply sections~\ref{SU The proof for iota and upsilon}, \ref{SU The proof for rho}\,--\,\ref{SU The proof for omega_m} outputting a new pair $K_{\!\mathcal X}$, $L_{\!\mathcal X}$ after each operation. 

\item
\label{Step 6 AL Algorithm for explicit embedding of a recursive group}
Load the final groups $K_{\!\mathcal X}$, $L_{\!\mathcal X}$ outputted in previous Step~\ref{Step 5 AL Algorithm for explicit embedding of a recursive group} into Theorem~\ref{TH Theorem B} to build the explicit embedding $\beta: T_{\!\mathcal X} \to \mathcal G$ of the group $T_{\!\mathcal X}$ from 
\eqref{EQ introducing T_X} into a finitely presented group $\mathcal G$, see sections~\ref{SU If AX is benign then}\,--\,\ref{SU The Higman Rope Trick} including the ``The Higman Rope Trick'' and the explicit presentation of $\mathcal G$ in \eqref{EQ relations of G FULL}. 

\item
\label{Step 7 AL Algorithm for explicit embedding of a recursive group}
Since our process was via the $2$-generator group $T_{\!G}$, the above $T_{\!\mathcal X}$ from 
Step~\ref{Step 6 AL Algorithm for explicit embedding of a recursive group} coincides with $T_{\! G}$ from Step~\ref{Step 1 AL Algorithm for explicit embedding of a recursive group}.
I.e., as an explicit embedding $\varphi$ of the initial group $G$ into the finitely presented group $\mathcal G$ we can output the composition $\varphi: G \to \mathcal G$ of $\alpha: G \to  T_G$ from \eqref{EQ embedding alpha} with $\beta: T_{\!\mathcal X} \to  \mathcal G$ from \eqref{EQ define beta}, see Section~\ref{SU Equality and the final embedding}.

\item
\label{Step 8 AL Algorithm for explicit embedding of a recursive group}
[\textit{Optional}] 
$G$ can be embedded into a $2$-generator finitely presented group $T_{\!\mathcal G}$ by the composition $\psi$ of $\varphi: G \to \mathcal G$ from Step~\ref{Step 7 AL Algorithm for explicit embedding of a recursive group} with the embedding $\gamma : \mathcal G \to T_{\!\mathcal G}$ in Section~\ref{SU Embedding G into the 2-generator group TG}.
\end{enumerate}
\end{Algorithm}

\subsection{Application of the algorithm for $\Q$}
\label{SU Application of the algorithm for Q}

As non-trivial applications of Algorithm~\ref{AL Algorithm for explicit embedding of a recursive group} we in \cite{On explicit embeddings of Q} suggested explicit embeddings of the additive group of rational numbers $\Q$ into a finitely presented group $\mathcal{Q}$, and into a $2$-generator finitely presented group $T_{\!\mathcal Q}$.

Being countable, the group $\Q$ certainly has an embedding into a finitely generated  group by the
remarkable theorem of Higman, Neumann and Neumann \cite{HigmanNeumannNeumann}. Since $\Q$ apparently has a recursive presentation, such as \eqref{EQ Q genetic code}, the above finitely generated group can even be finitely presented by \cite{Higman Subgroups of fP groups}.
The question of whether this (natural) embedding can be \textit{explicit}, is addressed
in Johnson's work \cite{Johnson on Higman's interest},\,
on page 53 in the monograph of De la Harpe  \cite{De La Harpe 2000},\,
in Problem 14.10\;(a) of Kourovka Notebook \cite{kourovka} by Bridson and De la Harpe,  etc. Discussing the Higman Embedding
Theorem in \cite{Johnson on Higman's interest}, Johnson for certain recursive groups builds their explicit embeddings into finitely presented ones. Then he expresses his gratitude to Higman for raising that problem, and mentions: \textit{``Our main aim, of embedding in a finitely presented group the additive group of rational numbers continues to elude us''}, see page 416 in \cite{Johnson on Higman's interest}.

It is interesing to compare Johnson's citation with the remark of Bridson and Nyberg-Brodda on page 12 in 
\cite{Bridson Nyberg-Brodda 2025}, where  this problem has been called \textit{``This particular challenge, which was a personal favorite of Higman''}, compare with Point~\ref{SU An observation by Bridson and Nyberg-Brodda} below.  

\smallskip
The steps of Algorithm~\ref{AL Algorithm for explicit embedding of a recursive group} have been applied for the group $\Q$ in \cite{On explicit embeddings of Q}: a finitely presented group $\mathcal{Q}$, and a finitely presented $2$-generator group $T_{\!\mathcal Q}$ with the embeddings 
$\varphi: \Q \to \mathcal Q$
and 
$\psi: \Q \to T_{\!\mathcal Q}$
are explicitly given in sections 9.1 and 9.2 in \cite{On explicit embeddings of Q}. 
Both of these embeddings were reported earlier in 2018 \cite{Explicit embeddings Moscow 2018}, without any proofs yet.

The first explicit examples with proofs for finitely presented groups holding $\Q$, were presented by 
Belk, Hyde and Matucci in \cite{Belk Hyde Matucci}.  
Their first embedding of $\Q$ in \cite{Belk Hyde Matucci} is into the group $\overline T$ from \cite{Ghys Sergiescu}, namely, $\overline T$ is the group of all piecewise-linear homeomorphisms $f:\mathbb R \to \mathbb R$ satisfying certain specific requirements, see Theorem 1 in \cite{Belk Hyde Matucci}. 
The second finitely presented group is the automorphism group of Thompson's group $F$, see Theorem 2 in \cite{Belk Hyde Matucci}. These embeddings allow further variations, say, $\overline T$ (together with its subgroup $\Q$) admits an embedding into two specific finitely presented simple groups $T\! \mathcal A$ and $V\! \mathcal A$.

See also the discussion in \cite{Mathoverflow An explicit example of a finitely presented group} where some additional mentions to this problem are attested, including personal recollections about how this question has been mentioned at the past conferences by Martin Bridson, Pierre de la Harpe, Laurent Bartholdi. 

Further ways of embedding of $\Q$ into a finitely \textit{generated} group can be deduced from other results in the literature, such as:
Theorem~2 in \cite{Hall 59 embeds Q}; \,
Theorem~31.2 and Corollary~31.2 in \cite{Olshanskii book} (compare to Theorem 1 and Corollary 1 in \cite{Ashmanov Olshanskii}), etc.

The current topic also concerns the \textit{Boone-Higman conjecture} proposing that a finitely generated group has solvable word problem if and only if it can be embedded into a finitely presented simple group. The survey \cite{Belk Bleak Matucci Zaremsky} puts into context known results on this conjecture and the construction in \cite{Belk Hyde Matucci}. 

Also, Section~3 in \cite{Belk Bleak Matucci Zaremsky} outlines some motivation for development of interesting parallelism between group theory and mathematical logic, leading to the Higman Embedding Theorem and to Boone-Higman conjecture. 

\begin{Remark}
\label{RE The embedding of Q must be natural}
Specifically, Problem 14.10\;(a) of Kourovka Notebook   asks about a \textit{natural} embedding of $\mathbb Q$ into a finitely presented group. The large presentation of the overgroup $T_{\!\mathcal Q}$ constructed in \cite{On explicit embeddings of Q} is not natural. However, the problem of embedding of $\Q$ into a finitely presented group can be found in the literature without the requirement ``natural', see \cite{Johnson on Higman's interest} mentioned at the beginning of this section.
\end{Remark}

\subsection{The problem of explicit  embedding for $GL(n, \Q)$, other possible problems}
\label{SU The problem of explicit  embedding for GL(n Q)}
Problem 14.10 in \cite{kourovka} contains one more question which in the current edition of Kourovka is coined as Problem 14.10\;(c): 
{\it ``Find an explicit and ``natural'' finitely presented group $\Gamma_n$ and an embedding
of $GL(n, \Q)$ in $\Gamma_n$.
Another phrasing of the same problems is: find a simplicial complex $X$ which
covers a finite complex such that the fundamental group of $X$ is $\Q$ or, respectively,
$GL(n, \Q)$''.}\;
This question was recently mentioned in \cite{Belk Bleak Matucci Zaremsky} also.

We would like to announce that, as another application of Algorithm~\ref{AL Algorithm for explicit embedding of a recursive group}, 
an explicit embedding of $GL(n, \Q)$ into some finitely presented  $\Gamma_n$ reflecting this question is suggested. Moreover, that group $\Gamma_n$ can even be $2$-generator \cite{Explicit embeddings GL(n Q)}.

\medskip
If the reader is aware of \textit{any other open problems} on explicit embeddings of recursive groups into finitely presented groups, it would be great to mention them in the discussion \cite{Mathoverflow Problems on explicit embeddings of recursive groups}. It goes without saying that all credit will be gratefully acknowledged in any publication in which we use the information kindly provided in  \cite{Mathoverflow Problems on explicit embeddings of recursive groups}.

\subsection{Other embedding methods}
\label{SU Other embedding methods}

Higman's original proof in \cite{Higman Subgroups of fP groups} is very elaborate, and this mo\-ti\-vated other authors to suggest alternative proofs for it.

Lindon and Schupp present in Section IV\,\!.\,\!7 of \cite{Lyndon Schupp} a proof related to Diophantine subsets of $\Z^n$.
Solving Hilbert's Tenth Problem Matiyasevich establishes  that a subset of $\Z^n$ is recursively enumerable if and only if it is Diophantine \cite{Matiyasevich Diophantine}. 
The proof in \cite{Lyndon Schupp}  uses this Diophantine characterization along with Valiev's approach in \cite{Valiev}.

An approach reflecting  Aanderaa's  proof \cite{Aanderaa} has been included by Rotman in Chapter~12 of \cite{Rotman}. It applies the auxiliary constructions interpreting Turing machines via semigroups with specific generators, and also uses the Boone-Britton group  \cite{Boone Certain simple unsolvable problems, Britton The word problem}. Group diagrams allow \cite{Rotman} to shorten the proof of \cite{Aanderaa}.

\smallskip 

Besides the above named two well known textbooks \cite{Lyndon Schupp, Rotman}, other important proofs can be found in the literature, such as:

The proof of Higman's theorem can be very much simplified using the S-machines, see \cite{BORS2002}. 
While the main construction developed by Birget, Ol'shanskii, Rips, and Sapir is difficult to build, its complexity stems from the paper's additional goals.
Dropping from \cite{BORS2002} everything used to estimate the isoperimetric function, it is possible to get a far briefer argumentation for existence of  Higman's embedding.



An even simpler proof can be obtained from the work of Ol'shanskii and  Sapir \cite{OS2001}, but it would again require dropping the consideration of the isoperimetric function from \cite{OS2001}, seel also the related work \cite{The conjugacy problem and Higman embeddings}.

Shoenfield's textbook \cite{Shoenfield} in Mathematical logic does not cover Group theory in its main body, but it contains a remarkable \textit{Appendix} touching decision problems in groups.  In particular, it presents an interesting modification of the proof for Higman's theorem. The notion of benign subgroups from \cite{Higman Subgroups of fP groups}  is generalized in \cite{Shoenfield} to benign isomorphisms, benign subsets and benign subgroups, and then the main passage from recursion to groups is done by means of Principal Lemma on page 330.

The survey article \cite{Adyan Durnev}  of Adyan and Durnev presents an outline of the proof for Higman's theorem following the lines of Aanderaa's paper  \cite{Aanderaa}, and it is related to \cite{Post Recursive unsolvability, 
Markov impossibility of certain algorithms,
Novikov on algorithmic undecidability, 
Boone Certain simple unsolvable problems, 
Britton The word problem}.

\begin{Remark}
The reason, why we for Algorithm~\ref{AL Algorithm for explicit embedding of a recursive group} stick with the  methods of Higman, is not only the fact that \cite{Higman Subgroups of fP groups} is the oldest and most celebrated approach, but also that, it is based on the \textit{classic definition} of recursion via  composition, primitive recursion and minimization, see \cite{Davis, Rogers, Boolos Burgess Jeffrey}.
The above cited other methods, on the contrary, are mostly deducing recursion to ``third party'' results, such as, 
the solution of Hilbert's Tenth Problem, 
the interpretation of Turing machines via semigroups, S-machines,
the Boone-Britton group, the group diagrams.
\end{Remark}

The following remark is also relevant for the choice of the methods we used for construction of the explicit embedding:

\begin{Remark}
While developing Algorithm~\ref{AL Algorithm for explicit embedding of a recursive group}, we naturally investigated whether an explicit analogue of Higman's embedding could be constructed using the \textit{other} methods listed above, including the S-machines. Beyond reviewing the literature, we discussed this question with leading specialists in these fields. However, we could not find any approaches that would yield a shorter, explicit embedding construction.
\end{Remark}

\subsection{Figures and illustrations}
\label{SU Figures and illustrations}

To better illustrate some of our constructions, we accompany them by fi\-gu\-res~\ref{FI Figure_01_Star_construction},
~\ref{FI Figure_09_C},
~\ref{FI Figure_10_A}
in Chapter~\ref{SE Some auxiliary constructions};
figures~\ref{Figure_11_Theorem_A}, 
\ref{Figure_12_Un_Intersect},
\ref{Figure_13_KP},
\ref{Figure_14_KQ},
\ref{Figure_15_Extracting_Q1_and_ArhoX}
in Chapter~\ref{SE Theorem A and its proof steps};
and 
figures~\ref{Figure_16_If_AX_is_benign_ZX_is_benign},
\ref{Figure_17_If_ZX_is_benign_QX_is_benign},
\ref{Figure_18_Higman_rope_trick}
in Chapter~\ref{SE Theorem B and the final embedding}. 
If some \textit{initial} benign subgroups are given in a particular group, we highlight them by gray color in a figure. Also, if some \textit{new} benign subgroups have to be constructed based on the initial ones, we highlight them by dashed lines, see figures~\ref{Figure_11_Theorem_A}\,--\,\ref{Figure_15_Extracting_Q1_and_ArhoX}
in Chapter~\ref{SE Theorem A and its proof steps}
and 
figures~\ref{Figure_16_If_AX_is_benign_ZX_is_benign}\,--\,\ref{Figure_18_Higman_rope_trick}
in Chapter~\ref{SE Theorem B and the final embedding} with these features.  This graphical visualization hopefully makes identification of benign subgroups inside the general constructions clearer.

Likewise, we recommend the reader to check figures 1\,--\,8 in 
\cite{Auxiliary free constructions for explicit embeddings} which illustrate the $\bigast$-construc\-tion and some other constructions we use. Only one of those eight figures from \cite{Auxiliary free constructions for explicit embeddings} has been copied here as Figure~\ref{FI Figure_01_Star_construction} in Chapter~\ref{SE Some auxiliary constructions}. 

\subsection*{Acknowledgments}
\label{SU Acknowledgments}
I am very much  grateful to Prof.~Alexander Yu.~Ol'shanskii for the opportunities to discuss this work, including both Higman's construction and some alternative methods, such as the S-machines, that allow to achieve the result of \cite{Higman Subgroups of fP groups}. 

Considering the related article~\cite{Higman's reversing operation}, I had the occasions to discuss our constructions in detail with Prof.~Oleg Bogopolski who, among other helpful points, also suggested to illustrate the most complex constructions with figures. That idea was used in the current work, as well, see illustrations listed in Section~\ref{SU Figures and illustrations} above.


It is a pleasure to me to thank the  State Committee of Science MES RA for the grant 25RG-1A187 which partially supports the current research.

I am also thankful to the German Academic Exchange Service DAAD Fachliteratur Program, grant A/97/13683, for the opportunity to acquire academic literature, including the volume \cite{De La Harpe 2000} referred above.

\bigskip 
\section{Glossary of terms, auxiliary constructions and references}
\label{SE Glossary of terms, auxiliary constructions and references}

\subsection{Necessary terms on computability: \textit{recursive, effective, computable, constructive} and \textit{explicit}}
\label{SU Necessary terms on computability} 

In what follows, we use certain terminology from mathematical logic and computability theory. Since these notions have historically gone by various names, we begin by fixing our terminology. 
We will take care to indicate how these terms could have been understood by Higman himself.

\subsubsection{Effective procedures and effectively enumerable sets}
\label{PO Effective procedures and effectively enumerable sets} 
By widely accepted consensus in computability theory, a procedure or function on the set of non-negative integers  $\mathbb N_0  \!=\! \{0,1,2,\ldots\}$ is called \textit{effective} if there theoretically exists an algorithm  -- formalized as a Turing machine -- to carry out that procedure. 

The latter does not have to be \textit{total}, in the sense that on some values of the argument the procedure may result in infinite loops.  
Even if that procedure is just \textit{partial}, any non-empty subset $\mathcal X$ of $\mathbb N_0$ can be the range of a partial effective function if and only if it is the range of a total effective function. 
Hence, we call such a subset an \textit{effectively enumerable} set, intuitively understanding it as a set that can somehow be listed element-by-element.    

Using a structural, bijective coding scheme, these notions can be generalized from $\mathbb N_0$ to any other countable set $\E$. 

Let us stress that the actual partial or total algorithm itself does \textit{not} have to be directly given, and the proof of its existence already is enough, see \cite{Davis, Rogers, Boolos Burgess Jeffrey} for background information.

\subsubsection{Recursive enumeration, recursive and computable sets}
\label{PO Recursive enumeration, recursive and computable sets} 

Higman uses the Kleene characterization of a \textit{recursively enumerable} subset $\mathcal X$ of the set $\mathbb N_0$, as a subset that can be obtained as the range of a partial recursive function on  $\mathbb N_0$, i.e., of a function that can be built from the zero, successor and projections functions by means of composition $\circ$, primitive recursion $\rho$ and minimization $\mu$ operations \cite{Davis, Rogers, Boolos Burgess Jeffrey}.  
A subset $\mathcal X$ is \textit{recursive}, if both $\mathcal X$ and its complement $\mathbb N_0 \backslash \mathcal X$ are recursively enumerable.

Since a bijective coding scheme can be used ``to code'' the elements of an arbitrary effectively enumerable set $\E$ by means of integers from $\mathbb N_0$ (Higman uses the G{\" o}del numbering for this purpose), the subsets $\mathcal X \subseteq \E$ can also be ``coded'' by subsets of $\mathbb N_0$, and hence, the  notion of recursion is generalized for such subsets of $\mathcal X$ also.

\smallskip
Under the Church-Turing Thesis, effective (effectively enumerable) is equivalent to recursive (recursively enumerable) \cite{Davis, Rogers, Boolos Burgess Jeffrey}. Hence the above introduced terms in points  \ref{PO Effective procedures and effectively enumerable sets} and \ref{PO Recursive enumeration, recursive and computable sets} can be used as synonyms.

Newer and more popular terms for these, especially in computer science related literature, are  \textit{computable} (\textit{comparably enumerable}); however, we adhere to Higman's traditional terminology here.

\subsubsection{Constructive and explicit procedures}
\label{SU Recursive enumeration and recursive groups?????????????} 

As we stressed at the very end of Point~\ref{PO Effective procedures and effectively enumerable sets}, 
effective (enumeration) or recursive (enumeration) on a subset $\mathcal X$ in $\N_0$ or in $\E$ do \textit{not} mean, that the algorithm carrying out the procedure to output $\mathcal X$ is directly given. 

By contrast, if the  algorithm is directly provided rather than merely proven to exist, then the procedure is called \textit{constructive} \cite{Troelstra van Dalen}, and the respective set is called (partial or total) constructively enumerable. 

\smallskip
A synonym we prefer for (total) constructive is \textit{explicit}, which is not only shorter, but also helps to avoid uncomfortable tautological phrasings such as ``constructive construction''.

\subsubsection{The terms assumed by Higman himself}
\label{PO The terms assumed by Higman} 

Higman  used the terms \textit{``recursive''} and `\textit{`recursively enumerable''} in \cite{Higman Subgroups of fP groups} very often. He did give the definitions in Kleenee's sense, and did give two references to the textbook of Davis, see pages 456 and 457 in \cite{Higman Subgroups of fP groups}.

Further, he used the terms \textit{``effectively enumerable''} or \textit{``effective enumeration''} to describe certain sets on pages 455, 456, 458, 462, 463, 466; and he described certain embeddings to be \textit{``effective''} on page
456 of \cite{Higman Subgroups of fP groups}.

Although Higman did \textit{not} define these three terms with ``effective'', we can be confident he relied on Davis's textbook, for, he deferred to \cite{Davis} for \textit{everything else} concerning logic and computability.
And this textbook is the only reference to a source in logic or computability mentioned in \cite{Higman Subgroups of fP groups}.

\smallskip
Chapter 1 and Chapter 3 of Davis \cite{Davis} show that the definitions of effective and recursive align precisely to what we outlined in previous points above, including their equivalence under Church-Turing Thesis.

To be even more accurate, notice that, Davis did \textit{not} give a special name to that hypothesis in \cite{Davis}, but he called it ``Church's Thesis'' in 
\cite{DavisSigalWeyuker} later.

\subsubsection{Explicitly given groups and their embeddings}
\label{SU Recursive enumeration and recursive groups ????} 

For understanding of the current work in most cases it will be enough to informally think of an \textit{explicitly} given group $G = \langle\, X \mathrel{|} R \,\rangle$ as of a group for which its generators $X$ and defining relations $R$ are somehow listed by some algorithm.  This has especially simple meaning if $X$ and $R$ are just finite, written down directly.

The same concerns a \textit{recursive} group $G = \langle\, X \mathrel{|} R \,\rangle$ given by some effectively enumerable $X$ and recursively enumerable $R$. 

This requirement on the set $R$ could well be replaced by the condition that $R$ is a recursive, see page 88 and page 151 (about the ``Craig trick'') in \cite{Lyndon Schupp}, also see \cite{hinman1978} for more details about that trick with padding words. Moreover, the words in $R$ can even be chosen to be \textit{positive} only, see page 451 in \cite{Rotman}. 

\smallskip 
Furthermore, an \textit{explicit embedding} $\varphi : G \to H$ is defined as an embedding under which the images of all generators of $G$ under $\varphi$ are explicitly expressed as words over the generating set of $H$.

\subsection{Integer-valued functions $f$}
\label{SU Integer functions f} 

Following \cite{Higman Subgroups of fP groups} denote by $\mathcal E$ the set of all functions $f : \Z \to \Z$ with finite supports, such as, the function $f$ sending the integers $-1, 0, 1, 2$ to $3,2,9,8$,\; and all other integers to $0$. 
When $f$ for a certain fixed $m=1,2,\ldots$ has the property that 
$f(i)=0$ for \textit{all} $i<0$ and $i\ge m$, then it is comfortable to record the function $f$ as a sequence $f=(j_0,\ldots,j_{m-1})$ assuming $f(i)=j_i$ for $i=0,\ldots,m-1$, e.g., the function $f=(2, 5, 3)$ sending the integers $0, 1, 2$ to $2, 5, 3$,\; and all other integers to $0$.
Clearly, $m$ may not be uniquely defined for $f$\!, and where necessary we may add extra zeros at the end of a sequence, e.g., the previous function can well be recorded as $f=(2, 5, 3,0,0)$ (the last two zeros change nothing in the way $f$ acts on $\Z$). In particular, the constant zero function can be written, say, as $f=(0)$ or as $f=(0,\ldots,0)$ where needed. 
See more in 2.2 of \cite{The Higman operations and  embeddings}.

Denote $\E_m\!=\big\{
(j_0,\ldots,j_{m-1})
\;\mathrel{|}\;
j_i\!\in \!\Z,\; i\!=\!0,\ldots,m-1\big\}$ for an $m=1,2,\ldots$ Clearly, $\E_m \subseteq \E$, and by the agreement above $\E_{m'}\subseteq \E_m$,
in case $m' \le m$.
Also denote $\Zz\!=\!\big\{(0)\big\}$ to consist of the constant \textit{zero} function only, and denote 
$\S\!=\!\big\{(n,\, n\!+\!1) \mathrel{|} n\in \Z\big\}$ to consist of all \textit{consecutive} integer pairs from $\E_2$ only.

For any $f\in \mathcal E$ and $k\in \Z$  define the function $f_{k}^+$\! as follows: 
$f_{k}^+(i)=f(i)$ for all $i\!\neq\! k$, and   
$f_{k}^+(k) = f(k)\!+\!1$.
When $f\in \E_m$, we shorten $f_{m-1}^+=f^+$\!. Say, for the above $f\!=\!(2, 5, 3)\in\E_3$ we have $f_{1}^+\!\!=(2, 6, 3)$ and 
$f^+\!\!=(2, 5, 4)$, we just add $1$ to the last  coordinate of $f$ to get $f^+$\!.

\subsection{The Higman operations}
\label{SU The Higman operations} 
Higman defines the following specific basic operations that  transform the subsets of $\E$ to some new subsets of  $\E$, see Section 2 in  \cite{Higman Subgroups of fP groups}:
\begin{equation}
\tag{H}
\label{EQ Higman operations}
\iota,\; 
\upsilon,\; 
\rho,\; 
\sigma,\; 
\tau,\; 
\theta,\; 
\zeta,\; 
\pi,\; 
\omega_m,
\end{equation}
$m=1,2,\ldots$
Although we are going to use the Higman operations  extensively, we do not include their definitions and main properties here, referring the reader to  Section 2 in  \cite{Higman Subgroups of fP groups} or to Section 2.3 in  \cite{The Higman operations and  embeddings}.

Also, in Section 2.4 in \cite{The Higman operations and  embeddings} we suggested some extra auxiliary Higman operations which hopefully make the work with the subsets of $\E$ more natural and intuitive:
\begin{equation}
\tag{H*}
\label{EQ auxiliary Higman operations}
\sigma^i, \;
\zeta_i,\;
\zeta_S,\;
\pi',\;
\pi_i,\;
\pi'_i,\;
\tau_{k,l},\;
\alpha,\;
\epsilon_S,\;
+,\;
\iota_n,\;
\upsilon_n.
\end{equation}
If a subset $\mathcal X$ of $\E$ can be obtained by operations \eqref{EQ Higman operations} and \eqref{EQ auxiliary Higman operations}, then it can be obtained by operations \eqref{EQ Higman operations} alone.

\subsection{Defining subgroups by integer sequences}
\label{SU Defining subgroups by integer sequences}

Let $F_3=\langle a, b, c\rangle$ be a free group of rank $3$. For any $i\in \Z$ denote $b_i=b^{c^i} = c^{-i} b \, c^i$ in $\langle b, c\rangle$. Then for any function $f\! \in \E$ one can define the following products in $\langle b, c\rangle$ and in $F_3$: 
\begin{equation}
\label{EQ defining b_f and a_b}
b_{f} = \cdots
b_{-1}^{f(-1)}
b_{0}^{f(0)} 
b_{1}^{f(1)}\cdots 
\quad \text{and} \quad
a_{f} = a^{b_{f}} = b^{-1}_{f} a\, b_{f}.
\end{equation}
Say, for the function 
$f$ sending\,  $-1, 0, 1, 2$ to $3,2,9,8$ in Section~\ref{SU Integer functions f} we have 
$b_f\!=\!b_{-1}^{3}
b_{0}^{2} 
b_{1}^{9}
b_{2}^{8}$.

When $f$ is in $\E_m$, we can more comfortably record it as 
$f=(j_0,\ldots,j_{m-1})$, see Section~\ref{SU Integer functions f}, and then write
$b_{f} = b_0^{j_0}  \cdots b_{m-1}^{j_{m-1}}$.
Say, for $f\!=\!(2, 5, 3)$
we may put
$b_f\!=\!
b_{(2, 5, 3)}\!=
b_{0}^{2} 
b_{1}^{5}
b_{2}^{3}$ and
$a_{(2, 5, 3)}\!=a^{b_{0}^{2} 
b_{1}^{5}
b_{2}^{3}}$.

For a set $\mathcal X$ of sequences from $\E$ denote by $A_{\mathcal X}$ the subgroup generated in $F$ by all the conjugates $a_{f} = a^{b_{f}}$ with $f\!\in \mathcal X$.
Say, $A_{\E_2}$ is the subgroup generated in $F_3$ by all words of type
$a^{b_{(j_0,\; j_1)}}
=  
b_{1}^{-j_1}
b_{0}^{-j_0}
\cdot a \cdot b_{0}^{j_0} 
b_{1}^{j_1}$, with all $j_0, j_1 \in \Z$.

For technical purposes we may use the above notation with some \textit{other} free generators, also. Say, in the free group $\langle d,e\rangle$
we may set  
$d_i=d^{e^i}$\!,\;\;
$d_{f} = \cdots
d_{-1}^{f(-1)}
d_{0}^{f(0)} 
d_{1}^{f(1)}\cdots $
Or in the free group $\langle g,h,k\rangle$
we may set  
$h_i=h^{k^i}$\!,\;\;
$h_{f} = \cdots
h_{-1}^{f(-1)}
h_{0}^{f(0)} 
h_{1}^{f(1)}\cdots $
and 
$g_{f} = g^{h_{f}}$\!.
Or else, we may take the isomorphic copy 
$\bar F=\langle \bar a, \bar b, \bar c\rangle$ of $F=\langle a, b, c\rangle$, and use inside it the elements $\bar b_i, \bar b_f, \bar a_f$ and the subgroup $\bar A_{\mathcal X}$ defined as expected.

\subsection{Benign subgroups}
\label{SU Benign subgroups and Higman operations} 

A subgroup $H$ in a finitely generated group $G$ is called a \textit{benign subgroup} in $G$, if $G$ can be embedded into a finitely presented group $K$ with a finitely generated subgroup $L\le K$ such that $G \cap L = H$.
For detailed information on benign subgroups we refer to  Sections 3, 4 in \cite{Higman Subgroups of fP groups}, see also sections 2, 4, 5 in \cite{Auxiliary free constructions for explicit embeddings}.

\begin{Remark}
\label{RE finite generated is benign}
From the definition of benign subgroup it is very easy to see that arbitrary \textit{finitely generated} subgroup $H$ in any \textit{finitely presented} group $G$ is benign in $G$.  Because $G$ itself can be chosen as the finitely presented overgroup $K_H$ of $G$ with a finitely generated subgroup $H = L_H$, such that 
$H\cap L_H = 
H \cap \,H = H$ inside $K_H=G$. We are going to often use this remark in the sequel.
\end{Remark}

\subsection{Free constructions}
\label{SU Free constructions} 

For background information on free products with amalgamations and on HNN-extensions we refer to 
\cite{Bogopolski} and \cite{Lyndon Schupp}. 
Our usage of the \textit{normal forms} in
free constructions is closer to~\cite{Bogopolski}.
Notation varies  in the literature, and to maintain uniformity we are going to adopt notation we used in \cite{Auxiliary free constructions for explicit embeddings}. 

Namely, if any groups $G$ and $H$ have subgroups, respectively, $A$ and $B$ isomorphic under  $\varphi : A \to B$, then the (generalized) free product of $G$ and $H$ with amalgamated subgroups $A$ and $B$ is denoted by
$G*_{\varphi} H$ (we are \textit{not} going to use the alternative notation $G*_{A=B} H$). When $G$ and $H$ are overgroups of the same subgroup $A$, and $\varphi$ is just the identical isomorphism on $A$, we write the above as $G*_{A} H$.

If $G$ has subgroups $A$ and $B$ isomorphic under  $\varphi : A \to B$, then the HNN-extension  of the base $G$ 
by some stable letter $t$
with respect to the isomorphism 
$\varphi$ is denoted by
$G*_{\varphi} t$.
In case when $A=B$ and $\varphi$ is  identity map on $A$, we denote $G*_{A} t$.
We also use HNN-extensions $G *_{\varphi_1, \varphi_2, \ldots} (t_1, t_2, \ldots)$  
with more than one stable letters, see \cite{Auxiliary free constructions for explicit embeddings} for details.

\medskip
Below we are going to use a series of facts about certain specific subgroups in  free constructions $G*_{\varphi} H$, \, $G*_{A} H$, \,$G*_{\varphi} t$ and $G*_{A} t$. We have collected them in Section 3 of  \cite{Auxiliary free constructions for explicit embeddings} to refer to that section whenever needed. 

\medskip
Notice that some of the constructions in this article could be replaced by shorter analogs using the wreath product methods we suggested in
\cite{Subnormal embedding theorems}\,--\,\cite{Subvariety structures}. However, here we intentionally use free constructions only to stay closer to the methods of Higman.

\subsection{The ``conjugates collecting'' process}
\label{SU The conjugates collecting process}

We are going to use the following simple, well known combinatorial trick. 
If $\X$ and $\Y$ 
are any disjoint subsets in a group $G$, then it is easy to verify that any element $w\in \langle \X,\Y \rangle$ can be written as:
\begin{equation}
\label{EQ elements <X,Y>}
w=u\cdot v
= x_1^{\pm v_1}
x_2^{\pm v_2}
\cdots
x_{k}^{\pm v_k}
\cdot
v
\end{equation}
where $v_1,v_2,\ldots,v_k,\, v\in\langle \Y \rangle$, and 
$x_1,x_2,\ldots,x_k \in \X$.
Indeed, first write $w$ as a product of some elements from $\X$, from $\Y$ and of their inverses. 
Next, by grouping
the elements from $\Y$ where necessary, and by adding some auxiliary trivial elements we rewrite $w$ as:
\begin{equation}
\label{EQ preliminary step}
w=z_1^{\vphantom8} x_1^{\pm1} 
z_2^{\vphantom8} \, x_2^{\pm1} z_3^{\vphantom8}
\cdots
z_k^{\vphantom8}\, x_k^{\pm1} z_{k+1}^{\vphantom8}
\end{equation}
where $x_1,\ldots,x_k \in \X$ and 
$z_1,\ldots,z_{k+1} \in \langle \Y \rangle$.
Say, if $\X =\{x_1, x_2, x_3\}$ and  
$\Y =\{y_1, y_2\}$, then 
$w=x_2^{-1} y_1^3  y_2^{\vphantom8} x_1^2 x_3^{\vphantom8}$ can be rewritten as 
$w=z_1\, x_2^{-1} z_2^{\vphantom8} x_1 z_3^{\vphantom8} x_1
z_4^{\vphantom8}
x_3^{\vphantom8}
z_5^{\vphantom8}
$,  
where
$z_1=1$,\,
$z_2=y_1^3  y_2^{\vphantom8}$,\,
$z_3=z_4=z_5=1$ are in $\langle \Y \rangle$. Then \eqref{EQ preliminary step} can be transformed to:
$$w=z_1^{\vphantom8} x_1^{\pm1}
z_1^{-1} \!\cdot\,
z_1^{\vphantom8}
z_2^{\vphantom8} \, x_2^{\pm1}
(z_1^{\vphantom8}
z_2^{\vphantom8})^{-1}
\!\cdot\,
z_1^{\vphantom8}
z_2^{\vphantom8}
z_3^{\vphantom8}
\cdots
(z_1^{\vphantom8}\!\cdots z_k^{\vphantom8})\, x_k^{\pm1}
(z_1^{\vphantom8}\!\cdots z_k^{\vphantom8})^{-1}
z_1^{\vphantom8}\!\cdots z_k^{\vphantom8} z_{k+1}^{\vphantom8},
$$
which is \eqref{EQ elements <X,Y>} for
$v_1^{\vphantom8}\!=z_1^{-1}$\!\!\!, \;
$v_2^{\vphantom8}\!=(z_1z_2)^{-1}$\!\!\!\!, \; 
$v_3^{\vphantom8}\!=(z_1 z_2 z_3)^{-1}$\!\!\!\!, \, $\ldots$
$v_k^{\vphantom8}\!=(z_1 \cdots z_k)^{-1}$\!\!\!\!, \;  
$v\!=z_1^{\vphantom8}\!\cdots\; z_k^{\vphantom8} z_{k+1}^{\vphantom8}$.

In a simplest case $\X=\{x\}$ and $\Y=\{y\}$ are of carnality $1$, and in the $2$-generator group $G=\langle x, y\rangle$ we can write any $w \in G$ as a product
of some conjugates of $x$ with a power of $y$:
\begin{equation}
\label{EQ elements from <x,y>}
w=x^{\pm y^{n_1}}\!x^{\pm y^{n_2}}\!\!\cdots\, x^{\pm y^{ n_s}}\!\!\cdot y^k=\!u\cdot v\,.
\end{equation}
As we will see later, this argument is helpful in HNN-extensions and other free constructions.

\bigskip 
\section{Higman's original construction}
\label{SE Higman's original construction}

\noindent 
While working on explicit embeddings of recursive groups, we were most often asked questions of the following kind: \textit{``Couldn’t Higman’s method already be applied to achieve this?'', ``Isn't Higman's construction already explicit?''}, among others.


It therefore seems appropriate to precede our treatment of Algorithm~\ref{AL Algorithm for explicit embedding of a recursive group} with a detailed analysis of Higman’s embedding construction \cite{Higman Subgroups of fP groups}, identifying those of its steps for which no explicit algorithm exists.

This analysis will be carried out in sections~4\,--5 of this chapter. First, however, we would like to review several earlier mentions that address the same question from different perspectives.

\subsection{Earlier mentions in the literature}
\label{SU Earlier mentions in the literature}  
Suspicion, that not all steps of Higman's original construction \cite{Higman Subgroups of fP groups} may be explicit, arise already from the fact that for some well-known recursive groups the problem of their explicit embedding  into finitely presented groups has been (or still is) open for decades.
Also, neither Higman himself nor his descendants had ever stated the explicitness of the embedding of \cite{Higman Subgroups of fP groups} in the literature. 

\subsubsection{An observation by Johnson}
\label{SU An observation by Johnson} 
It is reasonable to refer again to Johnson’s work \cite{Johnson on Higman's interest}, already mentioned in Section~\ref{SU Application of the algorithm for Q} in connection with explicit embeddings of $\Q$.
His article is specifically dedicated to the discussion of \cite{Higman Subgroups of fP groups} and to the construction of \textit{explicit} embeddings of several recursive groups into finitely presented ones. On page 416 Johnson thanks Higman for raising the problem and mentions that \textit{``the main aim''} of the embedding of $\Q$ into a finitely presented group \textit{``continues to elude us''}.

Yet another example of this type is the matrix group $GL(n, \Q)$ mentioned in Section~\ref{SU The problem of explicit  embedding for GL(n Q)}.

\subsubsection{An observation by Valiev}
\label{SU An observation by Valiev}  

In \cite{Valiev example} Valiev remarks:
\textit{``G. Higman's proof is constructive in the sense that from  constructive description of the group $G$ one can, in principle, extract a description for [the finitely presented group] $K_G$\,, but this construction is so vast that it is practically impossible to do this''.}

Under $K_G$ Valiev means the finitely presented group into which the recursive groups is embedded $G$ (what we denoted by $\mathcal G$ above).

We are going to return to this reference in Section~\ref{SU Concerning the remarks of Valiev, Bridson and Nyberg-Brodda}, after we finish the detailed analyzes of the steps of \cite{Higman Subgroups of fP groups}.
%

\subsubsection{An observation by Bridson and Nyberg-Brodda}
\label{SU An observation by Bridson and Nyberg-Brodda}  

In the recent preprint \cite{Bridson Nyberg-Brodda 2025}
Bridson and Nyberg-Brodda mention: 
\textit{``Higman’s proof of the Embedding Theorem is effective'' but ``it leaves open the challenge of finding ``natural'' embeddings of particular recursively presented groups''.}
We will return to this reference in Section~\ref{SU Concerning the remarks of Valiev, Bridson and Nyberg-Brodda}.
 
We cited \cite{Bridson Nyberg-Brodda 2025} earlier in Section~\ref{SU Application of the algorithm for Q}, where we noticed how the remark on
\textit{``...personal favorite of Higman''} matches Johnson's older remark on page 416 of \cite{Johnson on Higman's interest}.

\bigskip 
In addition to such indirect arguments of rather ``historical'' nature, it is worthwhile to analyze the steps of Higman’s embedding \cite{Higman Subgroups of fP groups} to directly indicate its parts which suggest \textit{no explicit mechanism} to accomplish them, i.e., to state what makes some parts of \cite{Higman Subgroups of fP groups} non-explicit.

\medskip 
We outline Higman's construction in Section~\ref{SU The main steps of construction in H}, and then  analyze in Section~\ref{SU Is Higman's construction explicit} which parts are \textit{not} explicit in it, see points~\ref{SU Construction of X from R}--\ref{SU Obtaining KX and LX for the benign subgroup AX}, examples therein, and \nameref{SU Conclusion} in \ref{SU Conclusion}.


\subsection{The main steps of the construction in \cite{Higman Subgroups of fP groups}}
\label{SU The main steps of construction in H} 

Let us start by a summary of the main steps of Higman's construction in sections 2--5 of \cite{Higman Subgroups of fP groups} by which he embeds a recursive group $G$ into a finitely presented group. 
We outline the construction in below five steps because it is presented that way at the end of ``Introduction'' in \cite{Higman Subgroups of fP groups}.
These five steps will be used by us for two purposes below: 
\begin{itemize}
\item[$\circ$]
in Section~\ref{SU Is Higman's construction explicit} we use them to state which are the \textit{non-explicit} parts of \cite{Higman Subgroups of fP groups};

\item[$\circ$]
and in Chapter~\ref{SE The modified construction of the current work} we use them to  explain \textit{what we modify} to \cite{Higman Subgroups of fP groups} to overcome the found obstacles, and to have an explicit embedding.
\end{itemize}

\subsubsection{\textbf{Step 1.} 
Characterization of the  recursively enumerable subsets of $\E$ via Higman operations: page 457 and Section 2 in \cite{Higman Subgroups of fP groups}}
\label{SU Step 1. Characterization of recursively enumerable subsets via Higman operations}

Higman begins by the classic Kleene characterization of a recursively enumerable subsets of $\mathbb N_0$ which we mentioned in Point~\ref{PO Recursive enumeration, recursive and computable sets} with references. 
Then the G{\" o}del numbering allows to generalize the notion of recursively enumerable sub\-sets of $\mathbb N_0$ to \textit{any} subsets of an \textit{effectively enumerable} set $\E$, see Point~\ref{PO Effective procedures and effectively enumerable sets}, in particular, of the set $\E$ of all integer-valued functions $f\! : \Z \to \Z$ with finite supports from the above Section~\ref{SU Integer functions f}, as this $\E$ certainly is  effectively enumerable. 

The first main technical result of \cite{Higman Subgroups of fP groups} is Theorem 3 that suggests an \textit{alternative characterization} for recursively enumerable subsets $\mathcal X$ of the above $\E$, without using the Kleene method and the G{\" o}del numbering, at all. Namely, they turn out to be the subsets constructed from two specific basic sets 
$\Zz$ and
$\S$, see Section~\ref{SU Integer functions f} above, by means of the Higman ope\-ra\-tions 
$\iota,\; 
\upsilon,\; 
\rho,\; 
\sigma,\; 
\tau,\; 
\theta,\; 
\zeta,\;
\pi,\; 
\omega_m$, 
see \eqref{EQ Higman operations} in Section~\ref{SU The Higman operations}.   
The set of all such subsets of $\E$ is denoted via $\mathscr{S}$\!. 
By Theorem 3 in \cite{Higman Subgroups of fP groups}, \textit{a subset $\mathcal X$ of $\E$ is recursively enumerable if and only if $\mathcal X$ is in $\mathscr{S}$\!.}

\subsubsection{
\textbf{Step 2.} 
Characterization of the recursively enumerable subsets of $\E$ via benign subgroups in the free group of rank $3$: Section 3 and Section 4 in \cite{Higman Subgroups of fP groups}}
\label{SU Step 2. Characterization of the recursively enumerable subsets via benign subgroups}

In the above Step 1 Higman uses practically no group-theoretical constructions yet,\, they actually occur in the current step.
In the free group $F_3\!=\!\langle a,b,c\rangle$, using the conjugates $b_i=b^{c^i}$\!,\, the specific elements $b_{f}$ and $a_f$ are being defined for any function $f\! \in \E$, see \eqref{EQ defining b_f and a_b} in
Section~\ref{SU Defining subgroups by integer sequences} above.
This allows to define the respective subgroup $A_{\!\mathcal X} =\langle
a_f \;|\; f\in \mathcal X
\rangle$ inside $F_3$ for any subset $\mathcal X$ of $\E$.

Theorem~4 in \cite{Higman Subgroups of fP groups} states that \textit{$\mathcal X$ is recursively enumerable in $\E$ if and only if the respective subgroup $A_{\!\mathcal X}$ is \textit{benign} in $F_3$}, see definition and references for benign subgroups in Section~\ref{SU Benign subgroups and Higman operations} above.
Thus, discussion of recursive enumeration for subsets of $\E$ is being translated to the language of benign subgroups in $F_3$.

\medskip
These Step~1 and Step~2 occupy the main part of \cite{Higman Subgroups of fP groups} (sections 2--4), while the remaining three steps fit into a couple of pages in Section 5 \textit{``Conclusion''} of Higman's work.

\subsubsection{
\textbf{Step 3.} 
Characterization of the recursively enumerable subsets in free groups of rank $2$: page 473, Section 5 in \cite{Higman Subgroups of fP groups}}
\label{SU Step 3. Characterization of benign subsets in free groups of rank 2}
Lemma 5.1 in \cite{Higman Subgroups of fP groups} establishes connection between \textit{all} recursively enumerable subsets in the free group of rank $2$ and benign subgroups of a certain \textit{specific type} in the free group of rank $3$.

\subsubsection{
\textbf{Step 4.} 
Characterization of the recursively enumerable subgroups in free groups of any finite rank: page 474, Section 5 in \cite{Higman Subgroups of fP groups}}
\label{SU Step 4. Characterization of benign subgroups in free groups of finite rank}
As an adaptation of the previous step, Lemma 5.2 in \cite{Higman Subgroups of fP groups} shows that in \textit{any finitely generated free group, a subgroup is recursively enumerable if and only if it is benign}.
It is enough to argument this claim for the free group of rank $2$, as the general cases can be deduced to this.

\subsubsection{
\textbf{Step 5.} 
The final embedding by ``Higman Rope Trick'': pages 474\,--\,475, Section 5 in \cite{Higman Subgroups of fP groups}}
\label{SU Step 5. Final embedding by Higman Rope Trick}
Let the given finitely generated recursive group $G$ have a presentation $G = \langle\, X \mathrel{|} R \,\rangle$ with a finite $X$ and a recursively enumerable $R$, that is, $G \cong F/\langle R \rangle^F$ for the free group $F$ of rank $|\,X|$. As $R$ is recursively enu\-merable, its normal closure $\langle R \rangle^F$\! is also recursively enu\-merable for very simple combinatorial reasons and, it is \textit{benign} in $F$ by the previous step. 

Thus, by Lemma~3.5 in \cite{Higman Subgroups of fP groups} providing two alternative definitions for benign subgroups, the free product $H$ of two copies of $F$ amalgamated in $\langle R \rangle^F$ can be embedded into a certain finitely presented group $K$. 

Finally, a specific HNN-extension of the direct product $K\! \times \! G$ is being constructed to be the finitely presented overgroup of $G$ we are looking for, see page 475 in \cite{Higman Subgroups of fP groups}. A process often called the ``Higman Rope Trick'' then shows that all but \textit{finitely} many  of the relations of this HNN-extension are redundant, also see the related discussion \cite{Higman rope trick}.

\subsection{Is Higman's embedding explicit?}
\label{SU Is Higman's construction explicit} 

Before starting this section, the reader may wish to review Section~\ref{SU Necessary terms on computability}, which recalls the terms \textit{recursive, effective, computable, constructive,} and \textit{explicit}.

\medskip
In order to examine the explicitness of Higman’s embedding, suppose that a recursive group is explicitly given as $G = \langle\, X \mathrel{|} R \,\rangle$, and follow the steps of \cite{Higman Subgroups of fP groups} \textit{literally} to see which obstacles arise. In the points and examples below, we illustrate these obstacles, in particular for the group $\Q$.

\begin{Remark}
\label{RE we are not using the below steps}
Notice that in our modified constructions for the proofs of
Theorem~\ref{TH Theorem A}
and
Theorem~\ref{TH Theorem B}
in chapters~\ref{SE The modified construction of the current work}\,--\ref{SE Theorem B and the final embedding}
below, we are \textit{not} going to use the steps causing obstacles. We mention those steps just to study explicitness of \cite{Higman Subgroups of fP groups}.
\end{Remark}

\subsubsection{Construction of the sequences set $\mathcal X$ from $R$}
\label{SU Construction of X from R}   

To build the finitely presented group $K$, used in Step 5 for the given recursive group $G = \langle\, X \mathrel{|} R \,\rangle \cong F/\langle R \rangle^F$, see Point~\ref{SU Step 5. Final embedding by Higman Rope Trick} above, Higman starts from the recursively enumerable relations set $R$ in the free group $F$ of rank $|\,X|$. Since the  enumeration of $R$ is very easy to continue on its normal closure $\langle R \rangle^F$\!, the latter is also recursively enumerable. \cite{Higman Subgroups of fP groups} denotes that closure by $R$, but here we prefer $\langle R \rangle^F$\!, as we wish to distinguish it from the set $R$ clearly.

Then ``moving backwards'' via Higman's
Step~4\,--\,Step~2 for $\langle R \rangle^F$\!, see points~\ref{SU Step 4. Characterization of benign subgroups in free groups of finite rank}\,--\,\ref{SU Step 2. Characterization of the recursively enumerable subsets via benign subgroups} above, one should arrive to the respective  subset $\mathcal X$ inside $\E$, and then to the benign subgroup $A_{\!\mathcal X}$ inside the free group $F_3=\langle a,b,c\rangle$ of rank $3$.

To do this Higman first embeds the free group $F$ into a free group $F_2=\langle x,y \rangle$ of rank $2$, see page 474 in \cite{Higman Subgroups of fP groups}. That embedding can be built using any of the well known textbook methods \cite{Robinson, Rotman, Bogopolski, Kargapolov Merzljakov}. 
By the remark preceding Lemma~3.8 in \cite{Higman Subgroups of fP groups}, the normal subgroup $\langle R \rangle^F$ is benign in $F$ if and only if its image under this embedding is benign in $F_2$. 

After this embedding, we can imagine all the relations from $\langle R \rangle^F$\! are rewritten in just \textit{two} letters $x,y$.
The purpose of this passage is that, it allows to directly output the integer sequences $f\in \mathcal X$ from those words on $x,y$.


\begin{Example}
\label{EX Going from Q to X is harsh PART 1}
A presentation of $\Q  = \langle\, X \mathrel{|} R \,\rangle$ by generators and defining relations is given by \eqref{EQ Q genetic code} in Example~\ref{EX first presentation for Q SECOND}. 
The free group $F= \langle a_1, a_2,\ldots \rangle$ needed for $\Q$, is of countable rank.

Choose, say, two relations
$w_3=a_3^3\,a_{2}^{-1}$ 
and 
$w_2=a_2^2\,a_{1}^{-1}$
from Example~\ref{EX first presentation for Q SECOND}. 
Conjugating these relations by the random words, say, $u = a_1 a_2^2$ and $v = a_3 a_2^5$ from $F$, we in $\langle R \rangle^F$\! may get a word of type:
\begin{equation*}
\begin{split}
w & = w_3^{\,u} \cdot w_2^{v}\cdot w_3^{u^3}\! \cdot w_2^{v^{-1}}\\
& =
a_2^{-2} a_1^{-1}  a_3^3\,a_{2}^{-1}  a_1  \,
a_2^{-3}  a_3^{-1} a_2^2 \,a_{1}^{-1} a_3  \, a_2^3\, a_1^{-1} a_2^{-2} a_1^{-1} a_2^{-2}              \\
& \hskip20mm 
\cdot
a_1^{-1} a_3^3\,   a_{2}^{-1} a_1\,  a_2^2    a_1\, a_2^2\, a_1\, a_2^2\,    a_3\,  a_2^{7} \,  a_{1}^{-1}\! a_2^{-5}  a_3^{-1} 
\end{split}
\end{equation*}
(some cancellations have been done, where needed).

Next embed $F=\langle a_1, a_2, \ldots \rangle$ into $F_2=\langle x,y \rangle$  by the rule:
$$
a_1 \to x, \quad 
a_2 \to x^y, \quad 
a_3 \to x^{y^2}\!\!,\;\; \ldots  
$$
(Higman does not suggest to use \textit{this} embedding necessarily, see page 474 in \cite{Higman Subgroups of fP groups}, but we apply it, as it is a popular textbook trick).
After this embedding the above obtained word $w$ (in three letters $a_1, a_2, a_3$) from $F$, maps to the following word (in two letters $x,y$) 
from $F_2$:
\begin{equation*}
\begin{split}
w & \to  (x^y)^{-2} x^{-1} (x^{y^2})^3 \,(x^y)^{-1} x \,(x^y)^{-3} (x^{y^2})^{-1}  (x^y)^2     \\
& \hskip14mm 
\cdot\;    x^{-1} x^{y^2} (x^y)^3 \, x^{-1}  (x^y)^{-2}x^{-1} (x^y)^{-2}     \\
& \hskip14mm 
\cdot\;  x^{-1}(x^{y^2})^3\,(x^y)^{-1}  x\,           (x^y)^2    x\,  (x^y)^2 x\,  (x^y)^2  \\
& \hskip14mm 
 \cdot\;  x^{y^2} (x^y)^7  x^{-1}  (x^y)^{-5}  (x^{y^2})^{-1}, 
\end{split}
\end{equation*}
which can be simplified to:  
\begin{equation*} 
\begin{split}
&  y^{\!-1} \! x^{\!-2} \! y \!\cdot\! x^{\!-1}\!   \!\cdot\! 
 y^{\!-2}\! x^3\! y^2\! \!\cdot\!
 y^{\!-1}\! x^{\!-1}\! y \!\cdot\! x\!\cdot\! y^{\!-1}\! x^{\!-3}\! y\!\cdot\!
y^{\!-2}\! x^{\!-1}\! y^2\!  \!\cdot\! y^{\!-1}\! x^2\! y   \!\cdot\! y^{\!-1}\! x^{\!-2}\! y        \\
& \hskip3mm  \cdot   x^{-1} \cdot y^{-2} x  y^2 \cdot  y^{-1} x^3 y      \cdot x^{-1} \cdot y^{-1}         x^{-2} y   \cdot x^{-1}        
\\
& \hskip3mm \cdot    x^{-1} \!\cdot\! y^{-2}  x^3  y^2 \!\cdot\! y^{-1} x^{-1} y \!\cdot\! x      \!\cdot\! y^{-1}  x^{2} y \!\cdot\!
x \!\cdot\! y^{-1}x^{2} y \!\cdot\! x   \!\cdot\! y^{-1} x^{2} y         \\
& \hskip3mm \cdot  y^{-2} x y^2 \cdot y^{-1} x^{7} y \cdot x^{-1}  \cdot y^{-1} x^{-5} y \cdot 
y^{-2} x^{-1} y^2.
\end{split}
\end{equation*}
From such a word Higman extracts its exponents to record an integer-valued sequence $f$ to be included in the set $\mathcal X$:
\begin{equation} \small 
\label{EQ Long f for Q}
\begin{split}
& {\normalsize f} \!=  \big(0,  \!-1, \!\!-2, \!1,  \!\!-1,\!\!-2,\!3,\!2, \!\!-1,\!\!-1, \!1, \!1,\!\!-1,\!\! -3, \!1, \!\!-2, \!\!-1, \!2, \!\!-1, \!2, \!1, \!\!-1, \!\!-2, \!1 , \\
& \hskip20mm  -1, -2, 1, 2 -1, 3,  1, -1, -1, -2 , 1, -1, \\ 
& \hskip20mm  -1, -2, 3,2, -1, -1, 1, 1, -1, 2, 1, 1,     -1, 2, 1, 1, -1, 2, 1, \\
& \hskip20mm  -2, 1, 2, -1, 7, 1, -1, -1, -5, 1, -2 -1, 2    \big),
\end{split}
\end{equation}
notice how it starts by $0$ as we have to append an $x^0$ before $w$. Such a sequence has to be loaded into $\mathcal X$ for \textit{every} word in $\langle R \rangle^F$\!. 

This step alone shows what a huge set $\mathcal X$ of ``astronomic'' size one could get following Higman's steps \textit{literally}: there is no bound for the lengths of sequences $f$ in $\mathcal X$, because instead of the above short word $w$ one could take a product of many more words $w_i$, conjugated by much longer words from $F$.  

There is certainly no doubt that the obtained set $\mathcal X$ is recursively enumerable, in the sense that there exists a partial recursive function that outputs $\mathcal X$  as its image (equivalently, there is a Turing machine to output it). But this set is practically impossible to use in the next steps to come. 

\medskip 
Check sections~\ref{SU The embedding alpha into a 2-generator group} and \ref{SU Using the 2-generator group to get the set X} to see how we suggest to overcome this issue. In particular, in Example~\ref{EX first presentation for Q SECOND} and Example~\ref{EX writing X from TQ for rational Q}, our modified embedding method uses for $\Q$ the by far slimmer set $\mathcal X = \big\{ 
f_k \mathrel{|} k=2,3,\ldots \big\}$ of sequences of type \eqref{EQ Higman code for Q} \textit{only}, all of them of length $19$. The sequences \eqref{EQ Higman code for Q} are all of certain simple ``pattern'', which is going to make the work with them manageable in the coming steps. 
\end{Example}

\subsubsection{Missing explicit construction of $\mathcal X$ by \eqref{EQ Higman operations}}
\label{SU Explicit construction of X via Higman operations}

Higman's next objective is the construction of the set $\mathcal X$, prepared in the previous point, via the operations
\eqref{EQ Higman operations} from two basic subsets $\Zz$ and $\S$ of $\E$, see sections~\ref{SU Integer functions f} and \ref{SU The Higman operations} above for definitions and notation.

In construction of our explicit embeddings, in particular for the group $\Q$ in \cite{On explicit embeddings of Q}, we never follow this fragment of \cite{Higman Subgroups of fP groups}, see Remark~\ref{RE we are not using the below steps}. We use the much shorter constructions from \cite{The Higman operations and  embeddings} where applicable, see Example~\ref{EX Going from Q to X is harsh PART 2}.

However, here we briefly outline the respective fragment from \cite{Higman Subgroups of fP groups}, in order to state that it is \textit{explicit}, and to show what makes it very \textit{uncomfortable} to use, nevertheless.

\medskip  
Lemma 2.8 in \cite{Higman Subgroups of fP groups} provides a very long but \textit{yet explicit} algorithm to construct certain partial recursive functions $f(n,r)$, $a(r)$, $b(r)$, where $r=0,1,2,\ldots$ These functions one-by-one ``record'' all the sequences $g\in\mathcal X$, and ``mark'' their start- and end points, in the sense that for any such $g$ there is some $r$ for which $g(n)=f(n,r)$ for all $n\in \Z$, while $f(n,r)=0$ for all the coordinates $n<a(r)$ and $n>b(r)$. 
Such functions 
$f_{\E}(n,r)$, 
$a_{\E}(r)$, 
$b_{\E}(r)$  
are explicitly built first for the case of the whole set $\mathcal X = \E$ of all sequences with finite supports, and then it is noticed that for any other generic recursively enumerable subset $\mathcal X \subseteq \E$ there exists a partial recursive function $h(s)$ such that the $r$'th function $f_{\E}(n,r)$ corresponds to a function in $\mathcal X$ if and only if $r=h(s)$ for some $s$. Existence of such $h(s)$ is clear, as $\mathcal X$ is recursively enumerable, and it consists of functions with finite supports only. 

Then for every multi-variable integer-valued function $f(x_1,\ldots,x_n)$ its \textit{graph} is being defined before Lemma 2.2, and $\mathscr F$ is  denoted to be the set of all functions whose graphs are in $\mathscr S$. Lemmas 2.1\,--\,2.7 show that $\mathscr F$ contains all the partial recursive functions. 
Hence, the graphs of the above functions $f(n,r)$, $a(r)$, $b(r)$ also are in $\mathscr F$, i.e., their graphs can be constructed via the operations \eqref{EQ Higman operations}.

Using the above steps, Higman finishes the proof of Theorem 3 on pages 463\,--\,464 by showing how the earlier picked sequence $g\in \mathcal X$ can be obtained  via \eqref{EQ Higman operations} using the graphs of these functions $f(n,r)$, $a(r)$, $b(r)$.  

\medskip 
Although each of the listed steps is doable, the \textit{only} method to construct $\mathcal X$ via \eqref{EQ Higman operations} suggested in \cite{Higman Subgroups of fP groups}, is to present the functions $f(n,r)$, $a(r)$, $b(r)$ by  Kleene's characterization, that is, to construct $f(n,r)$, $a(r)$, $b(r)$ from the basic functions 
$n(x)=0$,\, 
$s(x)=x\!+\!1$, \,
$u_m^i(x_1,\ldots,x_m)$\,
by means of compo\-sition $\circ$, primitive recursion $\rho$, and minimization $\mu$, and then to apply a series of operations \eqref{EQ Higman operations} for each of the steps of that characterization to construct their graphs via \eqref{EQ Higman operations}, then to load these into the explicit proof of Theorem 3.

These specific proofs are valid, but they require such a vast routine of steps that their application to non-trivial examples is not a manageable task. 


\medskip 
Check Section~\ref{SU Writing X via Higman operations} to see how we plan to overcome this issue, using the method offered in \cite{The Higman operations and  embeddings}. In brief, \cite{The Higman operations and  embeddings} suggests how one could build a subset $\mathcal X \subseteq \mathcal E$ by the operations \eqref{EQ Higman operations}, following certain generic ``patterns'' that $\mathcal X$ may have.

\subsubsection{Defining the subgroup $A_{\mathcal X}\le F_3$ corresponding to the set $\mathcal X$}
\label{SU Construction of A_X from X} 

Next \cite{Higman Subgroups of fP groups} uses the conjugates
$b_i=b^{c^i}$\!,\, $i\in \Z$, to build in $F_3 \!= \langle a, b, c \rangle$ the product $b_f$, and the conjugate $a_f = a^{b_f}$ for each of the above sequences $f \! \in \mathcal X$, see the notation \eqref{EQ defining b_f and a_b}  in Section~\ref{SU Defining subgroups by integer sequences}. 
All such conjugates $a_f$ generate inside $F_3$  the subgroup $A_{\mathcal X} =\langle a_f \mathrel{|} f \in \mathcal X \rangle$ corresponding to the set $\mathcal X$.   

According to Higman's original construction, there can be such words $b_f$ of \textit{arbitrarily} high length, see the remark after Example~\ref{EX Going from Q to X is harsh PART 1}, and this makes the subgroup $A_{\mathcal X}$ very large. 

Whereas in our modified method all such $b_f$, together with $a_f$ are of some specific type only, and this much simplifies our work with the respective subgroup $A_{\mathcal X}$, see Section~3.2 in \cite{On explicit embeddings of Q}.
A sample of such a much simpler set $\mathcal X$ with a much simpler subgroup $A_{\mathcal X}$ can be found in Example~\ref{EX Going from Q to X is harsh PART 3} in Section~\ref{SU Building K_X and L_x for the benign subgroup A_X in F} below.

\subsubsection{Construction of $K_{\mathcal X}$ and $L_{\mathcal X}$ for the benign subgroup $A_{\mathcal X}$ of $F_3$}
\label{SU Obtaining KX and LX for the benign subgroup AX}

By Theorem 4 in \cite{Higman Subgroups of fP groups}, if the relations set $R$ is recursive, then the above defined subgroup $A_{\mathcal X}$ is benign in $F_3\!=\langle a, b, c \rangle$. That is, there exist a finitely presented overgroup $K_{\mathcal X}$ of $F_3$ with a finitely generated subgroup $L_{\mathcal X} \le K_{\mathcal X}$, such that $F_3 \cap K_{\mathcal X} = A_{\mathcal X}$  holds, see definitions in Section~\ref{SU Benign subgroups and Higman operations}, notice Remark~\ref{RE finite generated is benign}.

To construct the needed $K_{\mathcal X}$ and $L_{\mathcal X}$, Higman launches a step-by-step recursive procedure, ``directed'' by those steps which were earlier applied for construction of $\mathcal X$ (from $\Zz$ and $\S$, by means of the operations \eqref{EQ Higman operations}), see Point~\ref{SU Explicit construction of X via Higman operations} above. 
%
Namely, \cite{Higman Subgroups of fP groups} shows that:
\begin{itemize}
\item[$\circ$]   
The initial subgroups $A_{\mathcal Z}$ and $A_{\mathcal S}$ are benign in $F_3$, that is, there exist finitely presented groups 
$K_{\mathcal Z}$ and $K_{\mathcal S}$, with subgroups 
$L_{\mathcal Z}$ and $L_{\mathcal S}$, respectively,
such that 
$F_3 \cap L_{\mathcal Z} = A_{\mathcal Z}$ and
$F_3 \cap L_{\mathcal S} = A_{\mathcal S}$
hold.
\item[$\circ$] 
If the subset 
$\mathcal Y$ by any of the operations \eqref{EQ Higman operations} is obtained from certain subsets $\mathcal X$, for which $A_{\mathcal X}$ is benign in $F_3$, then $A_{\mathcal Y}$ is also benign in $F_3$, that is, there exist 
$K_{\mathcal Y}$ and $L_{\mathcal Y}$, 
such that $F_3 \cap L_{\mathcal Y} = A_{\mathcal Y}$ holds.

\item[$\circ$] 
After the previous step, $\mathcal X$ is replaced by $\mathcal Y$, and using the next one of the  operations \eqref{EQ Higman operations} we applied, the next pair of $K_{\mathcal Y}$ and $L_{\mathcal Y}$ is being constructed. 

\item[$\circ$] 
When the construction of the desired set $\mathcal X$ by operations \eqref{EQ Higman operations} reaches its terminal step, then the \textit{last} outputted $K_{\mathcal X}$ and $L_{\mathcal X}$ are the groups we are looking for.
\end{itemize}

This process is done in the proof for Higman's Theorem~4, see Lemmas 4.4\,--\,4.10 in \cite{Higman Subgroups of fP groups}. 

\medskip 
Some of these steps are perfectly explicit, such as, the construction of  $K_{\mathcal Z}$, $K_{\mathcal S}$, $L_{\mathcal Z}$, $L_{\mathcal S}$
for two subgroups $A_{\mathcal Z}$, $A_{\mathcal S}$ of $\langle a, b, c \rangle$, see the proof of Lemma 4.4 on page 470 in 
\cite{Higman Subgroups of fP groups}.

\medskip 
However, there are two problems which do \textit{not} allow us to \textit{directly} use these proofs from \cite{Higman Subgroups of fP groups}  for our \textit{explicit} embedding in Algorithm~\ref{AL Algorithm for explicit embedding of a recursive group}:

\medskip 
\textit{First problem.} Higman notices on page 468 in  \cite{Higman Subgroups of fP groups} that, if some group $A$ is a subgroup of two finitely generated groups $G$ and $H$, which are embeddable into some finitely presented group, then $A$ is benign in both $G$ and $H$, or in neither. For the theoretical purposes of \cite{Higman Subgroups of fP groups} this makes it very convenient to sometimes ``switch'' the group in which we show $A_{\mathcal X}$ is benign. This much shortens the proofs of lemmas 4.4\,--\,4.10 in \cite{Higman Subgroups of fP groups}, but this is \textit{not} appropriate for our practical purposes, because after each Higman operation \eqref{EQ Higman operations} we wish to always keep the current $A_{\!\mathcal X}$ benign in the \textit{same} group $F_3$, necessarily, so that we are in position to recursively apply the next operation as many times as needed, to get the terminal group $A_{\!\mathcal X}$ again benign in $F_3$, and not in some \textit{other} group. 

\medskip 
\textit{Second problem.} \cite{Higman Subgroups of fP groups} often uses Lemma 3.7, stating that if $H$ is the image of $G$ under a homomorphism $\phi$, then under certain conditions, $A$ is benign in $G$ if and only if $\phi(A)$ is be\-nign in $H$. Again, this shortens the proofs in \cite{Higman Subgroups of fP groups} very much, as it sometimes is enough to show that the given $A_{\!\mathcal X}$ is an image (or a pre-image) of a known benign subgroup. But this is \textit{not} appropriate for purposes of Algorithm~\ref{AL Algorithm for explicit embedding of a recursive group}, because one has to make sure if the function $\phi$ itself is \textit{explicitly} given, and because extra steps must be added to go from the finitely presented overgroup obtained for $\phi(A_{\!\mathcal X})$ (or for $\phi^{-1}(A_{\!\mathcal X})$) to that needed for $A_{\!\mathcal X}$.

\medskip 
We overcome the above two problems by modifying the constructions to make sure, that after each of the Higman operations \eqref{EQ Higman operations}, the newly obtained subgroup  $A_{\mathcal Y}$ is benign in $F_3$ \textit{necessarily}, see details in Section~\ref{SU Building K_X and L_x for the benign subgroup A_X in F} below.

\subsubsection{Proceeding from the benign subgroup of $F_3$ to a benign subgroup of $F_2$}
\label{SU Proceeding from the benign subgroup AX of <a b c> to a benign subgroup of <x y>}

Further \cite{Higman Subgroups of fP groups} uses the above benign subgroup $A_{\mathcal X}$ of $F_3=\langle a,b,c\rangle$, and the obtained groups 
$K_{\mathcal X}, L_{\mathcal X}$, to const\-ruct an auxiliary benign subgroup in the free group $F_2=\langle x,y\rangle$, together with its fi\-ni\-tely presented overgroup $K$ and the finitely generated subgroup $L \le K$.

This step of \cite{Higman Subgroups of fP groups} is explicit, and we could use it as it is for our construction. However, we use a much simpler idea with a new group $F_{\mathcal X}$, see Section~\ref{SU Reusing Higman's final step with Higman Rope Trick} below.

\subsubsection{Returning from $F_2$ to the initial free group $F$}
\label{SU Returning from F_2 to F}

As we warned in Point~\ref{SU Construction of X from R}, the initial group $G = \langle\, X \mathrel{|} R \,\rangle$ was given as a factor group $G \cong F/\langle R \rangle^F$ of some free group $F = \langle X \rangle$. 
Whereas $K_{\mathcal X}, L_{\mathcal X}$ are found for another free group $F_2=\langle x,y\rangle$. 

Higman replaces this $F$ by the $2$-generator group $F_2$ in order to replace all the relations from $\langle R \rangle^F$ by some words in two letters $x,y$, and thus, to be able to extract the set of sequences $f\in \mathcal X$ from $R$, see Point~\ref{SU Construction of X from R} above.

\medskip
We confirm that this step in \cite{Higman Subgroups of fP groups} is explicit, and one could use it as it is. 
However, we do \textit{not} need this step, at all, because, as we will mention in Section~\ref{SU The embedding alpha into a 2-generator group}, we first embed $G$ into an auxiliary $2$-generator group $T_G$, and continue the construction for this $T_G$. That is, the groups $K_{\mathcal X}, L_{\mathcal X}$ obtained after the previous step, can be used with \textit{no changes}, see also Section~\ref{SU Reusing Higman's final step with Higman Rope Trick}.

\subsubsection{The ``Higman Rope Trick''}
\label{SU The Higman Rope Trick FIRST occurance}
At the very end, \cite{Higman Subgroups of fP groups} proceeds to Step 5 with the ``Higman Rope Trick'' for some HNN-extension of the direct product $K\!\times \!G$ mentioned in Point~\ref{SU Step 5. Final embedding by Higman Rope Trick} above. The proofs of this step also are explicit, and we can use them with minimal adaptation to build the wanted finitely presented group $\mathcal G$ containing both $T_G$ and $G$. We use this idea with minor changes only, see Section~\ref{SU Reusing Higman's final step with Higman Rope Trick}.

\subsubsection{Conclusion}
\label{SU Conclusion}

As points~\ref{SU Construction of X from R}--\ref{SU The Higman Rope Trick FIRST occurance} with examples indicate, 
Higman's original embedding involves not only steps that demand extra amendments to be made explicit, but also steps for which it is entirely unclear \textit{how} to develop their explicit analogs.

Even for some uncomplicated groups, such as $\Q$, it is virtually impossible to build the needed embedding by just following Higman's steps literally, compare Example~\ref{EX Going from Q to X is harsh PART 1} above with examples~\ref{EX Going from Q to X is harsh PART 2},
\ref{EX Going from Q to X is harsh PART 3} in Chapter~\ref{SE The modified construction of the current work}.  

Higman's proof is \textit{not} designed to be carried out on concrete groups.
It is hard to say whether it could be possible to append  some new chapters to his original proof to try to make it explicit.
In any case, doing so would produce a much longer proof cluttered with details, which might obscure the main purpose of Higman's work.

The central aim of Higman’s seminal research 
in \cite{Higman Subgroups of fP groups} is arguably to reveal  deep connections between Computability theory and Group theory. 
Hence, adding much further technical detail could be counterproductive to that aim.

\medskip 
On the other hand, Higman's construction contains \textit{no} theoretical obstacles to developing its explicit version, albeit at the cost of very substantial modifications.  The current work is an attempt to do that.

\subsection{Concerning the remarks of Valiev, Bridson and Nyberg-Brodda}
\label{SU Concerning the remarks of Valiev, Bridson and Nyberg-Brodda}

In Section~\ref{SU Earlier mentions in the literature} we recalled the observations
of Johnson \cite{Johnson on Higman's interest},
of Valiev \cite{Valiev example},
and of Bridson and Nyberg-Brodda \cite{Bridson Nyberg-Brodda 2025} on the Higman Embedding Theorem.
Having examined the steps of Higman's construction~\cite{Higman Subgroups of fP groups}
in detail in points~\ref{SU Construction of X from R}--\ref{SU Conclusion}
of Section~\ref{SU Is Higman's construction explicit}, we are now in a position
to return to these observations.
For a glossary of the terms \textit{recursive}, \textit{effective}, \textit{computable},
\textit{constructive}, and \textit{explicit}, we again refer the reader to
Section~\ref{SU Necessary terms on computability}.

\medskip  
The observation of Bridson and Nyberg-Brodda \cite{Bridson Nyberg-Brodda 2025}:
\textit{``Higman’s proof of the Embedding Theorem is effective...''}, see Point~\ref{SU An observation by Bridson and Nyberg-Brodda}, agrees exactly with the characterization of the construction as \textit{``effective''} given by Higman himself on page~456 of~\cite{Higman Subgroups of fP groups}; 
with the terminology of Davis's textbook~\cite{Davis} which Higman cites as his reference for computability theory; 
and with the terms we used in
points~\ref{SU Construction of X from R}--\ref{SU Conclusion}.
As we saw, Higman's proof relies on sets that are certainly recursively
enumerable (such as the set $\mathcal X$, obtained from the normal closure of the recursively enumerable set of defining relations $R$ in the free group $F$, see Point~\ref{SU Construction of X from R}). Under the Church--Turing thesis effective and recursive are synonyms, and so such sets \textit{can} be enumerated as the range of some effective function, see Point~\ref{PO Recursive enumeration, recursive and computable sets}.
However, a proof that such an effective function \textit{exists} does not by
itself yield an explicit (constructive) algorithm to find it.
The \textit{effective} construction in \cite{Higman Subgroups of fP groups} proves theoretical existence of the embedding, without a method of how to construct it. 

\medskip  
We can be much less confident about Valiev's observation~\cite{Valiev example}:
\textit{``G. Higman's proof is constructive in the sense that from a constructive
description of the group $G$ one can, in principle, extract...''} , see Point~\ref{SU An observation by Valiev}.
The use of \textit{``constructive''} may suggest that the method is
explicit, whereas \textit{``can, in principle''} and \textit{``it is practically impossible to do this''} suggest that the
extraction is possible only in theory.
We examined Valiev's original article~\cite{Valiev example} in detail, as well
as~\cite{Valiev}, but neither contains references to the literature on Computability theory or Mathematical logic that would allow us to verify the precise meaning of the terms he used more than 50 years ago.

\bigskip
\section{An outline of our modified construction}
\label{SE The modified construction of the current work}

\noindent 
Some steps of Algorithm~\ref{AL Algorithm for explicit embedding of a recursive group} are substantially different from Higman's construction, while others follow \cite{Higman Subgroups of fP groups}, adding where needed, some elements to guarantee explicitness of the embedding. In this chapter we look at the steps of Algorithm~\ref{AL Algorithm for explicit embedding of a recursive group}, firstly, to refer the reader to the respective proofs (inside or outside this article), and secondly, to stress the parts with substantial changes to \cite{Higman Subgroups of fP groups}.

\smallskip
As mentioned in the \nameref{SE Introduction}, our proofs here relay on constructions that we have proposed in \cite{Embeddings using universal words,
The Higman operations and  embeddings,
Auxiliary free constructions for explicit embeddings, 
On explicit embeddings of Q}.  
In order to avoid any repetition of parts from those articles, we outline some brief results in sections 
\ref{SU The embedding alpha into a 2-generator group},
\ref{SU The *-construction},
\ref{SU Construction of the group A},
\ref{SU Computing the conjugation of by d j}. 

In particular, in Section~\ref{SU The embedding alpha into a 2-generator group} we present the algorithm for explicit embedding of a countable group $G$ into a $2$-generator group $T_{\!G}$  from \cite{Embeddings using universal words}. We use it as the first tool for our embedding, see  steps~\ref{Step 1 AL Algorithm for explicit embedding of a recursive group}, \ref{Step 2 AL Algorithm for explicit embedding of a recursive group} in Algorithm~\ref{AL Algorithm for explicit embedding of a recursive group}. 

Further Section~\ref{SU The *-construction} introduces the generic $\bigast$-\textit{construction}: a ``nested'' combination of HNN-exten\-sions and free products with amalgamated subgroups. We suggested this construction in
\cite{Auxiliary free constructions for explicit embeddings} as it simplifies and unifies many of our proofs needed for steps \ref{Step 5 AL Algorithm for explicit embedding of a recursive group}--\ref{Step 7 AL Algorithm for explicit embedding of a recursive group} in Algorithm~\ref{AL Algorithm for explicit embedding of a recursive group}. 
Sections~\ref{SU Construction of the group A} 
and 
\ref{SU Computing the conjugation of by d j} define a specific group $\mathscr{A}$ and its basic properties to use them in later sections. 
This group $\mathscr{A}$  has already been used in \cite{On explicit embeddings of Q}.

\subsection{The embedding $\alpha: G \to  T_{\!G}$ \,[Steps 
\ref{Step 1 AL Algorithm for explicit embedding of a recursive group}, \ref{Step 2 AL Algorithm for explicit embedding of a recursive group}
in Algorithm~\ref{AL Algorithm for explicit embedding of a recursive group}]}
\label{SU The embedding alpha into a 2-generator group}

As we will see in sections~\ref{SU Using the 2-generator group to get the set X} and \ref{SU Equality and the final embedding}, serious simplification for the embedding of an initial recursive group $G = \langle\, X \mathrel{|} R \,\rangle$  can be achieved, if $G$ first is embedded into a specific \textit{intermediate} $2$-generator recursive group $T_{\!G}$.

Being recursive, the group $G$ is at most \textit{countable}, and we can apply the Higman, Neumann and Neumann Theorem \cite{HigmanNeumannNeumann} stating that any countable group $G$ can be embedded into a 
$2$-generator group.
It is not hard to deduce from the original proofs in \cite{HigmanNeumannNeumann} that the relations of that $2$-generator group can explicitly be given, in case the relations $R$ are known. However, the process of finding those relations by the methods of \cite{HigmanNeumannNeumann} is not very simple, and it may require very long routine for some groups. 

To simplify this we in \cite{Embeddings using universal words}  have suggested a specific $2$-generator group $T_G$ together with an algorithm explicitly defining the embedding
$\alpha: G \to  T_G$,
and automatically writing down a set of defining relations for $T_{\!G}$ from the given defining relations  $R$ of $G$. 
Let us briefly outline that algorithm with some slight modification of notation from \cite{Embeddings using universal words}. 

Write our recursive group as
$G = \langle\, X \mathrel{|} R \,\rangle = \langle a_1, a_2,\ldots \mathrel{|} R \,\rangle$
where the generators $a_1, a_2,\ldots$ are effectively enumerated, and the relations $w \in R$ are recursively enumerable, i.e., they form the image of some partial recursive function, or roughly speaking, there is an algorithm writing them down one by one (in whatever order).

In the free group
$F_2=\langle
x,y
\rangle$ consider some specific ``universal words'':
\begin{equation}
\label{EQ definition of a_i(x,y)}
a_i(x,y) = y^{(x y^i)^{\,2}\, x^{\!-1}} 
\!\! y^{-x} \!
,\quad i=1,2,\ldots
\end{equation}
If a relation 
$w \!\in R$
is a word of length, say, $k$ on  letters 
$a_{i_{1}},\ldots,a_{i_{k}}\!\! \in X$, then replacing inside $w=w(a_{i_{1}},\ldots,a_{i_{k}})$ each letter $a_{i_{j}}$ by the $i_{j}$\!'th word $a_{i_{j}}(x,y)$ we get a new word:
\begin{equation}
\label{EQ definition of r'_s}
w' (x,y)=
w \big(a_{i_{1}}\!(x,y),\ldots,a_{i_{k}}\!(x,y)\big)
\end{equation}
on just \textit{two} letters $x,y$ in the free group $\langle
x,y
\rangle$, see Example~\ref{EX first presentation for Q SECOND} below for $\Q$.
The set $R'$ of all such new words $w' (x,y)$ in has $\langle
x,y
\rangle$ a normal closure $\langle\, R' \,\rangle^{\! F_2}$\!, the factor group by which is the $2$-generator group $T_{\!G} = F_2 / \langle\, R' \,\rangle^{\! F_2}
= \langle\, x,y \mathrel{|} R' \,\rangle$ we look for.
The map $\alpha$ sending each generator $a_i$ to the word $a_i(x,y) \! \in F_2$, and then to the coset 
$\langle\, R' \,\rangle^{\! F_2} \, a_i(x,y) \in T_{\!G}$ can be continued to the desired embedding of $G$ into $T_{\!G}$:

\begin{Theorem}[Theorem~1.1 in \cite{Embeddings using universal words}]
\label{TH universal embedding}
For any countable group $G = \langle a_1, a_2,\ldots \mathrel{|} R \,\rangle $ the above map $\alpha: a_i \to a_i(x,y)$,\; $i=1,2,\ldots$\,, defines an injective embedding:
\begin{equation}
\label{EQ embedding alpha}
\alpha: G \to  T_{\!G},
\end{equation}
of $G$ into the $2$-generator group 
$
T_{\!G}=\big\langle x,y 
\,\mathrel{|}\,
R'
\big\rangle.
$
\end{Theorem}

For some specific cases the formula for ``universal words'' can be simplified. Say, for a \textit{torsion free} group $G$ the words $a_i(x,y)$ in \eqref{EQ definition of a_i(x,y)} can be replaced by shorter words:
\begin{equation}
\label{EQ definition of bar a_i(x,y)}
a_i(x,y)
=
y^{(x y^i)^{\,2} x^{\!-1}}\!
,\quad i=1,2,\ldots,
\end{equation}
see Theorem~3.2 in \cite{Embeddings using universal words}. 
We stress this case here because \eqref{EQ definition of bar a_i(x,y)} was used for the (torsion free) group $\Q$ in \cite{On explicit embeddings of Q}, and it is going to be used in examples below.

\medskip 
Regardless which of formulas \eqref{EQ definition of a_i(x,y)} or \eqref{EQ definition of bar a_i(x,y)} we use, from Theorem~\ref{TH universal embedding} it is evident that:

\begin{Corollary}
\label{CO if G is recursive then TG is also recirsive}
In the above notation, if the group $G$ is recursive, then $T_{\!G}$ is also recursive.
\end{Corollary}

\begin{Example}
\label{EX first presentation for Q SECOND}
Let us explain
steps 
\ref{Step 1 AL Algorithm for explicit embedding of a recursive group}, \ref{Step 2 AL Algorithm for explicit embedding of a recursive group}
of Algorithm~\ref{AL Algorithm for explicit embedding of a recursive group}, i.e., the transaction from $G$ to $T_{\!G}$, by applying them for the rational group $\Q$. 
As an initial presentation for $\Q$ we may take:
\begin{equation}
\label{EQ Q genetic code}
\Q = \big\langle a_1, a_2,\ldots \mathrel{|} 
a_{k}^{k}=a_{k-1},\;\, k=2,3,\ldots \big\rangle
\end{equation}
where $a_k$ corresponds to the rational ${1 \over k!}$ with $k=2,3,\ldots$\,, see page 70 in \cite{Johnson} and elsewhere.
To apply the algorithm of \cite{Embeddings using universal words} first rewrite each relation $a_k^k=a_{k-1}$ as 
$w_k=a_k^k\,a_{k-1}^{-1}=1$. Then, using the shorter formula \eqref{EQ definition of bar a_i(x,y)} for the (torsion free) group $\Q$,
map each letter $a_k$ to the word 
$
\alpha (a_k )=y^{(x y^k)^{\,2} y^{\!-1}}
$\!\!.
To each relation $w_k$ put into correspondence the new relation
$$
w'_k (x,y)
=\,
(y^k)^{(x y^k)^{\,2} x^{\!-1}} 
y^{-(x y^{k-1})^{\,2} x^{-1}}\!\!,
$$
to get the embedding $\alpha:\Q\to T_\Q$
of $\Q$ into the $2$-generator recursively presented group:
\begin{equation}
\label{EQ TG genetic code}
T_\Q=\big\langle x,y 
\;\mathrel{|}\;
(y^k)^{(x y^k)^{\,2} x^{\!-1}} 
y^{-(x y^{k-1})^{\,2} x^{\!-1}}
\!\!,\;\; k=2,3,\ldots
\big\rangle.
\end{equation}
Any finitely presented overgroup of $T_\Q$ is an overgroup for $\Q$ also, so we have the freedom to continue the work to embed the $2$-generator group  $T_\Q$ (apparently recursive with relations \eqref{EQ TG genetic code}) into a finitely presented overgroup.
See also Remark~\ref{RE One could try to produce X from Q directly} below.
\end{Example}

\subsection{Using the $2$-generator group $T_{\!G}$ to get the set $\mathcal X$
\,[Step  
\ref{Step 3 AL Algorithm for explicit embedding of a recursive group} 
in Algorithm~\ref{AL Algorithm for explicit embedding of a recursive group}]}
\label{SU Using the 2-generator group to get the set X}

As we warned in Point~\ref{SU Construction of X from R},  
following Higman's construction literally, we have to construct from $R$ the subset  $\mathcal X$ of $\E$, and we have some serious problems in that path, see Example~\ref{EX Going from Q to X is harsh PART 1}. In Algorithm~\ref{AL Algorithm for explicit embedding of a recursive group} all those issues vanish, as we work with the $2$-generator recursive group $T_{\!G} 
= \langle\, x,y \mathrel{|} R' \,\rangle$ built above. 
Then each relation $w'(x,y) \in R'$ already is a word on just two letters $x,y$:
\begin{equation}
\label{EQ random relation on x and y}
w(x,y)' = x^{j_0}y^{j_1} \cdots 
x^{j_{2r}}y^{j_{2r+1}}
\end{equation}
written for some integers $j_0,j_1,\ldots,j_{2r},j_{2r+1}$; 
the cases $j_0=0$, or $j_{2r+1}$ are \textit{not} ruled out (e.g., a relation may start by $y$ and it may end by $x$).
In such a case we can  \textit{directly write down} the sequence for $w'$: 
\begin{equation}
\label{EQ f_i occurs first}
f =(j_0, j_1, \ldots ,j_{2r},\,j_{2r+1}),
\end{equation}
i.e.,  \eqref{EQ f_i occurs first} is the respective function (sequence of integers) to be included into $\mathcal X$ for the relation $w'(x,y)$.
The set $\mathcal X$ produced this way for all $w' \in R'$ clearly is recursively enumerable. 

\medskip
See also Section~\ref{SU Equality and the final embedding} with Remark~\ref{RE why embedding into 2-generator was good} stressing why usage of the $2$ generator group $T_{\!G} = F_2 / \langle\, R' \,\rangle^{\! F_2}
= \langle\, x,y \mathrel{|} R' \,\rangle$  is useful  for the \textit{last} steps of our algorithm, also.

\medskip
Let us continue Example~\ref{EX first presentation for Q SECOND} to see how the set $\mathcal X$ may look for the group $T_\Q$ holding $\Q$:

\begin{Example}
\label{EX writing X from TQ for rational Q} 
For the group $\Q$ written via \eqref{EQ Q genetic code} we in Example~\ref{EX first presentation for Q SECOND} found the group $T_\Q$ writ\-ten via \eqref{EQ TG genetic code}.
Each of its relations $w'_k(x,y)$ can be rewritten as:
\begin{equation}
\label{EQ long version for w(x,y)}
w'_k (x,y)
=
x\,y^{-k} x^{-1} y^{-k} x^{-1} \cdot y^k \cdot x\, y^k x\;  y \cdot
x^{-1} y^{1-k} x^{-1} \cdot y^{-1} \cdot
x \, y^{k-1}\, x\, y^{k-1} x^{-1},
\end{equation} 
$k=2,3,\ldots$ in order to match the format of \eqref{EQ random relation on x and y}.
Using the passage from \eqref{EQ random relation on x and y} to \eqref{EQ f_i occurs first} we trivially output the respective sequence:
\begin{equation}
\label{EQ Higman code for Q}
f_k=
\big(1,-k,-1, -k , -1,\; k,\;  1,\;  k, \; 1,\;  1, -1,\;  1\!\!-\!k, -1,-1,\;  1,\;  k\!-\!\!1,\;  1,\;  k\!-\!\!1,\;  -1\big),
\end{equation} 
$k=2,3,\ldots$, see also Section~3.2 in \cite{On explicit embeddings of Q}. 
Thus, we get a subset 
$\mathcal X = \big\{ 
f_k \mathrel{|} k=2,3,\ldots \big\}$ of $\E$
to work with in the next steps (in \cite{On explicit embeddings of Q} this set is denoted by $\mathcal T$).
See also Section~9 in \cite{On explicit embeddings of Q} where we explicitly give two finitely presented overgroups $\mathcal{Q}$ and $T_{\!\mathcal{Q}}$ both holding $\Q$.
\end{Example}

\begin{Remark} 
\label{RE One could try to produce X from Q directly}
Compare the ease of writing the uniformly written sequences \eqref{EQ Higman code for Q}  (differing from each other by the value $k=2,3,\ldots$ \textit{only}),  with the  difficulties and unclear issues with sequences of type \eqref{EQ Long f for Q} (which are not limited even in length) in the original construction in \cite{Higman Subgroups of fP groups}. 
In fact, we earlier started our attempts to explicitly embed $\Q$ into a finitely presented group by trying to write down the set $\mathcal X$ for $\Q$ applying the methods of \cite{Higman Subgroups of fP groups} to  \eqref{EQ Q genetic code}. The futility of those attempts forced us to come to the trick with $T_{\!{\Q}}$.
\end{Remark}

Compare this remark with later Remark~\ref{RE why embedding into 2-generator was good}, in which we stress one more important advantage of the embedding 
$\alpha: G \to  T_{\!G}$.

\subsection{Writing $\mathcal X$ via Higman operations \,[Step  
\ref{Step 4 AL Algorithm for explicit embedding of a recursive group} 
in Algorithm~\ref{AL Algorithm for explicit embedding of a recursive group}]}
\label{SU Writing X via Higman operations}

After the set $\mathcal X$ is known, one has to build this set $\mathcal X$ from $\Zz$ and $\S$ via the operations \eqref{EQ Higman operations}.
As we saw in Point~\ref{SU Explicit construction of X via Higman operations} above, the method with functions $f(n,r)$, $a(r)$, $b(r)$ used in \cite{Higman Subgroups of fP groups} is explicit and, theoretically, it \textit{can} be used in Algorithm~\ref{AL Algorithm for explicit embedding of a recursive group} as it is. However, it involves very many routine steps, and we have \textit{never} used it in our particular embeddings, including those for $\Q$. 

\smallskip
We in \cite{The Higman operations and  embeddings} studied some alternative methods to construct $\mathcal X$ from $\Zz$ and $\S$ via the operations \eqref{EQ Higman operations}, based on the \textit{structure} of its sequences, see Remark~3.8, Example~4.11 and Remark~4.12 in \cite{The Higman operations and  embeddings}. These alternative methods do not cover the cases of all recursive groups, but they are usable for wide classes of groups, such as the free abelian, metabelian, soluble, nilpotent groups, the quasicyclic group $\Co_{p^\infty}$, 
divisible abelian groups, etc.
In particular, for the group $\Q$ this method has been applied in \cite{On explicit embeddings of Q} to construct $\mathcal X$ via some elementary steps.      

Also, in \cite{The Higman operations and  embeddings} we in addition to \eqref{EQ Higman operations} suggested certain ``auxiliary'' operations \eqref{EQ auxiliary Higman operations} that simplify the process even more. They allow to simultaneously apply more than one Higman operations at once.

\begin{Example}
\label{EX Going from Q to X is harsh PART 2}
For the group $\Q$ we in \cite{On explicit embeddings of Q} did not use Higman's functions $f(n,r)$, $a(r)$, $b(r)$, at all. Instead, we deduced the process of construction of $\mathcal X$ by \eqref{EQ Higman operations} from the simple ``pattern'' of the sequences  of type \eqref{EQ Higman code for Q} in 
$\mathcal X = \big\{  
f_k \mathrel{|} k=2,3,\ldots \big\}$, also compare to sequence \eqref{EQ Long f for Q} in Example~\ref{EX Going from Q to X is harsh PART 1} above.

Under simple ``pattern'' we mean that all the sequences $f_k$ in this $\mathcal X$ for $\Q$  are of length $19$, and they differ from each other in $k$ only. Also, their coordinates are mostly $\pm 1$, and the few other coordinates are 
either $\pm k$ or $\pm (k-1)$. As it is shown in Chapters 4\,--7
of \cite{On explicit embeddings of Q}, from this structure it is not hard to construct this particular set $\mathcal X$ by \eqref{EQ Higman operations}. 
Those chapters are not brief enough to explain them here, but the reader can check in \cite{On explicit embeddings of Q}, how we:
\begin{itemize}
\item[$\circ$]  
go from to $\mathcal Z$ to $\zeta_1 \mathcal Z$ in Section 4.2 (here $\zeta_1 = \sigma  \zeta \sigma^{-1}$);

\item[$\circ$]  
then go from $\mathcal S$ to $\tau \mathcal S$ in Section 4.3;

\item[$\circ$]  
next, we go from these two already obtained sets $\zeta_1 \mathcal Z$ and $\tau \mathcal S$ to the set $\upsilon(\zeta_{\!1} \Zz , \tau\S)$ in Section 4.4,  etc... 
\end{itemize}
By such steps we eventually conclude the construction of  $\mathcal X = \big\{ 
f_k \mathrel{|} k=2,3,\ldots \big\}$ via \eqref{EQ Higman operations} in Point~7.5 of \cite{On explicit embeddings of Q}. 

\end{Example}

\subsection{Building $K_{\!\mathcal X}$ and $L_{\!\mathcal X}$ for the benign subgroup $A_{\!\mathcal X}$ in $F_3$ and in $F$
\,[Step  
\ref{Step 5 AL Algorithm for explicit embedding of a recursive group} 
in Algorithm~\ref{AL Algorithm for explicit embedding of a recursive group}]}
\label{SU Building K_X and L_x for the benign subgroup A_X in F}

Assume the set $\mathcal X$ has already been built from the sets $\Zz$ and $\S$, using the operations \eqref{EQ Higman operations} in any of two methods mentioned in the Step~\ref{Step 4 AL Algorithm for explicit embedding of a recursive group} of Algorithm~\ref{AL Algorithm for explicit embedding of a recursive group}.

Following \cite{Higman Subgroups of fP groups}, we use the conjugates
$b_i=b^{c^i}$\!,\, $i\in \Z$, to build the product $b_f$, and the conjugate $a_f = a^{b_f}$\! for each of the above sequences $f\in \mathcal X$, see the notation \eqref{EQ defining b_f and a_b} in Section~\ref{SU Defining subgroups by integer sequences}. 
All such conjugates $a_f$ generate inside $F_3\!=\langle a, b, c \rangle$ the subgroup $A_{\mathcal X} =\langle a_f \mathrel{|} f \in \mathcal X \rangle$ corresponding to the set $\mathcal X$.

Since in our modified construction utilizes much smaller sets $\mathcal X$, we get much simpler elements $b_f$, $a_f$ that belong to a single specific type only. This much simplifies our work with the respective subgroup $A_{\mathcal X}$, see Section~3.2 in \cite{On explicit embeddings of Q}, and Example~\ref{EX Going from Q to X is harsh PART 3} below. 

\medskip 
We next build the finitely presented overgroup  $K_{\!\mathcal X}$ of $F_3$, and the finitely generated subgroup $L_{\!\mathcal X}\le K_{\!\mathcal X}$ for the \textit{benign} subgroup $A_{\!\mathcal X}$ in $F_3$, such that $F_3 \cap L_{\!\mathcal X} = A_{\!\mathcal X}$ holds.
This is done in the proof of Theorem~\ref{TH Theorem A} in Chapter~\ref{SE Theorem A and its proof steps}, which follows the general scheme of the proof for Theorem 4 in Chapter~4 of \cite{Higman Subgroups of fP groups}.

However, to address two problems mentioned in Point~\ref{SU Obtaining KX and LX for the benign subgroup AX} above, we have done the following main changes:

\smallskip
\textit{Change 1.} As we warned in \ref{SU Obtaining KX and LX for the benign subgroup AX},\, \cite{Higman Subgroups of fP groups} often ``switches'' the groups in which $A_{\!\mathcal X}$ is benign, and this creates problems for the construction of the explicit embedding. Theorem~\ref{TH Theorem A} below always makes sure the current benign subgroup $A_{\!\mathcal X}$ is benign in $F_3$ \textit{necessarily}. 
This is important because we may use the operations from \eqref{EQ Higman operations} for many times, so that each operation accepts an initial $A_{\!\mathcal X}$ benign in $F_3$ (with explicitly known $K_{\!\mathcal X}$ and $L_{\!\mathcal X}$), then it produces a new set $\mathcal Y$ from $\mathcal X$, and outputs a new subgroup $A_{\!\mathcal Y}$ benign again in $F_3$ (still with explicitly known $K_{\!\mathcal Y}$ and $L_{\!\mathcal Y}$).
This allows us to rename $\mathcal Y$ by $\mathcal X$ after each step, and to repeat such steps as many times as needed to arrive to the desired set $\mathcal X$ by which we in Step \ref{Step 4 AL Algorithm for explicit embedding of a recursive group} in Algorithm~\ref{AL Algorithm for explicit embedding of a recursive group} have ``coded'' the defining relations of $T_{\!G}$.

\smallskip
\textit{Change 2.}
Theorem 4 in \cite{Higman Subgroups of fP groups} discusses theoretical \textit{possibility} of construction of $K_{\!\mathcal X}$ and  $L_{\!\mathcal X}$ without writing them down explicitly for all operations \eqref{EQ Higman operations}, whereas in each point of Theorem~\ref{TH Theorem A} we make sure that \textit{the current} $K_{\!\mathcal X}$ is explicitly written (via generators and defining relations), and $L_{\!\mathcal X}$ is explicitly indicated (by its generators) after each application of any of Higman operations.

\medskip
The proofs for this step are in Chapter~\ref{SE Theorem A and its proof steps}, and they consume the most part of this article, see sections~\ref{SU The proof for the case of Z and S}, \ref{SU The proof for iota and upsilon}, \ref{SU The proof for rho}\,--\,\ref{SU The proof for omega_m} in the proof of Theorem~\ref{TH Theorem A}, and the respective groups  $K_{\!\mathcal X}$ and  $L_{\!\mathcal X}$ for all operations \eqref{EQ Higman operations} are explicitly recorded in 
\ref{SU The proof for the case of Z and S}, 
\ref{SU The proof for iota and upsilon}, 
\ref{SU Writing K rho X  by its generators and defining relations}, 
\ref{SU Writing K sigma X  by its generators and defining relations SHORT}, 
\ref{SU Writing K zeta X  by its generators and defining relations}, 
\ref{SU Writing K pi X  by its generators and defining relations}, 
\ref{SU Writing K theta X  by its generators and defining relations}, 
\ref{SU Writing K tau X  by its generators and defining relations}, 
\ref{SU Writing K_omega B by generators and defining relations}.

\begin{Example}
\label{EX Going from Q to X is harsh PART 3}
Continuing 
Example~\ref{EX Going from Q to X is harsh PART 2}, we for the set $\mathcal X = \big\{ 
f_k \mathrel{|} k=2,3,\ldots \big\}$ of sequences \eqref{EQ Higman code for Q} easily get the products $b_{f_k}$ and the conjugates $a_{f_k}$ in $F_3$. Say, for the sequence 
$$
f_3=
\big(1,-3,-1, -3 , -1,\; 3,\;  1,\;  3, \; 1,\;  1, -1,\;  -2, -1,-1,\;  1,\;  2,\;  1,\;  2,\;  -1\big)
$$
choosen for $k=3$, we have:
$$
b_{f_3} = 
b_0^{1}\,
b_1^{-3}
b_2^{-1}
b_3^{-3}
b_4^{-1}
b_5^{3}
b_6^{1}
b_7^{3}
b_8^{1}
b_9^{1}
b_{10}^{-1}
b_{11}^{-2}
b_{12}^{-1}
b_{13}^{-1}
b_{14}^{1}
b_{15}^{2}
b_{16}^{1}
b_{17}^{2}
b_{18}^{-1},
$$
together with the respective conjugate:
$$
a_{f_3} = a^{b_{f_3}}
= 
b_{18}^{1}\,
b_{17}^{-2}\,
\cdots 
b_{1}^{3}\,
b_{0}^{-1}
\; \cdot \;
a 
\; \cdot \; 
b_0^{1}\,
b_1^{-3}
\cdots \,
b_{17}^{2}\,
b_{18}^{-1} .
$$ 

Let us  indicate that for the group $\Q$ the terminal set of sequences $\mathcal X$ is denoted by $\mathcal T$ in \cite{On explicit embeddings of Q}. The  explicit procedure of construction of the respective $K_{\mathcal X}=K_{\mathcal T}$ and $L_{\mathcal X}=L_{\mathcal T}$ is completed in Sections 4\,--7 of \cite{On explicit embeddings of Q}, and they are written down in Section 7.6 of \cite{On explicit embeddings of Q}, see also \cite{AV24}. 
\end{Example}

\subsection{The ``Higman Rope Trick'' for the $2$-generator case 
\,[Steps  
\ref{Step 6 AL Algorithm for explicit embedding of a recursive group}\,--\,\ref{Step 8 AL Algorithm for explicit embedding of a recursive group} 
in Algorithm~\ref{AL Algorithm for explicit embedding of a recursive group}]}
\label{SU Reusing Higman's final step with Higman Rope Trick}

For a subset $\mathcal X\subseteq \E$ we in Chapter~\ref{SE Theorem B and the final embedding} define a new group $T_{\!\mathcal X}$ via \eqref{EQ introducing T_X}, and prove Theorem~\ref{TH Theorem B} on its explicit embedding $\beta: T_{\!\mathcal X} \to \mathcal G$  into a finitely presented group $\mathcal G$, provided that the groups  $K_{\!\mathcal X}$, $L_{\!\mathcal X}$ are explicitly known for the benign subgroup $A_{\!\mathcal X}$ of $F$. 
Theorem~\ref{TH Theorem B} has been proved by some adaptations of Higman's constructions from Section 5 in \cite{Higman Subgroups of fP groups}, and the group $\mathcal G$ is explicitly given via  \eqref{EQ relations of G FULL}. Also see some related discussion in \cite{Higman rope trick}. 

Then we once again use the advantage that our embedding was through the $2$-generator group $T_{\!G}$, because in such a particular case $T_{\!\mathcal X}$ simply  \textit{is equal} to the group $T_{\! G}$ from Step~\ref{Step 1 AL Algorithm for explicit embedding of a recursive group}.
Hence, $\beta$ is an embedding of $T_{\! G}$ into $\mathcal G$ also, and as an explicit embedding $\varphi$ of the initial $G$ into  $\mathcal G$ one can take the composition $\varphi: G \to \mathcal G$ of $\alpha$ from 
Step~\ref{Step 1 AL Algorithm for explicit embedding of a recursive group} with this $\beta$, see Section~\ref{SU Equality and the final embedding}.

\smallskip
For the \textit{optional} Step \ref{Step 8 AL Algorithm for explicit embedding of a recursive group} 
in Al\-go\-rithm~\ref{AL Algorithm for explicit embedding of a recursive group}
we need one more explicit embedding $\gamma : \mathcal G \to T_{\!\mathcal G}$ of $\mathcal G$  into a certain $2$-generator finitely presented $T_{\!\mathcal G}$ built in Section~\ref{SU Embedding G into the 2-generator group TG}, again utilizing the ``universal words'' from \cite{Embeddings using universal words}.
Then the embedding $\psi:G\to T_{\!\mathcal G}$ is the composition \eqref{EQ composition psi} of $\varphi$ with $\gamma$.

\bigskip
\section{Some auxiliary constructions}
\label{SE Some auxiliary constructions}

\subsection{The $\bigast$-construction}
\label{SU The *-construction} 

Let us begin with the brief notation of the $\bigast$-\textit{construction }and  its basic properties. All the proofs with many more examples are given in \cite{Auxiliary free constructions for explicit embeddings}, and here we just state them in a format comfortable for  the chapters to follow.

Let $G, M, K_1,\ldots,K_r$ be any groups such that the conditions:
$$
\text{$G\le M  \le K_1,\ldots,K_r$ \;\;and\;\; $K_i \cap\,  K_j=M$}
$$
hold for any distinct indices $i,j=1,\ldots,r$.
Picking a subgroup $L_i\le K_i$ for each $i=1,\ldots,r$
we first build the HNN-extensions $K_i *_{L_i} t_i$ of the group $K_i$ with the base subgroup $L_i$ and with the stable letter $t_i$, and next using these $r$ groups we construct an auxiliary ``nested'' free construction:
\begin{equation}
\label{EQ initial form of star construction}
\Big(\cdots
\Big( \big( (K_1 *_{L_1} t_1) *_M (K_2 *_{L_2} t_2) \big) *_M  (K_3 *_{L_3} t_3)\Big)\cdots 
\Big) *_M  (K_r *_{L_r} t_r)
\end{equation}
by amalgamating all these HNN-extensions in their common subgroup $M$. 
For the sake of briefness let us denote the above bulky construction \eqref{EQ initial form of star construction} via
\begin{equation}
\label{EQ star construction short form} 
\textstyle
\bigast_{i=1}^{r}(K_i, L_i, t_i)_M,
\end{equation}
and also agree to set $A_i = G \cap \, L_i$ for each $i$, see Figure~\ref{FI Figure_01_Star_construction}.

\begin{figure}[h]
\includegraphics[width=390px]{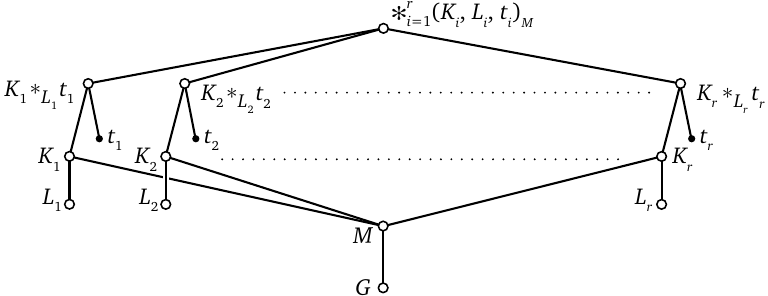}
\caption{Construction of the group\, $\bigast_{i=1}^{r}(K_i, L_i, t_i)_M$ in \eqref{EQ star construction short form}.} 
\label{FI Figure_01_Star_construction}
\end{figure}

The group 
\eqref{EQ star construction short form} 
may coincide with certain well known constructions in some specific particular cases:

\begin{Example}
When a certain group $G$ with its subgroups $A_1,\ldots,A_r$ is fixed, then taking $K_i=M=G$ and $L_i=A_i$ for each $i=1,\ldots,r$,  we 
have the HNN-extensions $G *_{A_i} t_i$, and then 
$\bigast_{i=1}^{r}(G, A_i, t_i)_G$ is the amalgamation of all such $G *_{A_i} t_i$ by their subgroup $G$. That is nothing but the usual HNN-extension with multiple stable letters $t_1,\ldots,t_r$: 
$$
G *_{A_1,\ldots,\,A_r } \!(t_1,\ldots,t_r),
$$
compare with the more general Lemma~\ref{LE intersection in bigger group multi-dimensional} below.
\end{Example} 

\begin{Example}
If we again put $K_i=M=G$ for each $i=1,\ldots,r$, and choose \textit{trivial} subgroups  $L_i=A_i=\1$, then $\bigast_{i=1}^{r}(G, A_i, t_i)_G$ simply is the ordinary free product of $G$ with the free group $\langle t_1,\ldots,t_r \rangle\cong F_r$ of rank $r$:
$$
G*\langle t_1 \rangle * \cdots * \langle t_r \rangle 
= G * \langle t_1,\ldots,t_r \rangle
= G * F_r.
$$ 
\end{Example} 

\begin{Example}
Consider the case when $L_i=K_i=M=G$ for each $i=1,\ldots,r$.  Then in $G *_{G} t_i$ conjugation by stable letter $t_i$ just fixes the whole $G$, which means this HNN-extension is the direct product $G \times \langle t_i \rangle$ for $\langle t_i \rangle \cong \Z$. Then the $\bigast$-construction $\bigast_{i=1}^{r}(G, G, t_i)_G$ turns out to be  the \textit{direct} product: 
$$
G \times \langle t_1,\ldots,t_r \rangle
\cong
G \times F_r.
$$ 
\end{Example}

The main reason why we introduce this construction is that many of rather complicated constructions, used in \cite{Higman Subgroups of fP groups} and elsewhere in the literature, turn out to be \textit{particular cases} of \eqref{EQ star construction short form} even if they are looking very differently. Hence, we found it reasonable to define one general construction and to collect its basic properties in \cite{Auxiliary free constructions for explicit embeddings} in order to refer to them wherever needed.  

The below lemmas 
\ref{LE intersection in bigger group multi-dimensional}\,--\,\ref{LE Ksi for G} are proven in sections 3, 4 in \cite{Auxiliary free constructions for explicit embeddings}.

\begin{Lemma}
\label{LE intersection in bigger group multi-dimensional}
If $G\le M  \le K_1,\ldots,K_r$ are groups mentioned above, then in\, ${\bigast}_{i=1}^{r}(K_i, L_i, t_i)_M$
the following equality holds:
$$
\langle G, t_1,\ldots,t_r \rangle= 
G *_{A_1,\ldots,\,A_r } \!(t_1,\ldots,t_r).
$$
\end{Lemma}

\begin{Lemma}
\label{LE intersection in HNN extension multi-dimensional}
Let $A_1,\ldots,\,A_r$ be any subgroups in a group $G$ with the intersection 
$I=\bigcap_{\,i=1}^{\,r} \,A_i$. 
Then in $G *_{A_1,\ldots,\,A_r} (t_1,\ldots,t_r)$ we have:
\begin{equation}
\label{EQ gemeral intersection in HNN}
\textstyle
G \cap G^{t_1 \cdots\, t_r}
= I.
\end{equation}
\end{Lemma}

\begin{Lemma}
\label{LE join in HNN extension multi-dimensional}
Let $A_1,\ldots,\,A_r$ be any subgroups in a group $G$ with the join
$J=\big\langle\bigcup_{\,i=1}^{\,r} \,A_i\big\rangle$. 
Then in $G *_{A_1,\ldots,\,A_r} (t_1,\ldots,t_r)$ we have:
\begin{equation}
\label{EQ gemeral intersection in HNN multi-dimensional}
\textstyle 
G \cap \big\langle 
\bigcup_{\,i=1}^{\,r} \,G^{t_i} \big\rangle
=J.
\end{equation}
\end{Lemma}

These lemmas allow to build new benign subgroups from the existing ones:

\begin{Corollary}
\label{CO intersection and join are benign multi-dimensional}
If the subgroups $A_1,\ldots,\,A_r$ are benign in a finitely generated group $G$, then:
\begin{enumerate}
\item 
\label{PO 1 CO intersection and join are benign multi-dimensional}
their intersection $I=\bigcap_{\,i=1}^{\,r} \,A_i$ is also benign in $G$;
\item 
\label{PO 2 CO intersection and join are benign multi-dimensional}
their join $J=\big\langle\bigcup_{\,i=1}^{\,r} \,A_i\big\rangle$ is also benign in $G$.
\end{enumerate}
Moreover, if the finitely presented groups $K_i$ with their finitely generated subgroups $L_i$ can be 
given for each $A_i$ explicitly, then the respective finitely presented overgroups $K_I$ and $K_J$ with finitely generated  subgroups $L_I$ and $L_J$
can also be given for $I$ and for $J$ explicitly.
\end{Corollary}

Check Section~4.3 in \cite{Auxiliary free constructions for explicit embeddings} to see that as $K_I$ and $K_J$ one may take $\textstyle{\bigast}_{i=1}^{r}(K_i, L_i, t_i)_M$ which evidently is \textit{finitely presented} for $M=G$. Also, one may choose the subgroups $L_I=G^{t_1 \cdots\, t_r}$ and $L_J=\big\langle 
\bigcup_{\,i=1}^{\,r} \,G^{t_i} \big\rangle$ which evidently are \textit{finitely generated}. 
See Figure~6 in 
\cite{Auxiliary free constructions for explicit embeddings} illustrating the proof of Corollary~\ref{CO intersection and join are benign multi-dimensional}.  

\begin{Remark}
\label{RE one purpose of the *-construction}
Corollary~\ref{CO intersection and join are benign multi-dimensional} stresses one of the reasons why the $\bigast$-construction
\eqref{EQ star construction short form} may be helpful in composition of finitely presented groups containing the given $G$. If the subgroups $A_1,\ldots,\,A_r$ are \textit{not} finitely generated, then $G *_{A_1,\ldots,\,A_r } \!(t_1,\ldots,t_r)$ may \textit{not} be finitely presented, since the non-finitely generated subgroups $A_i$ may add infinitely many new defining relations for this HNN-extension. 
However, if all  $A_i$ are \textit{benign}, we can embed that HNN-extension into a \textit{finitely presented} $\bigast$-construction $\bigast_{i=1}^{r}(K_i, L_i, t_i)_M$ in which we have the freedom to choose as large finitely presented groups $K_i$ as needed, just making sure the subgroups $L_i$ and $M$ are finitely generated (then they will bring just finitely many new defining relations for \eqref{EQ star construction short form}). This trick will be used repeatedly below, and in many cases the choice $M=G$ will already be enough.
\end{Remark} 

The following technical fact 
proved in Section~4.4 of \cite{Auxiliary free constructions for explicit embeddings} displays some ``bigger'' free products inside HNN-extensions and inside $\bigast$-constructions \eqref{EQ star construction short form}, as soon as some ``smaller'' free products are known inside $G$:

\begin{Corollary}
\label{CO smaller free product to the larger free product}
Let $A_1,\ldots,\,A_r$ be any subgroups in a group $G$ such that their join $J$ in $G$ is isomorphic to their free product
$\prod_{i=1}^{r} \,A_i$.  
Then the join
$\big\langle 
\bigcup_{\,i=1}^{\,r} G^{t_i} \big\rangle$
is isomorphic to the free product $\prod_{i=1}^{r} G^{t_i}$
in $G *_{A_1,\ldots,\,A_r } \!(t_1,\ldots,t_r)$, and hence in $\bigast_{i=1}^{r}(K_i, L_i, t_i)_M$.
\end{Corollary}

\subsection{Construction of the group $\mathscr{A}$}
\label{SU Construction of the group A}   

Let us use the tools above to construct the group $\mathscr{A}$ to use it in the proofs in the coming sections. 
In the free group $\langle b,c \rangle$ of rank $2$ we denoted  $b_i = b^{c^i}\!\!$,\; see Section~\ref{SU Defining subgroups by integer sequences}. 
For a fixed integer
$m$ define two isomorphisms
$\xi_m$ and $\xi'_m$ by the rules:
$\xi_m(b)=b_{-m+1},\;
\xi'_m(b)=b_{-m}$,\;
$\xi_m(c)=\xi'_m(c)=c^2$
of $\langle b,c \rangle$, and using them set the HNN-extension: 
\begin{equation}
\label{EQ Ksi two defined}
\Xi_m = \langle b,c \rangle *_{\xi_m, \xi'_m} (t_m, t'_m).
\end{equation}
See Figure~7 of 
\cite{Auxiliary free constructions for explicit embeddings} illustrating the construction of the group $\Xi_m$.  
We have proved the following technical lemmas in \cite{Auxiliary free constructions for explicit embeddings}:
\begin{Lemma}
\label{LE Ksi}
In the above notation the following equalities hold for any $m$ in $\Xi_m$:
\begin{equation}
\label{EQ Ksi two equality}
\begin{split}
\langle b,c \rangle \cap \langle b_m, t_m, t'_m\rangle &= \langle b_m, b_{m+1},\ldots\rangle,
\\
\langle b,c \rangle \cap \langle b_{m-1}, t_m, t'_m\rangle &= \langle b_{m-1}, b_{m-2},\ldots\rangle.
\end{split}
\end{equation}
\end{Lemma}

\begin{Lemma}
\label{LE Ksi for G}
In the above notation the following equalities hold for any $m$ in $\langle a \rangle * \,\Xi_m$:
\begin{equation}
\label{EQ a simple example of bening subgroup}
\begin{split}
F_3 \cap \langle b_m, t_m, t'_m\rangle = \langle b_m, b_{m+1},\ldots\rangle
\;\;\; {\it and} \;\;\;  
F_3 \cap \langle a, b_m, t_m, t'_m\rangle = \langle a, b_m, b_{m+1},\ldots\rangle,
\hskip3mm \\
F_3 \! \cap \!\langle b_{m-1}, t_m, t'_m\rangle  \! =  \! \langle b_{m-1}, b_{m-2},\ldots\rangle
\;\; {\it and} \;\;\,  
F_3 \! \cap \! \langle a, b_{m-1}, t_m, t'_m\rangle \!  =  \! \langle a, b_{m-1}, b_{m-2},\ldots\rangle.
\end{split}
\end{equation}
\end{Lemma}

These lemmas provide us with infinitely generated benign subgroups of four types inside the free groups $\langle b,c \rangle$ and $F_3\!=\langle a,b,c \rangle$. In particular, 
$\langle b_1, b_2,\ldots \rangle$ is benign in $F_3$ for the finitely presented overgroup $\langle a \rangle * \Xi_1$, and for its finitely generated subgroup $\langle b_1, t_1, t'_1\rangle$.
Also the subgroup $\langle a, b_{0}, b_{-1},\ldots\rangle$ is benign in $F_3$ for the same finitely presented $\langle a \rangle * \Xi_1$ and for its finitely ge\-nerated subgroup $\langle a, b_{0}, t_1, t'_1\rangle$.

\begin{figure}[h]
\includegraphics[width=390px]{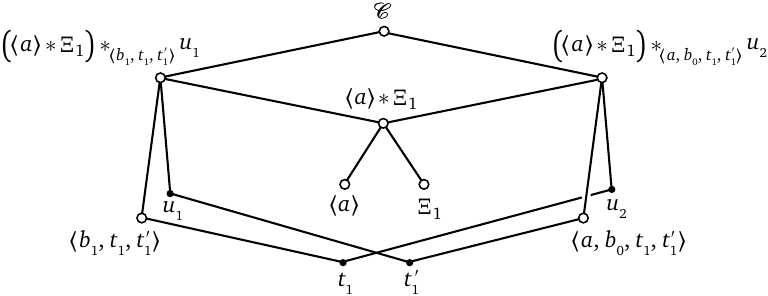}
\caption{Construction of the group  $\mathscr{C}$.}
\label{FI Figure_09_C}
\end{figure}

Use these groups to build the 
$\bigast$-construction:
\begin{equation}
\label{EQ Defining C}
\mathscr{C} 
= \Big( \big(\langle a \rangle\! * \Xi_1\big) *_{\langle b_1, t_1, t'_1 \rangle} u_1 \Big)
\,*_{\langle a \rangle *\, \Xi_1}
\Big(\big(\langle a \rangle\! * \Xi_1\big) *_{\langle a, b_{0}, t_1, t'_1 \rangle} u_2 \Big),
\end{equation}
see Figure~\ref{FI Figure_09_C},
which is finitely presented, and which can explicitly be given by generators and defining relations via:
\begin{equation}
\label{EQ C by generators and defining relations}
\begin{split}
\mathscr{C} & = \big\langle a, b, c, t_1, t'_1, u_1, u_2 \mathrel{\;|\;}  
b^{t_1}= b,\;
b^{t'_1}= b^{c^{-1}}\!\!\!\!,\;\;
c^{t_1}=c^{t'_1}=c^2; \\
& \hskip37mm 
\text{$u_1$ fixes $b^c, t_1, t'_1$};\;\;\;
\text{$u_2$ fixes $a, b, t_1, t'_1$}
\big\rangle
\end{split}
\end{equation}
where ``fixes'' means ``fixes under conjugation'', e.g., $t_1^{u_1}=t_1$.

By Corollary~\ref{CO intersection and join are benign multi-dimensional}\;\eqref{PO 2 CO intersection and join are benign multi-dimensional} the join $J$ of $\langle b_1, b_2,\ldots  \rangle$ and $\langle a, b_{0}, b_{-1},\ldots\rangle$ is benign in $F_3$ for the finitely presented overgroup $K_J=\mathscr{C}$ and for its finitely generated subgroup
$L_J=\langle F_3^{u_1}, F_3^{u_2} \rangle$.

Further, in $F_3$ the subgroups 
$\langle b_1, b_2,\ldots  \rangle$ and $\langle a, b_{0}, b_{-1},\ldots\rangle$
clearly generate their \textit{free} pro\-duct.
Hence, by Corollary~\ref{CO smaller free product to the larger free product} the groups  $F_3^{u_1}$ and $F_3^{u_2}$ also generate their free product $F_3^{u_1} *\, F_3^{u_2}$ in $\mathscr{C}$, as well as in its subgroup: 
$$
F_3*_{F_3 \,\cap\, \langle b_1, t_1, t'_1 \rangle, 
\;\;\; 
F_3\,\cap\,\langle a, b_{0}, t_1, t'_1 \rangle} 
(u_1,u_2)
\;=\;
F_3 *_{\langle b_1, b_2,\ldots  \rangle, \;\; \langle a, b_{0}, b_{-1},\ldots\rangle} (u_1,u_2),
$$ 
see Lemma~\ref{LE intersection in bigger group multi-dimensional} and Lemma~\ref{LE Ksi for G}. Then \textit{arbitrary} two isomorphisms defined on $F_3^{u_1}$ and on $F_3^{u_2}$ can be continued to an isomorphism on the whole subgroup  $F_3^{u_1}* F_3^{u_2}$ inside $\mathscr{C}$.
Choose the trivial automorphism in $F_3^{\,u_1}$ and the conjugation by $b^{u_2}$ in $F_3^{\,u_2}$\!,\; and denote their common continuation in $F_3^{u_1}* F_3^{u_2}$ via $\omega$.
Inside $F_3$ this $\omega$ leaves the elements $ b_1, b_2,\ldots $ intact, but it sends $a, b_{0}, b_{-1},\ldots$ to their conjugates $a^b\!,\, b_{0}^b,\, b_{-1}^b,\ldots$

\begin{figure}[h]
\includegraphics[width=390px]{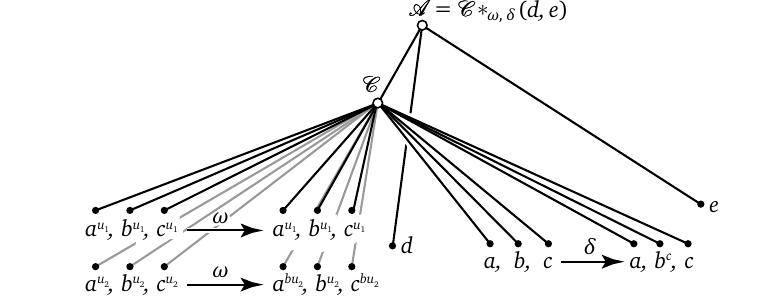}
\caption{Construction of the group  $\mathscr{A}$.}
\label{FI Figure_10_A}
\end{figure}

Next, denote by $\delta$ the isomorphism of $F_3$ sending $a, b, c$ to $a, b^c, c$. Now we can define one of the key technical groups of this article as the HNN-extension:
$$
\mathscr{A} = \mathscr{C} \! *_{\omega, \delta}\!(d,e)
$$
which is finitely presented because  $\mathscr{C}$ is finitely presented, while $\omega$ is determined by its values on just six conjugates $a^{u_1}, b^{u_1}, c^{u_1}, a^{u_2}, b^{u_2}, c^{u_2}$, and 
$\delta$ is determined by its values on just three generators $a, b, c$, see Figure~\ref{FI Figure_10_A}.\,
This group can explicitly be written as:
\begin{equation}
\label{EQ relations A}
\begin{split}
\mathscr{A} 
& \!=\! 
\big\langle a, b, c, t_1, t'_1, u_1, u_2, d,e \mathrel{\;|\;}  
b^{t_1}\!=\! b,\;
b^{t'_1}\!=\! b^{c^{-1}}\!\!\!\!,\;
c^{t_1}\!=\!c^{t'_1}\!=\!c^2; \\[-2pt]
& \hskip17mm 
\text{$u_1$ fixes $b^c\!, t_1, t'_1$};\;\;\;\;
\text{$u_2$ fixes $a, b, t_1, t'_1$}; \\[-2pt]
& \hskip17mm  
\text{$d$ fixes 
$a^{u_1}$\!,\, 
$b^{u_1}$\!,\, 
$c^{u_1}$};  \\[-2pt]
& \hskip17mm \text{$d$ sends $a^{u_2}\!\!,\; b^{u_2}\!\!,\; c^{u_2}$ to $a^{b u_2}\!,\; b^{u_2}\!,\; c^{b u_2}$};\\[-2pt]
& \hskip17mm 
\text{$e$\; sends $a,b,c$ \;to\; $a,b^c\!,\; c$}
\big\rangle,
\end{split}
\end{equation}
with ``$d$ sends $a^{u_2}$ to $a^{b u_2}$''
simply meaning 
$a^{u_2 d}=a^{b u_2}$.
For later purposes denote the set of generators of $\mathscr{A}$ by 
\begin{equation}
\label{EQ generators of XA}
X_{\! \mathscr{A}}=\big\{
a, b, c, t_1, t'_1, u_1, u_2, d,e
\big\}.
\end{equation}
Denoting the set of defining relations of $\mathscr{A}$ from \eqref{EQ relations A} by 
$R_{\! \mathscr{A}}$ we have
$\mathscr{A}=\langle
\,X_{\! \mathscr{A}}
\;|\;
R_{\! \mathscr{A}}
\rangle$, i.e., 
$\mathscr{A}$ is given by $9$ generators and $1+1+2+3+4+3+3+3=20$ relations.

\subsection{Computing the conjugation of $a_f$ by $d_j$ in $\mathscr{A}$}
\label{SU Computing the conjugation of by d j} 

A useful computational feature takes place in $\mathscr{A}$. 
Namely, using the earlier notation 
$f_{j}^+$\!,\, 
$f_{j}^-$\!,\, 
$f^+$\!,\, 
$f^-$ 
from Section~\ref{SU Integer functions f} 
we can for a given $a_f$ consider the elements, say, $a_{f_{j}^+}$ or $a_{f^-}$ in $\mathscr{A}$, \textit{inside} $F_3$.
Also, using the remark about $\langle d,e\rangle$ in Section~\ref{SU Defining subgroups by integer sequences} we can use the elements $d_i$ and $d_f$ in $\mathscr{A}$, \textit{outside} $F_3$.

The following lemma uses this notation, and it is one of the main reasons for the sake of which the group $\mathscr{A}$ was thus constructed:

\begin{Lemma}
\label{LE action of d_m on f} For any $f \in \mathcal E$ and any $j\in \Z$ 
we have 
$
a_f^{d_j} =\! a_{f_{j}^+}$ 
and
$
a_f^{\,d_j^{-1}}\!\! = a_{f_{j}^-}
$ in $\mathscr{A}$.
\end{Lemma}

Its proof is uncomplicated after all preparations above. Hence we just bring simple examples that fully explain the  argument.
Taking, say,  $f=(2,5,3)$ and $j=1$ we write $b_f=b_0^{2}\,b_1^5\,b_2^{3}$ and calculate the routine: 
\begin{equation*}
\begin{split}
a_f^{d_1} & =
\big(
b_2^{\!-3}b_1^{\!-5}b_0^{\!-2}
\; a\;
b_0^{2}b_1^{5}b_2^{3} \,
\big)^{d_1}
\!\! \\
& =
\big(b_2^{\!-3} \big)^{d_1} 
\big(b_1^{\!-5} \big)^{d_1}
\big(b_0^{\!-2} \big)^{d_1}
\; (a)^{d_1} \;
\big(b_0^{2}\big)^{d_1}
\big(b_1^{5}\big)^{d_1}
\big(b_2^{3} 
\big)^{d_1}
\!\! \\
& = 
b_2^{-3}\,
\big(b_1^{\!-1} b_1^{\!-5} b_1^{\vphantom8}\big)
\big(b_1^{\!-1} b_0^{\!-2} b_1^{\vphantom8}\big)
\,\, \big(b_1^{\!-1} a b_1^{\vphantom8}\big) \,
\big(b_1^{\!-1} b_0^{2}b_1^{\vphantom8}\big)\,
\big(b_1^{\!-1} b_1^{5}b_1^{\vphantom8}\big)
\,b_2^{3}\\
& = b_2^{\!-3}
\big(b_1^{\!-1} b_1^{\!-5}\big)
b_0^{\!-2}
\;a\;
b_0^{2}
\big( b_1^{5} b_1^{1}\big)
b_2^{3}
\;=\; b_2^{\!-3}b_1^{\!-\,6}  b_0^{\!-2}
\;a\;
b_0^{2}b_1^6b_2^{3}\\
&= a_{f_{1}^+}
\end{split}
\end{equation*}
for the sequence
$f_{1}^+ \!= (2,\,\boldsymbol{5\!+\!1}\,,3)= (2,\boldsymbol{6},3)$.
Taking $j=2$  we would have $a_f^{d_2}=a_{f_2^+}\!=a_{f^+}$ where 
$f_{2}^+ \!= f^+ \!\!= (2,5,\,\boldsymbol{3\!+\!1})= (2,5,\boldsymbol{4})$.

Hopefully, the calculation routine in the displayed example does not entomb the simple meaning of Lemma~\ref{LE action of d_m on f}: the conjugation by $d^j$ just ``lifts''  by $1$ the exponent of the the factor corresponding to the $j$'th coordinate of $f$ inside $a_f$.

\begin{Remark}
\label{RE order of d_i does not matter}
The following feature of this lemma will be used repeatedly. 
The \textit{order} of elements $d_i$ acting on $a_f$ does \textit{not} matter, i.e., $a_f^{d_{j_1}d_{j_2}}$ and $a_f^{d_{j_2}\,d_{j_1}}$ are equal for any $f, j_1, j_2$. Say, for the above $f=(2,5,3)$ we have  
$a_f=a_{(2,5,3)}^{d_{1}d_{2}}
= a_{(2,5,3)}^{d_{2}d_{1}}
=a_{(2,\boldsymbol 6, \boldsymbol 4)}$.
\end{Remark}

\bigskip
\section{Theorem~\ref{TH Theorem A} and its proof steps 
}
\label{SE Theorem A and its proof steps}

\subsection{Theorem~\ref{TH Theorem A} on benign subgroups}
\label{SU Theorem A on benign subgroups} 

For notation of the sets $\E$, $\Zz$, $\S$ see Section~\ref{SU Integer functions f}, and for the Higman operations $\iota, 
\upsilon, 
\rho, 
\sigma, 
\tau, 
\theta, 
\zeta, 
\pi, 
\omega_m,
$
$m=1,2,\ldots$
in \eqref{EQ Higman operations} over the subsets of $\E$ 
see Section~\ref{SU The Higman operations}. 
Of these operations $\iota, 
\upsilon$  
are \textit{binary}, and the rest are \textit{unary} operations.
For the subgroup $A_{\!\mathcal X}$ defined in the free group $F_3\!=\langle a,b,c \rangle$ for a subset $\mathcal X\! \subseteq \E$ see Section~\ref{SU Defining subgroups by integer sequences}.
If $A_{\!\mathcal X}$ is benign in $F_3$, then we denote the respective finitely presented overgroup of $F_3$ by $K_{\!\mathcal X}$, and the  respective finitely generated subgroup of the latter by $L_{\!\mathcal X}$, see Section~\ref{SU Benign subgroups and Higman operations}. 
Under an \textit{explicitly given} group we understand a group explicitly given by its generators and defining relations, see Section~\ref{SU Recursive enumeration and recursive groups}.

\begin{figure}[h]
\includegraphics[width=390px]{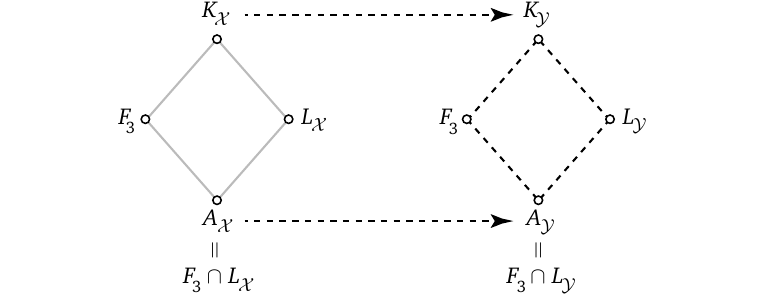}
\caption{An illustration for statement \eqref{PO 3 TH Theorem A} in Theorem~\ref{TH Theorem A}.}
\label{Figure_11_Theorem_A}
\end{figure}

With this notation the theorem below allows to explicitly build new benign subgroups from the existing ones:

\begin{TheoremABC}{A}
\label{TH Theorem A}
Let $\mathcal Y$ be a subset of $\E$ for one of the following cases:
\begin{enumerate}
\item 
\label{PO 1 TH Theorem A}
$\mathcal Y$ is one of sets and $\Zz$ or $\S$;

\item
\label{PO 2 TH Theorem A}
$\mathcal Y$ is obtained from sets $\mathcal X_1, \mathcal X_2 \subseteq \E$ by any of operations  
$\iota$ or $\upsilon$,
with $A_{\mathcal X_1}$, $A_{\mathcal X_2}$  benign in $F_3$;

\item
\label{PO 3 TH Theorem A}
$\mathcal Y$ is obtained from $\mathcal X \subseteq \E$ by any of operations
$
\rho,
\sigma,
\tau, 
\theta, 
\zeta, 
\pi,
\omega_m
$, with $A_{\mathcal X}$  benign in $F_3$.
\end{enumerate}
Then $A_{\mathcal Y}$ is benign in $F_3$. Moreover, for case \eqref{PO 1 TH Theorem A} the groups $K_{\mathcal Y}$ and $L_{\mathcal Y}$ can always be given explicitly. 
For case \eqref{PO 2 TH Theorem A} they can be given explicitly, if $K_{\mathcal X_1}$, $K_{\mathcal X_2}$ and 
$L_{\mathcal X_1}$, $L_{\mathcal X_2}$ are known explicitly.  
For case \eqref{PO 3 TH Theorem A} 
they can be given explicitly, if $K_{\mathcal X}$ and  $L_{\mathcal X}$ are known explicitly.
\end{TheoremABC}

The proof will be given in sections \ref{SU The proof for the case of Z and S}\,--\,\ref{SU The proof for omega_m} below, and the promised groups $K_{\mathcal Y}$ and $L_{\mathcal Y}$ will be  written down explicitly in \ref{SU The proof for the case of Z and S}, 
\ref{SU The proof for iota and upsilon}, 
\ref{SU Writing K rho X  by its generators and defining relations}, 
\ref{SU Writing K sigma X  by its generators and defining relations SHORT}, 
\ref{SU Writing K zeta X  by its generators and defining relations}, 
\ref{SU Writing K pi X  by its generators and defining relations}, 
\ref{SU Writing K theta X  by its generators and defining relations}, 
\ref{SU Writing K tau X  by its generators and defining relations}, 
\ref{SU Writing K_omega B by generators and defining relations}.

\begin{Remark}
\label{RE What is novelty}
Comparing this theorem with lemmas~4.4\,--\,4.10 in \cite{Higman Subgroups of fP groups}, we see that the only  novelty of Theorem~\ref{TH Theorem A} is that it constructs an \textit{explicit} finitely presented group $K_{\mathcal Y}$ and its \textit{explicit} finitely generated subgroup $L_{\mathcal Y}$ for each $\mathcal Y$ involved. Also, Theorem~\ref{TH Theorem A} makes sure each $A_{\mathcal Y}$ is benign in the \textit{same} free group $F_3$ of rank $3$ \textit{necessarily}, see the differences stressed in Section~\ref{SU Building K_X and L_x for the benign subgroup A_X in F}.
These features are required for explicit embedding of recursive groups later.
We stress that some of the steps of Higman's construction \textit{already are explicit} in \cite{Higman Subgroups of fP groups}, and for them we do not need the constructions developed in \cite{Auxiliary free constructions for explicit embeddings}.
\end{Remark}

The extra ``auxiliary'' Higman operations 
$\sigma^i, 
\zeta_i,
\zeta_S,
\pi',
\pi_i,
\pi'_i,
\tau_{k,l},
\alpha,
\epsilon_S,
+,\;
\iota_n,
\upsilon_n$ of \eqref{EQ auxiliary Higman operations} were introduced in \cite{The Higman operations and embeddings} to simplify usage of the Higman operations, see Section~\ref{SU The Higman operations} above. As we have seen in 
Section~2.4 of \cite{The Higman operations and embeddings}, each of \eqref{EQ auxiliary Higman operations} is a combination of some of original Higman operations \eqref{EQ Higman operations}. Hence the analog of  Theorem~\ref{TH Theorem A} holds true for extra ``auxiliary'' operations \eqref{EQ auxiliary Higman operations} also.

\subsection{The proof for the case of $\Zz$ and $\S$}
\label{SU The proof for the case of Z and S} 

The first case $\mathcal Y = \Zz$ is trivial by Remark~\ref{RE finite generated is benign}:  
$A_{\Zz}$ is benign in $F_3$ simply because $A_{\Zz}=
\langle a_{(0)} \rangle
=
\langle a\rangle
$, 
defined by a \textit{single} function $f\!=\!(0)$, is a finitely generated (cyclic) group, and we can just pick 
$K_{\Zz}=F_3$ and $L_{\Zz}=\langle a\rangle$.

\medskip 
For the second case $\mathcal Y = \S$ as the finitely presented group $K_{\S}$ choose $\mathscr{A}$ from Section~\ref{SU Construction of the group A}.
Since $\S$ contains the function $f\!=\!(0,1)$, then repeatedly using the tech\-ni\-cal Lemma~\ref{LE action of d_m on f} with $a_{f}=a_{(0,1)}$  for $2n$ times we have:
\begin{equation}
\label{EQ d0d1 acts on a_01}
a_{(0,1)}^{(d_0 d_1)^n}
\!\!=\big(a_{(0,1)}^{d_0 }\big)^{d_1  \, (d_0 d_1)^{n-1}}
\!\!\!\!=a_{(0+1,\,1)}^{d_1  \, (d_0 d_1)^{n-1}}
\!\!\!=a_{(1,\,1+1)}^{(d_0 d_1)^{n-1}}
\!\!\!=a_{(2,\,2+1)}^{(d_0 d_1)^{n-2}}
\!\!\!= \cdots 
=a_{(n,n+1)}\, ,
\end{equation}
that is, $a_{(n,n+1)}$ belongs to $ \langle
a_{f},\; d_0 d_1
\rangle $ for any $n\in \Z$.
Since we also have $a_{(n,n+1)}\in F_3$, then $A_{\S}\subseteq F_3\cap \langle
a_{(0,1)}, \;d_0 d_1
\rangle$ holds.

On the other hand, applying the ``conjugates collecting'' process \eqref{EQ elements from <x,y>} 
for $x=a_{(0,1)}$
and
$y= d_0 d_1$,
we can rewrite \textit{any} element $w\in \langle
a_{(0,1)},\; d_0 d_1
\rangle$
as $w=u\cdot v$, where $u$ is a product of some conjugates $x^{\pm y^{n_i}}\!=a_{(0,1)}^{\pm (d_0 d_1)^{n_i}}$\!\!\!\!,\, and $v$ is equal to $y^k=(d_0 d_1)^k$ for a certain $n_i, k$.
By \eqref{EQ d0d1 acts on a_01} all those conjugates are in $F_3$.
Thus, if we additionally show that from $w \in F_3$ it follows  $v\in F_3$, then we will have  
$F_3 \cap \langle
a_{f}, \;d_0 d_1
\rangle \subseteq A_{\S}$, which together with the previous inclusion means $F_3 \cap \langle
a_{f}, \;d_0 d_1\rangle
= A_{\S}$, i.e.,  $A_{\S}$ is benign for the above $K_{\S}$ and its finitely generated subgroup 
$L_{\S}=\langle
a_{f},\; d_0 d_1
\rangle$.
As an explicit presentation of $K_{\S}=\mathscr{A}$ just pick \eqref{EQ relations A}.

It remains to verify that $v\in F_3$ takes place for  $v=1$ only.
$\mathscr{A}$ is the ``nested'' HNN-extension: 
$$\mathscr{A}= \mathscr{C} \! *_{\omega, \delta}\!(d,e) = \big(\mathscr{C} \! *_{\omega}\!d \big) 
*_{\delta} e.$$
By uniqueness of normal form in both HNN-extensions it is clear that the product $v = (d^{e^0} d^{e^1})^k  
= (d\, e^{-1}  d\, e)^k$ of the \textit{stable} letters $d,e$ is in normal form in $\mathscr{A}$. A normal form involving only stable letters is inside $\mathscr{C}$
(and in particular, in $F$) \textit{only} if it is trivial.
%

\subsection{The proof for the operations $\iota$ and $\upsilon$}
\label{SU The proof for iota and upsilon} 

Suppose $\mathcal Y = \iota(\mathcal X_1, \mathcal X_2)=\mathcal X_1 \cap \mathcal X_2$, and 
the finitely presented overgroups
$K_{\mathcal X_1}$ and $K_{\mathcal X_2}$ of $F_3$ together with finitely generated subgroups
$L_{\mathcal X_1} \le K_{\mathcal X_1}$
and
$L_{\mathcal X_2} \le K_{\mathcal X_2}$
are explicitly given: $K_{\mathcal X_1}=\langle
\,Z_{1}
\;|\;
S_{1}
\rangle$ and $K_{\mathcal X_2}=\langle
\,Z_{2}
\;|\;
S_{2}
\rangle$, while the (finitely many) generators of $F_3$, $L_{\mathcal X_1}$, $L_{\mathcal X_2}$ can effectively be computed via the generators from $Z_{1}$, $Z_{2}$, respectively.

Since $K_{\mathcal X_1}$ and $K_{\mathcal X_2}$ both are overgroups of $F_3$, their intersection contains $F_3$.
Without loss of generality we may assume $K_{\mathcal X_1} \!\cap K_{\mathcal X_2}$ is strictly \textit{equal} to $F_3$ because these two overgroups are built independently, and the only requirement they share is to contain $F_3$. I.e., we may assume none of the generators of $K_{\mathcal X_1}$, except $a,b,c$, has been used in construction of $K_{\mathcal X_2}$.

Then the $\bigast$-construction $K_{\mathcal Y}=\textstyle{\bigast}_{i=1}^{2}(K_{\mathcal X_i}, L_{\mathcal X_i}, v_i)_F$ 
build for $G=M=F$
is finitely presented, see Figure~\ref{Figure_12_Un_Intersect}, and  
by Lemma~\ref{LE intersection in bigger group multi-dimensional} and 
Lemma~\ref{LE intersection in HNN extension multi-dimensional}
we have: 
$$
F_3 \cap F_3^{v_1 v_2} = A_1 \cap\, A_2 = \big(F_3 \cap L_{\mathcal X_1}\big)\cap 
\big(F_3 \cap L_{\mathcal X_2}\big) 
\,=\,
A_{\mathcal X_1} \cap \, A_{\mathcal X_2}
= A_{\mathcal X_1 \cap\, \mathcal X_2}
=\, A_{\mathcal Y}.
$$
As a finitely generated subgroup of $K_{\mathcal Y}$ choose $L_{\mathcal Y} = F_3^{v_1 v_2}$ with just \textit{three} generators $a^{v_1 v_2}\!,$ $b^{v_1 v_2}\!,$ $ c^{v_1 v_2}$. As to explicit presentation of $K_{\mathcal Y}$, we may use \eqref{EQ star construction short form} and \eqref{EQ initial form of star construction} to write:
\begin{equation}
\label{EQ relations K_Y}
\begin{split}
K_{\mathcal Y}
& \!=\! 
\big\langle
\,Z_{1}, Z_{2}, v_1, v_2  
\mathrel{\,|\,}  
a(Z_{1})\!=\! a(Z_{2}),\, 
b(Z_{1})\!=\! b(Z_{2}),\,
c(Z_{1})\!=\! c(Z_{2});\\
& \hskip35mm 
\text{$v_1$ fixes the generators of $L_{\mathcal X_1}$}; \\
& \hskip35mm 
\text{$v_2$ fixes the generators of $L_{\mathcal X_2}$} 
\big\rangle
\end{split}
\end{equation}
where $a(Z_{1})$ is the copy of $a$ written as a word on $Z_{1}$, $a(Z_{2})$ is the copy of $a$ written as a word on $Z_{2}$, etc.; 
we made them equal to guarantee $K_{\mathcal X_1} \!\cap K_{\mathcal X_2} = F_3$.

Notice that in \eqref{EQ relations K_Y} we did not include any relations identifying two copies of $M$ 
(compare to \eqref{EQ initial form of star construction}) because here $M=F_3$, and we had already identified the copies of $F_3$ in both groups in the first row of \eqref{EQ relations K_Y}.

\begin{figure}[h]
\includegraphics[width=390px]{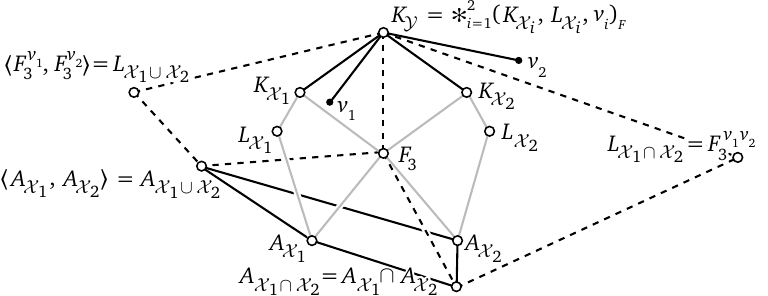}
\caption{Construction of $K_{\mathcal Y}$ for $\mathcal Y = \mathcal X_1 \cap \mathcal X_2$ and for $\mathcal Y = \mathcal X_1 \cup \mathcal X_2$.}
\label{Figure_12_Un_Intersect}
\end{figure}

The case  $\mathcal Y = \upsilon(\mathcal X_1, \mathcal X_2)=\mathcal X_1 \cup \mathcal X_2$ is analogous to the proof above, with one difference: we have $A_{\mathcal Y} =
A_{\mathcal X_1 \cup\, \mathcal X_2}
=\langle\,  A_{\mathcal X_1} ,\; A_{\mathcal X_2}\rangle
$. Using Lemma~\ref{LE intersection in bigger group multi-dimensional} and 
Lemma~\ref{LE join in HNN extension multi-dimensional}
we have: 
$$
F_3 \cap 
\langle\,  F_3^{v_1} \!,\; F_3^{v_2}\rangle
=\langle\,  A_1,\; A_2\rangle
=  \big\langle F_3 \cap L_{\mathcal X_1} ,\;
F_3 \cap L_{\mathcal X_2}\big\rangle 
\,=\,
\langle\,  A_{\mathcal X_1} ,\; A_{\mathcal X_2}\rangle
= A_{\mathcal X_1 \cup\, \mathcal X_2}
= \,A_{\mathcal Y}.
$$
This time we choose $L_{\mathcal Y} = \langle\,  F_3^{v_1} \!,\; F_3^{v_2}\rangle$ with just \textit{six} generators 
$a^{v_1}\!,$ $b^{v_1}\!,$ $ c^{v_1}\!,$
$a^{v_2}\!,$ $b^{v_2}\!,$ $ c^{v_2}$. 
As $K_{\mathcal Y}$ we take the same group $\mathscr{A}$ used above admitting explicit presentation \eqref{EQ relations K_Y}.

\subsection{Some auxiliary adaptations} 
\label{SU Some auxiliary adaptations}

Before we proceed to the remaining Higman operations we need some adaptation in notation and construction. 

\subsubsection{Adding $a,b,c$ to the generators $Z$}
\label{RE abc can be added} 

Assume the hypothesis of Theorem~\ref{TH Theorem A} holds for $\mathcal X$:
the group $K_{\mathcal X}=\langle
\, Z \;|\;  S \rangle$ with its subgroup $L_{\mathcal X}\le K_{\mathcal X}$ are given explicitly, and the embedding 
of $F_3=\langle a,b,c \rangle$ into $K_{\mathcal X}$ is explicitly known. 
Since this embedding is explicit, it is possible to write the free generators $a,b,c$ as certain words $a=a(Z),\, b=b(Z),\, c=c(Z)$ on the alphabet $Z$.
In many cases, such as the proofs in \cite{On explicit embeddings of Q}, $K_{\mathcal X}$ already is constructed so that $Z$ contains the letters $a,b,c$.
But even if $K_{\mathcal X}$ is given by some other generators \textit{not} involving $a,b,c$, we can apply  Tietze transformations: add the words $a=a(Z),\, b=b(Z),\,  c=c(Z)$ to the defining relations $S$, and add the letters $a,b,c$ to the generators $Z$. Hence we can always assume the generators of $F_3$ are included among the letters in $Z$. This is going to simplify the notation below.
%

\subsubsection{An auxiliary copy $\bar{\mathscr{A}}$ of $\mathscr{A}$}
\label{NOR SU Auxiliary copy of A built here}  

In analogy to the generating set $X_{\!\mathscr{A}}$ given in \eqref{EQ generators of XA} introduce a new generating set:  
\begin{equation}
\label{EQ generators of X bar A}
X_{\! \bar {\mathscr{A}}}=\big\{
\bar a,  \bar b,  \bar c,  \bar t_1,  \bar t'_1, \bar  u_1,  \bar u_2,  \bar d, \bar e
\big\}
\end{equation}
to construct a copy $\bar {\mathscr{A}}$ of the group $\mathscr{A}$ applying the same procedure as in Section~\ref{SU Construction of the group A}.
This group has the relations $R_{\! \bar {\mathscr{A}}}$ obtained from the relations $R_{\! {\mathscr{A}}}$ of \eqref{EQ relations A} by just appending bars on each letter, such as 
$\bar b^{\bar t_1}\!=\! \bar b$,\,
$\bar b^{\bar t'_1}\!=\! \bar b^{\bar c^{\; -1}}$\!\!\!,\,  etc.
In particular, inside $\bar {\mathscr{A}}$ the sub\-group
$\bar F \!=\! \langle \bar a, \bar c, \bar c\rangle$ is a free group of rank $3$. 

Next, in addition to the elements $b_i, b_f, a_f \!\in\! F_3$,\,
$d_i, d_f \!\in \!{\mathscr{A}}$ we  introduce the elements   
$\bar b_i, \bar b_f, \bar a_f\! \in\! \bar F_3$, \,
$\bar d_i, \bar d_f \!\in\! \bar {\mathscr{A}}$ 
expectedly defined as $\bar b_i\! = \bar b^{\bar c^i}$\!\!,\;
$\bar b_f\! 
= \cdots
\bar b_{-1}^{f(-1)}
\bar b_{0}^{f(0)} 
\bar b_{1}^{f(1)}\cdots$,
\,
$\bar a_f\!=\bar a^{\bar b_f}$;
\;
$\bar d_i = \bar d^{\bar e^i}$\!\!,\;\;
$\bar d_f 
= \cdots
\bar d_{-1}^{f(-1)}
\bar d_{0}^{f(0)} 
\bar d_{1}^{f(1)}\cdots$,
compare to Section~\ref{SU Defining subgroups by integer sequences}.

\subsubsection{Construction of the direct product $K_P=\bar {\mathcal K} \times \mathscr{A}$}
\label{SU Construction of the direct product K x A}  

If for the given $\mathcal X\subseteq \E$ the subgroup $A_{\mathcal X}$ is benign in $F_3$, then $\bar A_{\mathcal X}
= \langle \bar a_f \;|\; f\in \mathcal X\rangle$ clearly is benign in $\bar F_3$. In case the overgroup $K_{\mathcal X}=\langle
\, Z
\;|\;
S
\rangle$ and its subgroup $L_{\mathcal X}\le K_{\mathcal X}$ can explicitly be constructed for $A_{\mathcal X}$, the respective $\bar K_{\mathcal X}=\langle
\,\bar Z
\;|\;
\bar S
\rangle$ and  $\bar L_{\mathcal X}\le \bar K_{\mathcal X}$ can explicitly be built for $\bar A_{\mathcal X}$.

\medskip
Since $\mathscr{A}$ was built by adjoining some new letters $t_1, t'_1, u_1, u_2, d,e$ to $F_3\!=\langle a,b,c \rangle$, we may suppose none of these new letters was involved in construction of $K_{\mathcal X}$. Since $K_{\mathcal X}$ by construction contains $a,b,c$, compare to Point~\ref{RE abc can be added}, we can assume its intersection with  $\mathscr{A}$ is $F$ precisely, and so it is legal to define the group $\mathcal K = K_{\mathcal X} *_F \mathscr{A}$ in which: 
$$\mathscr{A} \cap L_{\mathcal X} 
= 
\big(
\mathscr{A} \cap L_{\mathcal X}
\big) \cap F
= 
\mathscr{A} \cap \big(L_{\mathcal X}
\cap F \big) 
= \mathscr{A} \cap A_{\mathcal X} \subseteq A_{\mathcal X}.$$
On the other hand, $A_{\mathcal X}$ is inside both $\mathscr{A}$ and $L_{\mathcal X}$, and so
$\mathscr{A} \cap L_{\mathcal X} = A_{\mathcal X}$. That is, $A_{\mathcal X}$ is also benign in the \textit{larger} group $\mathscr{A}$ for the finitely presented overgroup $\mathcal K$ and for the same finitely generated $L_{\mathcal X}$ mentioned above.  

Using the copy $\bar{\mathscr{A}}$ of $\mathscr{A}$ from  Point~\ref{NOR SU Auxiliary copy of A built here}, and modifying the steps above for the generators \eqref{EQ generators of X bar A}, we get the copies 
$\bar F,\,
\bar A_{\mathcal X},
\bar {\mathcal K}, 
\bar L_{\mathcal X}$ of the groups $F,\,
A_{\mathcal X},
{\mathcal K}, 
L_{\mathcal X}$, so that $\bar A_{\mathcal X}$ is also benign in $\bar{\mathscr{A}}$ for 
$\bar {\mathcal K}$
and 
$\bar L_{\mathcal X}$.

\begin{figure}[h]
\includegraphics[width=390px]{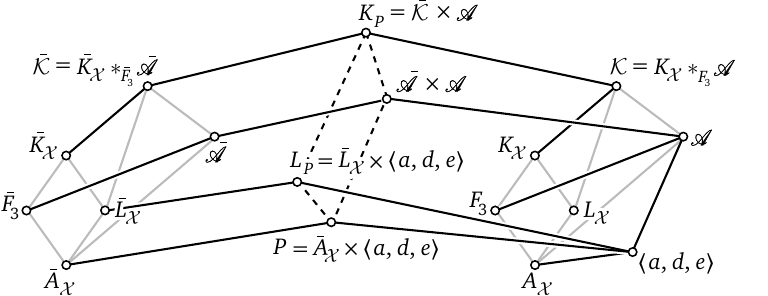}
\caption{Construction of $K_P=\bar {\mathcal K} \times \mathscr{A}$.}
\label{Figure_13_KP} 
\end{figure}

By Remark~\ref{RE finite generated is benign} 
the free subgroup 
$\langle a, d, e \rangle$ is benign in $\mathscr{A}$
for the finitely presented $\mathscr{A}$ and for the finitely generated $\langle a, d, e \rangle$.
Hence, the direct product: 
$$
P=\bar A_{\mathcal X} \times \langle a, d, e \rangle
$$ 
is benign in $\bar {\mathscr{A}} \times \mathscr{A}$ for the finitely presented overgroup
$K_P=\bar {\mathcal K} \times \mathscr{A}$ and for the finitely generated subgroup $L_P=\bar L_{\!\mathcal X} \!\times \langle a, d, e \rangle \le K_P$, see Figure~\ref{Figure_13_KP}.
\medskip 

For each $f\!\in\! \mathcal X$ we by Lemma~\ref{LE action of d_m on f} evidently have $a_f\! = a^{b_f}\! = a^{d_f}$\!, this simple fact can be explained for, say, $f=(2,5,3)$:
\begin{equation} 
\label{EQ a^d_f = a^d_f}
a^{d_f}\!
=a^{d_0^{2}\,d_1^5\,d_2^{3}}\!
=(a^{d_0^{2}})^{\,d_1^5\,d_2^{3}}\!
=(a^{b_0^{2}})^{\,d_1^5\,d_2^{3}}\!
=(a^{b_0^{2}\,b_1^5})^{\,d_2^{3}}\!
=a^{b_0^{2}\,b_1^5\,b_2^{3}}\!
=a^{b_f}\! =a_f.
\end{equation}
Hence,  $A_{\mathcal X} \subseteq \langle a, d, e \rangle$, and similarly,  
$\bar A_{\mathcal X} = \langle \bar a^{\bar b_f} \;|\;  f\in \mathcal X\rangle
= \langle \bar a^{\bar d_f} \;|\;  f\in \mathcal X\rangle\subseteq \langle \bar a, \bar d, \bar e \rangle$, i.e., the above product $P$ certainly is inside $\langle \bar a, \bar d, \bar e \rangle \times \langle a, d, e \rangle$ also.  
We are going to use this fact in the proofs below often.

\subsection{The proof for the operation $\rho$}
\label{SU The proof for rho} 

The case with operation $\rho$ was recently covered in \cite{Higman's reversing operation}, but we include it in points \ref{SU Obtaining the benign subgroup for rho Q}\,--\,\ref{SU Writing K rho X  by its generators and defining relations} below to have complete proofs for \textit{all} operations \eqref{EQ Higman operations} here. 
Denote $\mathcal Y = \rho \mathcal X$, say, for $f=(2,5,3)\in \mathcal X$ the function $\rho f$ sends $-2, -1, 0$ respectively to $3,5,2$, and all other integers $i$ to $0$; notice that one \textit{cannot} write $\rho f$ as $(3,5,2)$.
Also in analogy with the copy $\bar{\mathscr{A}}$ for $\mathscr{A}$ in Point~\ref{NOR SU Auxiliary copy of A built here}, we may pick a copy $\bar K_{{\mathcal X}}=\langle
\,\bar Z
\;|\;
\bar S
\rangle$ of $K_{\mathcal X}=\langle
\,Z
\;|\;
S
\rangle$ on some new generators $\bar Z$.
%

\subsubsection{Obtaining the benign subgroup $Q$ for $\rho$}
\label{SU Obtaining the benign subgroup for rho Q}

For each function $f\!\in \E$ define in ${\mathscr{A}}$ the couple of auxiliary products:
\begin{equation*}
\begin{split}
d_{\rho f} = \cdots
d_{1}^{f(-1)}
d_{0}^{f(0)} 
d_{-1}^{f(1)}\cdots 
\quad\;
{\rm and}
\quad\quad
\tilde d_{\rho f} = \cdots
d_{-1}^{f(1)}
d_{0}^{f(0)} 
d_{1}^{f(-1)}\cdots
\end{split}
\end{equation*}
where $\tilde d_{\rho f}$ differs from $d_{\rho f}$ by \textit{reverse order} of its factors $d_{i}^{f(-i)}$ only. For example, 
for $f=(2,5,3)$ we have 
$d_{\rho f}
=  d_{-2}^{3} d_{-1}^{5} d_{0}^{2}$
and 
$\tilde d_{\rho f}
=  d_{0}^{2} d_{-1}^{5} d_{-2}^{3}$; compare these with the element $\bar d_{f}
=  \bar d_{0}^{2} \bar d_{1}^{5} \bar d_{2}^{3}$ used above.  

In the direct product 
$\bar{\mathscr{A}} \times {\mathscr{A}}$ choose the pairs $\lambda_f=\big(\bar d_f,\; \tilde d_{\rho f} \! \big)$. 
The $3$-generator subgroup 
$T=\big\langle 
(\bar a,\; a),\;
(\bar d,\; d),\;
(\bar e,\; e^{-1})
\big\rangle$ of this direct product clearly contains such 
$\lambda_f$ for every $f\!\in \mathcal E$.
This uncomplicated fact requires routine calculations, which are easier to explain on a simple example for, say, $f=(2,5,3)$.
Clearly, $T$ contains the product
$$
(\bar e,\; e^{-1})^{-2} 
\cdot \,
(\bar d,\; d) 
\cdot \,
(\bar e,\; e^{-1})^2
=
\big(\bar d^{\,\bar e^{\;\,2}}
\!\!,\;
d^{e^{\,-2}}\big)
=
\big(\bar d_2
,\;
d_{\,-2}\big),
$$
together with the cube
$
\big(\bar d_2
,\;
d_{\,-2}\big)^3
=
\big(\bar d_2^3
,\;
d_{\,-2}^3\big)
$ of the latter. 
Similarly, the product
$$
\lambda_f
=
\big(\bar d_0^2
,\;
d_{0}^2\big)
\cdot
\big(\bar d_1^5
,\;
d_{-1}^5\big)
\cdot
\big(\bar d_2^3
,\;
d_{-2}^3\big)
=
\big( \bar d_{0}^{2} \bar d_{1}^{5} \bar d_{2}^{3},\;\; d_{0}^{2}d_{-1}^{5} d_{-2}^{3} \big)
=\big(\bar d_f,\; \tilde d_{\rho f} \! \big)
$$
is also in $T$. Trivially, $T$  contains the conjugates $(\bar a,\; a)^{\lambda_f}=
\big(\bar a^{\;\bar d_f},\; a^{\tilde d_{\rho f}}\!\big)$ for \textit{all} such $f$.

\medskip 

By Remark~\ref{RE finite generated is benign}\, $T$
is benign in $\bar{\mathscr{A}} \!\times\! {\mathscr{A}}$, and for it one can choose $K_{T} = \bar{\mathscr{A}} \!\times\! {\mathscr{A}}$ and $L_{T} = T$.

\begin{figure}[h]
\includegraphics[width=390px]{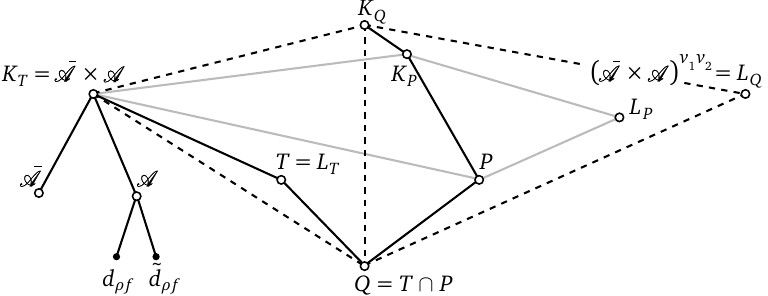}
\caption{Construction of $K_Q$.}
\label{Figure_14_KQ}
\end{figure}

Then according to Corollary~\ref{CO intersection and join are benign multi-dimensional}\;\eqref{PO 1 CO intersection and join are benign multi-dimensional} the intersection $Q=T \,\cap\, P$ of the above two benign subgroups is benign in $\bar{\mathscr{A}} \!\times\! {\mathscr{A}}$  for the finitely presented overgroup:
\begin{equation*}
\begin{split}
K_Q &
= \big(K_T *_{L_T} v_1\big) 
*_{\bar{\mathscr{A}} \!\times\! {\mathscr{A}}}
\big(K_P *_{L_P} v_2\big)\\
&= \Big(
\big(\bar{\mathscr{A}} \!\times\! {\mathscr{A}}\big) 
*_{T} 
v_1 
\Big)
\, *_{\bar{\mathscr{A}} \times {\mathscr{A}}}
\Big(
\big(\bar {\mathcal K} \!\times\! \mathscr{A}\big)
*_{\bar L_{\!\mathcal X} \times\, \langle a, d, e \rangle} 
v_2 
\Big)
\end{split}
\end{equation*}
with two new stable letters $v_1, v_2$, and for the $18$-generator subgroup $L_Q\!=\!\big(\!\bar{\mathscr{A}} \!\times\! {\mathscr{A}}\big)^{v_1v_2}$\! of $K_Q$, see Figure~\ref{Figure_14_KQ}.

\medskip  
$Q$ turns out to have simple structure. Namely, any couple from $Q$ is in $P=\bar A_{\mathcal X} \times \langle a, d, e \rangle$, and hence its first coordinate is an element in $\bar A_{\mathcal X}$  generated by some words $\bar a_f$ for certain $f\!\in \mathcal X$. For each of them using Lemma~\ref{LE action of d_m on f} we in analogy with \eqref{EQ a^d_f = a^d_f} have:
\begin{equation}
\label{EQ a_f = a^b_f = a^d_f}
\bar a_f 
= \bar a^{\bar b_f}
=\bar a^{\bar d_f}\!,
\end{equation}
i.e., that first coordinate can be rewritten as a word on $\bar a,\bar d, \bar e$.
On the other hand, our couple is inside the $3$-generator group $T=\big\langle 
(\bar a,\; a),\;
(\bar d,\; d),\;
(\bar e,\; e^{-1})
\big\rangle$, i.e., if its first coordinate is written as a word on 
$\bar a,\bar d, \bar e$, then the second coordinate can be obtained by replacing all  
$\bar a,\bar d, \bar e$ in that word by $a, d, e^{-1}$ respectively.
But that replacement simply transforms each $\bar a^{\bar d_f}$ to $a^{\tilde d_{\rho f}}$. The routine of this step is very easy to see on an example with, say, $f=(2,5,3)$ for which the first coordinate is:
\begin{equation*}
\begin{split}
\bar a_f &
= \bar a^{\bar b_f}
=\bar a^{\bar d_f}
= \bar a^{\bar d_{0}^{2} \bar d_{1}^{5} \bar d_{2}^{3}}
= \bar d_{2}^{\;-3}\bar d_{1}^{\;-5}  \bar d_{0}^{\;-2}\, \cdot \,
\bar a \,
\cdot\, \bar d_{0}^{2}\, \bar d_{1}^{5}\, \bar d_{2}^{3}\\
&= 
\bar e^{\;-2} \,\bar d^{\,-3} \,\bar e^{\;2} 
\cdot 
\bar e^{\;-1} \,\bar d^{\,-5} \,\bar e
\cdot 
\bar d^{-2}
\,\cdot\;\, 
\bar a 
\,\; \cdot \; \bar d^{2}
\cdot
\bar e^{\;-1} \,\bar d^{\,5} \,\bar e
\cdot 
\bar e^{\;-2} \,\bar d^{\,3} \,\bar e^{\;2}
\end{split}
\end{equation*}
with respect to which the second coordinate of the couple turns out to be: 
\begin{equation*}
\begin{split}
&
e^{\;2} \, d^{\,-3} \, e^{\;-2} 
\cdot \,
e \, d^{\,-5} \, e^{\;-1}
\cdot \,
d^{-2}
\cdot\,\,
a \,\,
\cdot \, d^{2}
\cdot\,
e \, d^{\,5} \, e^{\;-1}
\cdot \,
e \, d^{\,3} \, e^{\;-2}\\
& = 
d_{-2}^{\;-3} d_{-1}^{\;-5}   d_{0}^{\;-2}\, \cdot\,  \,
a \,
\, \cdot\,  d_{0}^{2}\,  d_{-1}^{5}\,  d_{-2}^{3}
= a^{d_{0}^{2}\,  d_{-1}^{5}\,  d_{-2}^{3}}=a^{\tilde d_{\rho f}}.
\end{split}
\end{equation*}
This means that the benign subgroup $Q$ of $\bar{\mathscr{A}} \!\times\! {\mathscr{A}}$ actually is of a simple format:
$$
Q = \big\langle 
\big(\bar a^{\bar d_f},\; a^{\tilde d_{\rho f}}\!\big)
\;|\;  f\in \mathcal X
\big\rangle
= \big\langle 
(\bar a,\; a)^{\lambda_f}
\;|\;  f\in \mathcal X
\big\rangle.
$$
Using this with \eqref{EQ a_f = a^b_f = a^d_f} we see that $Q$ lies inside $\bar F_3 \!\times\! F_3$, and so $Q$ is benign in $\bar F_3 \!\times\! F_3$ also for the same choice of $K_Q$ and $L_Q$ made earlier.
%

\subsubsection{``Extracting'' $A_{\rho \mathcal X}$ from $Q$}
\label{SU Extracting A rho X from Q} 

Next we have to modify the constructed benign subgroup $Q$ by a few steps  ``to extract'' the benign subgroup 
$A_{\rho \mathcal X}=
\big\langle a^{b_{\rho f}}
\;|\;  f\in \mathcal X
\big\rangle
$ from it.

Comparing \eqref{EQ a^d_f = a^d_f} to Remark~\ref{RE order of d_i does not matter} we see that $a^{\tilde d_{\rho f}}=a^{b_{\rho f}}=a_{\rho f}$, that is, $A_{\rho \mathcal X}$ is nothing but the group generated by the second coordinates $a^{\tilde d_{\rho f}}$ of all pairs from $Q$.

$\bar F_3 \cong \bar F_3 \times \1$ is benign in 
$\bar F_3 \times F_3$ for the finitely presented $\bar F_3 \times F_3$ and for the finitely generated $\bar F_3 \times \1$, see Remark~\ref{RE finite generated is benign}.  
Hence, by Corollary~\ref{CO intersection and join are benign multi-dimensional}\;\eqref{PO 2 CO intersection and join are benign multi-dimensional} the join 
$
Q_1=\big\langle \bar F_3 \!\times\! \1 ,\, Q \big\rangle
=\bar F_3 \!\times\! \langle
a_{\rho f} \mathrel{|} 
f \in \mathcal X \rangle
$ is benign in $\bar F_3 \!\times\! F_3$ for the finitely presented overgroup
$$
K_{Q_1}= 
\big((\bar F_3 \!\times\! F_3)\,*_{\bar F \times \1} w_1\big) 
\,*_{\bar F \times F}
\big(K_{Q} *_{L_{Q}} w_2\big)
$$
with two further new stable letters $w_1, w_2$,\, and for its $12$-generator subgroup
$$L_{Q_1} = \big\langle(\bar F_3 \!\times\! F_3)^{w_1}\!,\;(\bar F_3 \!\times\! F_3)^{w_2} \big\rangle.$$

\vskip-3mm
\begin{figure}[h]
\includegraphics[width=390px]{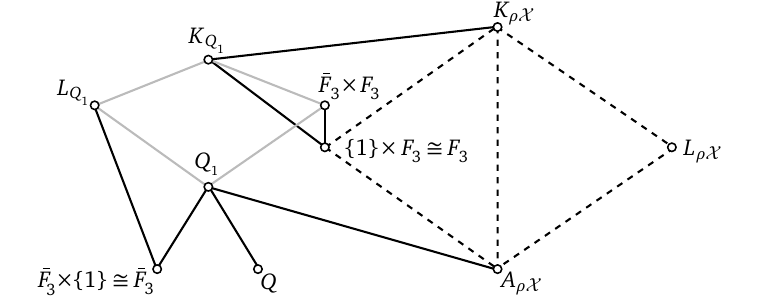}
\caption{``Extracting'' $Q_1$ and $A_{\rho \mathcal X}$ from $Q$.}
\label{Figure_15_Extracting_Q1_and_ArhoX}
\end{figure}

Finally,  
$F_3 = \1 \!\times\!  F_3 $ is benign in 
$\bar F_3 \!\times\! F_3$ for the finitely presented $\bar F_3 \!\times\! F_3$, and for the finitely ge\-ne\-ra\-ted $\1 \!\times\!  F_3$, see Figure~\ref{Figure_15_Extracting_Q1_and_ArhoX}.  
Hence, 
by Corollary~\ref{CO intersection and join are benign multi-dimensional}\;\eqref{PO 1 CO intersection and join are benign multi-dimensional}
the intersection:
$$
A_{\rho \mathcal X}=A_{\mathcal Y}=
\big(\1 \!\times\!  F_3 \big) \cap Q_1 
=
\langle
a_{\rho f} \mathrel{|} 
f \in \mathcal X \rangle
$$ 
is benign in $\bar F_3 \!\times\! F_3$ for the finitely presented overgroup:
$$
K_{\rho \mathcal X}
=
K_{\mathcal Y}
=
\big((\bar F_3 \!\times\! F_3)\,*_{\1 \times  F_3} w_3\big) 
\,*_{\bar F_3 \times F_3}
(K_{Q_1} *_{L_{Q_1}} w_4)
$$
with stable letters $w_3, w_4$,\, and for the $6$-generator subgroup
$L_{\rho \mathcal X}
=L_{\mathcal Y}
=(\bar F_3 \!\times\! F_3)^{w_3 w_4}$ inside $K_{\rho \mathcal X}$.
But since $A_{\rho \mathcal X}$ is in $F_3$, it is benign in $F_3$ also, for the same choice of $K_{\rho \mathcal X}$ and $L_{\rho \mathcal X}$ above.

\subsubsection{Writing $K_{\rho \mathcal X}$  explicitly}
\label{SU Writing K rho X  by its generators and defining relations}   

Now from sections \ref{SU Construction of the group A} and \ref{NOR SU Auxiliary copy of A built here}  we for $\mathscr{A}$ and  $\bar {\mathscr{A}}$ know their generating sets \eqref{EQ generators of XA} and \eqref{EQ generators of X bar A}, along with their  presentations 
$\mathscr{A}=\langle
\,X_{\! \mathscr{A}}
\;|\;
R_{\! \mathscr{A}}
\rangle$
and
$\bar {\mathscr{A}}=\langle
\,X_{\! \bar {\mathscr{A}}}
\;|\;
R_{\! \bar {\mathscr{A}}}
\rangle$. 
Also, the groups 
$K_{\mathcal X}=\langle
\,Z
\;|\;
S
\rangle$ and $\bar K_{{\mathcal X}}=\langle
\,\bar Z
\;|\;
\bar S
\rangle$ together with the finitely many generators of $L_{\!\mathcal X}$ and of $\bar L_{\!\mathcal X}$
are supposed to be explicitly known.
By Point~\ref{RE abc can be added} we may suppose  
$X_{\! \mathscr{A}} \cap\, Z = \{a,b,c\}$ and 
$X_{\! \mathscr{\bar A}} \cap\, \bar Z = \{\bar a,\bar b,\bar c\}$.

The amalgamated product $\mathcal K = K_{\mathcal X} *_F \mathscr{A}$
can be generated by the generators 
$Z \backslash \{a,b,c\}$ of $K_{\mathcal X}$, together with the generators $X_{\! {\mathscr{A}}}$ (we exclude $a,b,c$ from $Z$ because they already were included in $X_{\! {\mathscr{A}}}$, see Point~\ref{RE abc can be added}).
As defining relations for the group $\mathcal K$ we may take the union $S \cup R_{\! \mathscr{A}}$. We can similarly treat the copy $\bar {\mathcal K}$ of ${\mathcal K}$.

Then the direct product $K_P=\bar {\mathcal K} \times \mathscr{A}$
can be given by the above mentioned generators and relation, \textit{plus} the relations making sure $\bar {\mathcal K}$ commutes with $\mathscr{A}$:
\begin{equation} 
\label{EQ K_P presentation}
\begin{split}
K_P=
\Big\langle
X_{\! \mathscr{A}},\;  X_{\! \bar{\mathscr{A}}}\,;\;\;
\bar Z 
\,\backslash\, \{ \bar a, \bar b, \bar c \}\;
\;\mathrel{|} \;\; 
R_{\mathscr{\!A}};\;
R_{\mathscr{\!\bar A}};\;
\bar S;\\
&\hskip-63mm 
\text{generators $X_{\! \mathscr{A}}$ commute with  $X_{\! \bar{\mathscr{A}}}$}\,;\; \\
&\hskip-63mm 
\text{generators $X_{\! \mathscr{A}}$ commute with  $\bar Z 
\backslash \{ \bar a, \bar b, \bar c \}$}
\Big\rangle.
\end{split}
\end{equation}

Next, taking into account the way we constructed $K_{T}$, $K_{Q}$, $K_{Q_1}$, $K_{\rho \mathcal X}$ with the \textit{fixing} effect of our new letters $v_1, v_2;\; w_1, w_2, w_3, w_4$ on certain finitely generated subgroups, we produce:
\begin{equation}
\label{EQ K rho X}
\begin{split}
K_{\rho \mathcal X}=
\Big\langle
X_{\! \mathscr{A}},\;  X_{\! \bar{\mathscr{A}}};\;
\bar Z
\backslash \{ \bar a, \bar b, \bar c \};\;
v_1, v_2;\;  w_1, w_2, w_3, w_4   
\;\mathrel{|} \;\; 
R_{\mathscr{A}};\;
R_{\mathscr{\bar A}};\;
\bar S;\\
&\hskip-99mm 
\text{generators $X_{\! \mathscr{A}}$ commute with  $X_{\! \bar{\mathscr{A}}}$};\; \\
&\hskip-99mm 
\text{generators $X_{\! \mathscr{A}}$ commute with  $\bar Z 
\,\backslash\, \{ \bar a, \bar b, \bar c \}$};\; \\
&\hskip-99mm 
\text{$v_1$ fixes $\bar a a,\;
\bar d d,\;
\bar e e^{-1} $};
\;\;\;
\text{$v_2$ fixes $\bar L_{\!\mathcal X}$ and $a,d,e$};\; 
\\
&\hskip-99mm \text{$w_1$ fixes  
$\bar a,\bar b,\bar c$};
\;\;
\text{$w_2$ fixes  $X_{\!\! \mathscr{A}}^{v_1 v_2}$ and $X_{\!\! \bar{\mathscr{A}}}^{v_1 v_2}$};\;\;
\text{$w_3$ fixes $ a, b, c$;}
\\
&\hskip-99mm  
\text{$w_4$ fixes 
$\big\{
a,b,c, \bar a, \bar b, \bar c
\big\}^{\! w_1} \! \cup \big\{
a,b,c, \bar a, \bar b, \bar c
\big\}^{\! w_2}$}
\Big\rangle\,.
\end{split}
\end{equation}
Lastly, as $L_{\rho \mathcal X}$ we can take the $6$-generator subgroup
generated by 
$
\big\{
a,b,c, \bar a, \bar b, \bar c
\big\}^{w_3 w_4}$ in $K_{\rho \mathcal X}$ by Corollary~\ref{CO intersection and join are benign multi-dimensional}\;\eqref{PO 1 CO intersection and join are benign multi-dimensional}.
%
%
In \eqref{EQ K rho X}\; ``$w_2$ fixes  $X_{\!\! \mathscr{A}}^{v_1 v_2}$ and $X_{\!\! \bar{\mathscr{A}}}^{v_1 v_2}$'' means that conjugation by $w_2$ fixes 
the conjugates of \textit{each} of the generators $X_{\!\! \mathscr{A}}$ and $X_{\!\! \bar{\mathscr{A}}}$
from \eqref{EQ generators of XA} or \eqref{EQ generators of X bar A} by the product $v_1 v_2$.

If $K_{\mathcal X}$ has $m$ generators (which we may assume include $a,b,c$) and $n$ defining relations, and if 
$L_{\mathcal X}$ has $k$ generators, then the group $K_{\rho \mathcal X}$ in \eqref{EQ K rho X} has 
$9 + 9 + (m - 3) + 2 +4 = m +21$ generators and 
$20+20+n+9\cdot 9 + 9\cdot (m-3)
+3 + k +3 + 3  + 2\cdot 9 +3 +2\cdot 6 
= n+9m+k+136$ 
defining relations.

\medskip
In the above constructions we have supposed that the overgroup $K_{\mathcal X}$ of $F_3$ has $a,b,c$ among its generators by Point~\ref{RE abc can be added}. Also observe a formatting issue in \eqref{EQ K rho X}: we write not 
``$v_1$ fixes
$(\bar a,\; a)$,
$(\bar d,\; d)$,
$(\bar e,\; e^{-1})$''
but 
``$v_1$ fixes
$\bar a a$,
$\bar d d$,
$\bar e e^{-1}$''
which has the same meaning  as all generators in 
$X_{\! \mathscr{A}}$ already commute with those in $X_{\! \bar{\mathscr{A}}}$ according to the second line of \eqref{EQ K rho X}.

\begin{Remark}
\label{RE Compare this constructions for rho with Higmans brief remark}
Compare the presentation \eqref{EQ K rho X} with Higman's very brief first paragraph in the proof of Lemma 4.6 on page 470 in \cite{Higman Subgroups of fP groups}. That paragraph only states that if $A_{\mathcal X}$ is benign, then $A_{\rho \mathcal X}$ is benign because there is an automorphism $\alpha$ sending $d_i$ to $d_{\,-i}$.
Explicit construction of the respective $K_{\rho \mathcal X}$ and $L_{\rho \mathcal X}$ is a non-trivial routine task, as we saw above.
\end{Remark}

Our figures~\ref{Figure_12_Un_Intersect}\,--\,\ref{Figure_15_Extracting_Q1_and_ArhoX} were to illustrate the Higman operations $\iota,\, \upsilon,\, \rho$. Since the general graphical pattern is understandable, we are \textit{not} going to illustrate the remaining operations $\sigma,\, 
\tau,\, 
\theta,\, 
\zeta,\, 
\pi,$\, 
$\omega_m$ 
from \eqref{EQ Higman operations}.
We will only include more figures in Chapter~\ref{SE Theorem B and the final embedding} to illustrate the final stages of the embedding, including the \textit{Higman Rope Trick} in Figure~\ref{Figure_18_Higman_rope_trick}.

\subsection{The proof for the operation $\sigma$}
\label{SU The proof for sigma} 

Assume $A_{\mathcal X}$ is benign in $F_3$ for the explicitly given finitely presented group $K_{\mathcal X}=\langle
\, Z
\;|\;
S
\rangle$ and for its finitely generated subgroup $L_{\mathcal X}$. 
Then the copy $\bar K_{{\mathcal X}}=\langle
\,\bar Z
\;|\;
\bar S
\rangle$ of $K_{\mathcal X}$ can be defined on some new generators $\bar Z$.
Denoting $\mathcal Y = \sigma \mathcal X$ we, say, for 
$f=(2,5,3)\in \mathcal X$ have 
$\sigma f = (0, 2,5,3)\in \mathcal Y$ and $a_{\sigma f}=
a^{b_1^2 b_2^5 b_3^3}$.

\begin{Remark}
\label{RE same letter for all operations}
For the operation $\sigma$ in the current section (and for the rest of Higman operations in below sections 
\ref{SU The proof for zeta},
\ref{SU The proof for pi},
\ref{SU The proof for theta},
\ref{SU The proof for tau},
\ref{SU The proof for omega_m})
we are going to apply constructions considerably distinct from the method used for $\rho$ in Section~\ref{SU The proof for rho} above. However, these constructions share some similarities which we  wish to stress advisedly. Hence, where reasonable, we will use similar notation, such as 
$\lambda_f$,\! 
$T$,
$Q$,\,
$Q_1$,\! 
$P$,
$\mathcal K$, etc., to denote elements and groups with similar purposes in technical steps. Compare, for example, the finitely generated groups $T$ used in points~\ref{SU Obtaining the benign subgroup for rho Q}, 
~\ref{SU Obtaining the benign subgroup for sigma Q SHORT},
~\ref{SU Obtaining the benign subgroup for theta Q},
~\ref{SU Obtaining the benign subgroup for tau Q}, etc., in very similar tricks.
\end{Remark}

\subsubsection{Obtaining the benign subgroup $Q$ for $\sigma$}
\label{SU Obtaining the benign subgroup for sigma Q SHORT}
For each $f\!\in \E$ we in $F_3$ have:
$$
b_{\sigma f} 
= \cdots
b_{-1+1}^{f(-1)}
b_{0+1}^{f(0)} 
b_{1+1}^{f(1)}\cdots
\;\; = \;\;
\cdots
b_{0}^{f(-1)}
b_{1}^{f(0)} 
b_{2}^{f(1)}\cdots
$$
Say, 
for $f=(2,5,3)$ we have 
$b_{\sigma f}
=  b_{1}^{2} b_{2}^{5} b_{3}^{3}=  b_{0}^{0} b_{1}^{2} b_{2}^{5} b_{3}^{3}$.
In 
$\bar{F}_3 \!\times\! F_3$ define the pairs $\lambda_f=\big(\bar b_f,\; b_{\sigma f}\! \big)$ for $f\!\in \E$. 
Then the $3$-generator subgroup 
$T=\big\langle 
(\bar a,\; a),\;
(\bar b,\; b^{c}),\;
(\bar c,\; c)
\big\rangle$ clearly contains  
$\lambda_f$ for every $f\!\in \mathcal E$.
Indeed, say, for $f=(2,5,3)$ the group $T$ contains the conjugate:
$$
(\bar c,\; c)^{-2} 
\cdot \,
(\bar b,\; b^{c}) 
\cdot \,
(\bar c,\; c)^2
=
\big(\bar b^{\,\bar c^{\;\,2}}
\!\!,\;
b^{c^{1+2}}\big)
=
\big(\bar b_2
,\;
b_{3}\big),
$$
together with its cube 
$
\big(\bar b_2^3
,\;
b_{3}^3\big)
$. 
For similar reasons the product
$$
\lambda_f
=
\big(\bar b_0^2
,\;
b_{1}^2\big)
\cdot
\big(\bar b_1^5
,\;
b_{2}^5\big)
\cdot
\big(\bar b_2^3
,\;
b_{3}^3\big) 
=
\big( \bar b_{0}^{2} \bar b_{1}^{5} \bar b_{2}^{3},\;\; 
b_{1}^{2}b_{2}^{5} b_{3}^{3} \big)
=\big(\bar b_f,\; b_{\sigma f} \! \big)
$$
is also in $T$. Hence, $T$  contains all the conjugates $(\bar a,\; a)^{\lambda_f}
=
\big(\bar a^{\bar b_f},\; a^{b_{\sigma f}}\!\big)
=(\bar a_f,\; a_{\sigma f})$,\;  $f\in \E$.

\medskip
Clearly, $T$
is benign in $\bar F_3 \!\times\! F_3$, and for it we can choose $K_{T} = \bar F_3 \!\times\! F_3$ and $L_{T} = T$.
Since the subgroup 
$F_3$ evidently is benign in $F_3$, the direct product: 
$$
P=\bar A_{\mathcal X} \times F_3
$$ 
is benign in $\bar F_3 \!\times\! F_3$ for the finitely presented overgroup
$K_{P}=\bar K_{\mathcal X} \times F_3$ and for the finitely generated subgroup $L_{P}=\bar L_{\!\mathcal X} \!\times F_3 \;\le\; K_{P}$.

\medskip   
By Corollary~\ref{CO intersection and join are benign multi-dimensional}\;\eqref{PO 1 CO intersection and join are benign multi-dimensional} the intersection $Q=T \,\cap\, P$ of the above benign subgroups 
is benign in $\bar F_3 \!\times\! F_3$  for the finitely presented overgroup:
\begin{equation*}
\begin{split}
K_Q &
= \big(K_T *_{L_T} v_1\big) 
*_{\bar F_3 \times F_3}
\big(K_{P} *_{L_{P}} v_2\big)\\
&= \Big(
\big(\bar F_3 \!\times\! F_3\big) 
*_{T} 
v_1 
\Big)
\, *_{\bar F_3 \times F_3}
\Big(
\big(\bar K_{\mathcal X_3} \!\times\! F\big)
*_{\bar L_{\!\mathcal X} \times\, F_3} 
v_2 
\Big)
\end{split}
\end{equation*}
with two new stable letters $v_1, v_2$, and for the $6$-generator subgroup $L_Q\!=\big(\bar F_3 \!\times\! F_3\big)^{v_1v_2}$\!.

\medskip
Let us reveal the simple structure of $Q$. Namely, any couple from $Q$ is in $P=\bar A_{\mathcal X} \times F_3$, and hence, its first coordinate is an element in $\bar A_{\mathcal X}$  generated by some $\bar a_f=\bar a^{\bar b_f}$ for certain $f\!\in \mathcal X$. I.e., that first coordinate can be written as a word on $\bar a,\bar b, \bar c$.
On the other hand, our couple is in $T=\big\langle 
(\bar a,\; a),\;
(\bar b,\; b^{c}),\;
(\bar c,\; c)
\big\rangle$, i.e., if its first coordinate is  a word on 
$\bar a,\bar b, \bar c$, then the second coordinate can be obtained by replacing  in that word each of 
$\bar a,\bar b, \bar c$ by $a, b^c\!,\, c$ respectively.
But that just transforms $\bar a^{\bar b_f}$ to $a^{b_{\sigma f}}$ because, say, for $f=(2,5,3)$ the first coordinate is: 
\begin{equation*}
\begin{split}
\bar a_f &= 
\bar c^{\;-2} \,\bar b^{\,-3} \,\bar c^{\;2} 
\cdot 
\bar c^{\;-1} \,\bar b^{\,-5} \,\bar c
\cdot 
\bar b^{-2}
\,\cdot\;\, 
\bar a 
\,\; \cdot \; \bar b^{2}
\cdot
\bar c^{\;-1} \,\bar b^{\,5} \,\bar c
\cdot 
\bar c^{\;-2} \,\bar b^{\,3} \,\bar c^{\;2},
\end{split}
\end{equation*}
and the respective  second coordinate then has to be: 
\begin{equation*}
\begin{split}
&
c^{-2} \, (b^c)^{\,-3} \, c^{2} 
\cdot \,
c^{-1} \, (b^c)^{\,-5} \, c
\cdot \,
(b^c)^{-2}
\cdot\,\,
a \,
\cdot \, (b^c)^{\,2}
\cdot\,
c^{-1} \, (b^c)^{\,5} \, c
\cdot \,
c^{-2} \, (b^c)^{\,3} \, c\\
& = 
b_{2+1}^{\;-3} b_{1+1}^{\;-5}   b_{0+1}^{\;-2}\, \cdot\,  \,
a 
\, \cdot\,  b_{0+1}^{2}\,  b_{1+1}^{5}\,  b_{2+1}^{3}
= a^{b_{0+1}^{2} b_{1+1}^{5} b_{2+1}^{3}}
= a^{b_{1}^{2} b_{2}^{5} b_{3}^{3}}
=a^{b_{\,\sigma f}}\!.
\end{split}
\end{equation*}
Thus the benign subgroup $Q$ of $\bar F_3 \!\times\! F_3$, in fact, is:
$$
Q 
= \big\langle 
\big(\bar a_{f},\; a_{\sigma f}\!\big)
\;|\;  f\in \mathcal X
\big\rangle
= \big\langle 
\big(\bar a^{\bar b_f},\; a^{b_{\sigma f}}\!\big)
\;|\;  f\in \mathcal X
\big\rangle
= \big\langle 
(\bar a,\; a)^{\lambda_f}
\;|\;  f\in \mathcal X
\big\rangle.
$$

\subsubsection{``Extracting'' $A_{\sigma \mathcal X}$ from $Q$}
\label{SU Extracting A sigma X from Q SHORT}

Our next objective is to modify the obtained benign subgroup $Q$ via a few steps  ``to extract'' the benign subgroup 
$A_{\sigma \mathcal X}=
\big\langle a_{\sigma f}
\;|\;  f\in \mathcal X
\big\rangle
$ from it.

From the previous section we see that $A_{\sigma \mathcal X}$ is nothing but the group generated by the second coordinates $a^{b_{\sigma f}}=a_{\sigma f}$ of pairs from $Q$.

$\bar F_3 \cong \bar F_3 \times \1$ is benign in 
$\bar F_3 \times F_3$ for the finitely presented $\bar F_3 \times F_3$ and finitely generated $\bar F_3 \times \1$, see Remark~\ref{RE finite generated is benign}.  
Hence, by Corollary~\ref{CO intersection and join are benign multi-dimensional}\;\eqref{PO 2 CO intersection and join are benign multi-dimensional} the join 
$
Q_1=\big\langle \bar F_3 \!\times\! \1 ,\, Q \big\rangle
=\bar F_3 \!\times\! \langle
a_{\sigma f} \mathrel{|} 
f \in \mathcal X \rangle
$ is benign in $\bar F_3 \!\times\! F_3$ for the finitely presented:
$$
K_{Q_1}= 
\big((\bar F_3 \!\times\! F_3)\,*_{\bar F_3 \times \1} w_1\big) 
\,*_{\bar F_3 \times F_3}
\big(K_{Q} *_{L_{Q}} w_2\big)
$$
with two further new letters $w_1, w_2$,\, and for its $12$-generator subgroup:
$$
L_{Q_1} = \big\langle(\bar F_3 \!\times\! F_3)^{w_1}\!,\;(\bar F_3 \!\times\! F_3)^{w_2} \big\rangle.
$$
$F_3 = \1 \!\times\!  F_3$ is benign in 
$\bar F_3 \!\times\! F_3$ for the finitely presented $\bar F_3 \!\times\! F_3$ and finitely generated $\1 \!\times\!  F_3$.  
Hence, 
by Corollary~\ref{CO intersection and join are benign multi-dimensional}\;\eqref{PO 1 CO intersection and join are benign multi-dimensional}
the intersection:
$$
A_{\sigma \mathcal X}=A_{\mathcal Y}=
\big(\1 \!\times\!  F_3 \big) \cap Q_1 
=
\langle
a_{\sigma f} \mathrel{|} 
f \in \mathcal X \rangle
$$ 
is benign in $\bar F_3 \!\times\! F_3$ for the finitely presented group:
$$
K_{\sigma \mathcal X}
=
K_{\mathcal Y}
=
\big((\bar F_3 \!\times\! F_3)\,*_{\1 \times  F_3} w_3\big) 
\,*_{\bar F_3 \times F_3}
(K_{Q_1} *_{L_{Q_1}} w_4)
$$
with stable letters $w_3, w_4$,\, and for the $6$-generator subgroup
$L_{\sigma \mathcal X}
=L_{\mathcal Y}
=(\bar F_3 \!\times\! F_3)^{w_3 w_4}$.
But since $A_{\sigma \mathcal X}$ is inside $F_3$, it is benign in $F_3$ also, for the same choice of $K_{\sigma \mathcal X}$ and $L_{\sigma \mathcal X}$ above.

\subsubsection{Writing $K_{\sigma \mathcal X}$ explicitly}
\label{SU Writing K sigma X  by its generators and defining relations SHORT}   

Recall that the groups
$K_{\mathcal X}=\langle
\,Z
\;|\;
S
\rangle$ and $\bar K_{{\mathcal X}}=\langle
\,\bar Z
\;|\;
\bar S
\rangle$ together with the finitely many generators of $L_{\!\mathcal X}$ and of $\bar L_{\!\mathcal X}$
are explicitly given.
By  Point~\ref{RE abc can be added} we may suppose $Z$ includes
$a,b,c$, and $\bar Z$ includes
$\bar a,\bar b,\bar c$.

The finitely presented overgroup 
$K_{P}=\bar K_{\mathcal X} \times F_3$
of $P=\bar A_{\mathcal X} \times F_3$
is given by the relations of $\bar K_{\mathcal X}$, \textit{plus} the relations making sure $\bar K_{\mathcal X}$ commutes with $F_3$:
\begin{equation} 
\label{EQ K_P' presentation SHORT}
\begin{split}
K_P=
\Big\langle
a, b, c;\;  
\bar a, \bar b, \bar c\,;\;\;
\bar Z 
\,\backslash\, \{ \bar a, \bar b, \bar c \}\;
\;\mathrel{|} \;\; 
\bar S;\\
&\hskip-52mm 
\text{$a, b, c$ commute with $\bar a, \bar b, \bar c$ and $\bar Z 
\,\backslash\, \{ \bar a, \bar b, \bar c \}$}
\Big\rangle
\end{split}
\end{equation}
(in the first and second lines we exclude $\bar a, \bar b, \bar c$ from $\bar Z$ because of Point~\ref{RE abc can be added}).

Taking into account the way we constructed $K_{T}$, $K_{Q}$, $K_{Q_1}$, $K_{\sigma \mathcal X}$ with \textit{fixing} effect of the new letters $v_1, v_2;\; w_1, w_2, w_3, w_4$ (on respective finitely generated subgroups) we have:
\begin{equation}
\label{EQ K sigma X SHORT}
\begin{split}
K_{\sigma \mathcal X}=
\Big\langle
a, b, c;\;  
\bar a, \bar b, \bar c; \; 
\bar Z 
\backslash \{ \bar a, \bar b, \bar c \};\;
v_1, v_2;\;  w_1, w_2, w_3, w_4 
\;\;\mathrel{|} \;\; 
\bar S;\\
&\hskip-99mm 
\text{$a, b, c$ commute with $\bar a, \bar b, \bar c$ and with $\bar Z 
\,\backslash\, \{ \bar a, \bar b, \bar c \}$}\,;\; \\
&\hskip-99mm 
\text{$v_1$ fixes $\bar a a,\;
\bar b b^c,\;
\bar c c $};
\;\;\;
\text{$v_2$ fixes $\bar L_{\!\mathcal X}$ and $a,b,c$};\; 
\\
&\hskip-99mm \text{$w_1$ fixes  
$\bar a,\bar b,\bar c$};
\;
\text{$w_2$ fixes  $\big\{
a,b,c, \bar a, \bar b, \bar c
\big\}^{\!v_1 v_2}$}\!;\;
\text{$w_3$ fixes $ a, b, c$;}
\\
&\hskip-99mm  
\text{$w_4$ fixes 
$\big\{
a,b,c, \bar a, \bar b, \bar c
\big\}^{\! w_1} \! \cup \big\{
a,b,c, \bar a, \bar b, \bar c
\big\}^{\! w_2}$}
\Big\rangle\,.
\end{split}
\end{equation}
As $L_{\sigma \mathcal X}$ we can explicitly take the $6$-generator subgroup
$\big\langle
a,b,c, \bar a, \bar b, \bar c
\big\rangle^{w_3 w_4}$\! in $K_{\sigma \mathcal X}$ by Corollary~\ref{CO intersection and join are benign multi-dimensional}\;\eqref{PO 1 CO intersection and join are benign multi-dimensional}.
%
%
In \eqref{EQ K sigma X SHORT}\; ``$w_2$ fixes  $\big\{
a,b,c, \bar a, \bar b, \bar c
\big\}^{\!v_1 v_2}$'' means that conjugation by $w_2$ fixes 
the conjugates of \textit{each} of the generators $a,b,c, \bar a, \bar b, \bar c$ by the product $v_1 v_2$.

If $K_{\mathcal X}$ has $m$ generators (which we may assume include $a,b,c$, see Point~\ref{RE abc can be added}) and $n$ defining relations, and if 
$L_{\!\mathcal X}$ has $k$ generators, then the group $K_{\sigma \mathcal X}$ in \eqref{EQ K sigma X SHORT} has 
$3 + 3 + (m - 3) + 2 +4 = m +9$ generators and 
$n+3\cdot 3 + 3\cdot (m-3)
+3 + k +3 + 3  + 6 +3 +2\cdot 6 
= n+3m+k+30$ 
defining relations.

\subsection{The proof for the operation $\zeta$}
\label{SU The proof for zeta} 

Assume $A_{\mathcal X}$ is benign in $F_3$ for the finitely presented group $K_{\mathcal X}=\langle
\, Z
\;|\;
S
\rangle$ and for the finitely generated $L_{\mathcal X}\le K_{\mathcal X}$. Denote $\mathcal Y = \zeta \mathcal X$. If, say, 
$f=(2,5,3)\in \mathcal X$, then $\mathcal Y$ contains all possible triples $f=(n,5,3)$ with $n\in \Z$, and 
$A_{\mathcal Y}$ contains all the elements $a^{b_0^n b_1^5 b_2^3}$.

\subsubsection{Construction of $\mathcal K$ and $K_{\zeta \mathcal X}$}
\label{SU Construction of mathcal K and K zeta mathcal X}

For the group $\mathscr{A}$ from Section~\ref{SU Construction of the group A}
we can use the argument in Point~\ref{SU Construction of the direct product K x A}, to suppose $K_{\mathcal X} \cap \mathscr{A}=F_3$, and to define
$\mathcal K = K_{\mathcal X} *_{F_3} \mathscr{A}$.
For any $f\!\in \mathcal X$ and for any $n\in \Z$ we by Lemma~\ref{LE action of d_m on f} have in $\mathscr{A}$:
\begin{equation}
\label{EQ apply d^n on a_f}
a_f^{d^n}=\big(a_f^{d_0}\big)^{d_0^{n-1}}=
a_{f_0^+}^{d_0^{n-1}}
=\cdots = 
a_g
\end{equation}
where $g(0)=f(0)+n$, and $g(i)=f(i)$ for any $i\neq 0$. Here we assumed $n$ to be positive, but the negative case is covered using $f_0^-$, and so we have $
A_{\zeta \mathcal X} \subseteq \langle\, A_{\mathcal X}, d\rangle$.
Since $A_{\zeta \mathcal X} \subseteq F_3$, then also 
$A_{\zeta \mathcal X} \subseteq
F_3 \cap \langle\, A_{\mathcal X}, d\rangle$.

To show the reverse inclusion apply the ``conjugates collecting'' process \eqref{EQ elements <X,Y>} 
for the sets $\X=\big\{a_f \;|\; f\in \mathcal X\big\}$
and $\Y=\{d \}$.
We can write every word $w\in \langle\, A_{\mathcal X}, d\rangle= \langle \X,\Y \rangle$ as a product of words $u, v$ via:
\begin{equation}
\label{EQ write w as a product u.v for zeta}
w=u\cdot v
= 
a_{f_1}^{\pm v_1}
a_{f_2}^{\pm v_2}
\cdots
a_{f_k}^{\pm v_k}
\cdot
v
\end{equation}
where all $f_i$ are in $\mathcal X$ (and hence, all $a_{f_i}$ are in $\X$), and the  words $v_1,v_2,\ldots,v_k,\; v\in\langle \Y \rangle$
simply are some powers of $d$. 
As we saw above $a_{f_i}^{\pm v_i}
=a_{f_i}^{\pm d^{n_i}}
=a_{g_i}\in 
A_{\zeta \mathcal X}
$, i.e., $u$ always belongs to $A_{\mathcal Y}=A_{\zeta \mathcal X}$
in \eqref{EQ write w as a product u.v for zeta}, and so $u \in F_3$.
Thus, whenever $w\in F_3$, then also $v\in F_3$ holds. But from the last step of construction of $\mathscr{A}$ (as an HNN-extension) in Section~\ref{SU Construction of the group A} it is evident that a power of the stable letter $d$ is in $F_3$ only if it is trivial, and so
$
A_{\zeta \mathcal X} =
F_3 \cap 
\langle\, A_{\mathcal X}, d\rangle$.

$A_{\mathcal X}$ is benign not only in $F_3$ but also in $\mathcal K$ for the finitely presented $\mathcal K$ and for the same finitely generated $L_{\!\mathcal X}$ mentioned at the beginning of this section. $\langle d \rangle$ is benign in $\mathcal K$ by Remark~\ref{RE finite generated is benign}. Hence, 
by Corollary~\ref{CO intersection and join are benign multi-dimensional}\;\eqref{PO 2 CO intersection and join are benign multi-dimensional} the join $\langle\, A_{\mathcal X}, d\rangle$ is benign in $\mathcal K$ for the finitely presented group: 
$$
K_{\langle\, A_{\mathcal X}, d\rangle}
= \big(\mathcal K *_{L_{\mathcal X}} v_1 \big) \,*_{\mathcal K}
\big(\mathcal K *_{\langle d \rangle} v_2 \big),
$$
and for its finitely generated subgroup $L_{\langle\, A_{\mathcal X}, d\rangle}
= \langle\, 
\mathcal K^{\,v_1}, \;
\mathcal K^{\,v_2}
\rangle$.

As $F$ is also benign in $\mathcal K$, by Corollary~\ref{CO intersection and join are benign multi-dimensional}\;\eqref{PO 1 CO intersection and join are benign multi-dimensional} the intersection $A_{\zeta \mathcal X}$ is benign in $\mathcal K$ for the finitely presented overgroup:
$$
K_{\zeta \mathcal X}
=K_{\mathcal Y}
= 
\big(K_{\langle\, A_{\mathcal X}, d\rangle} *_{
L_{\langle\, A_{\mathcal X,\; d}\rangle}
} \!v_3\big)
\,*_{\mathcal K} 
\big(\mathcal K *_{F_3} v_4 \big), 
$$
and for its finitely generated subgroup 
$L_{\zeta \mathcal X}\!
=L_{\mathcal Y}\!
=\mathcal K^{\,v_3 v_4}$.
Since $A_{\zeta \mathcal X}$ entirely is inside $F_3$, then it is also benign in $F_3$ for the same choice of $K_{\mathcal Y}$ and $L_{\mathcal Y}$ above.

\subsubsection{Writing $K_{\zeta \mathcal X}$ explicitly}
\label{SU Writing K zeta X  by its generators and defining relations} 

Now we can write:
\begin{equation}
\label{EQ K zeta X}
\begin{split}
K_{\zeta \mathcal X}
=
\Big\langle
X_{\! \mathscr{A}};\; 
Z 
\backslash \{a, b, c \};\;v_1,\ldots, v_4 
\;\mathrel{|} \;\;  
R_{\mathscr{A}};\;
S; \\
&\hskip-57mm 
\text{$v_1$ fixes the generators of  $L_{\!\mathcal X}$};
\;\;\;
\text{$v_2$ fixes $d$};\; 
\\
&\hskip-57mm \text{$v_3$ fixes $X_{\! \mathscr{A}}^{v_1}$, $Z^{v_1}$, $X_{\! \mathscr{A}}^{v_2}$, $Z^{v_2}$\,;}\;\;\;
\text{$v_4$ fixes $a,b,c$}
\Big\rangle\,.
\end{split}
\end{equation}
In the first line we exclude $a, b, c$ from the generators $Z$ by Point~\ref{RE abc can be added}.
As a finitely generated subgroup  in $K_{\zeta \mathcal X}$ take $L_{\zeta \mathcal X}=
\big\langle
X_{\! \mathscr{A}}^{v_3 v_4} 
\cup
Z^{v_3 v_4}\big\rangle$.

If $K_{\mathcal X}$ has $m$ generators (which we may assume include $a,b,c$, see Point~\ref{RE abc can be added}) and $n$ defining relations, and if 
$L_{\!\mathcal X}$ has $k$ generators, then the group $K_{\zeta \mathcal X}$ in \eqref{EQ K zeta X}
has 
$9+ (m-3)+4 = m+10$ generators and 
$20+n+k+1
+ 9 + m + 9 + m + 3 
= 2m + n+ k + 42
$ relations.

\subsection{The proof for the operation $\pi$}
\label{SU The proof for pi} 

Assume $A_{\mathcal X}$ is benign in $F$ for the finitely presented group $K_{\mathcal X}=\langle
\, Z
\;|\;
S
\rangle$ and for the finitely generated $L_{\mathcal X}\le K_{\mathcal X}$. Denote $\mathcal Y = \pi \mathcal X$. If, say, 
$f=(2,5,3)\in \mathcal X$, then $\mathcal Y$ contains all possible tuples of type $f=(2,j_1,\ldots,j_{m-1})$ for all $m=1,2,\ldots$
%
For this case we are going to mix some constructions from sections \ref{SU The proof for zeta} and \ref{SU The *-construction}.
%

\subsubsection{Construction of a new $\mathcal K$ and $K_{\pi \mathcal X}$}
\label{SU Construciton of mathcal K* and K pi mathcal X}

In  $\mathscr{A}$ the stable letters $d,e$ clearly generate a free subgroup $\langle d,e\rangle$ of rank $2$. 
In analogy to the isomorphisms $\xi_m, \xi'_m$ on $\langle b,c\rangle$ in Section~\ref{SU The *-construction}, 
we can define  
isomorphisms
$\psi_m, \psi_m'$ by the rules:
$\psi_m(d)=d_{\,-m+1},\;
\psi_m'(d)=d_{\,-m}$,\;
$\psi_m(e)=\psi_m'(e)=e^2$\!.\;
Taking $m=1$ we have 
$\psi_1(d)=d_0=d,\;
\psi_1'(d)=d_{\,-1}$,\;
$\psi_1(e)=\psi_1'(e)=e^2$.
Use these $\psi_1, \psi_1'$ to construct the HNN-extension
$\mathscr{A}^* = \mathscr{A} *_{\psi_1, \psi_1'} (x_1, x'_1)$ in which $d, e, x_1, x'_1$ evidently generate the subgroup:
$$
\Psi_1 =\langle d, e, x_1, x'_1\rangle
= \langle d, e\rangle
*_{\psi_1, \psi_1'} (x_1, x'_1)
\;\cong\; \Xi_1.
$$
Since the analogs of  statements on $\Xi_m$ from Section~\ref{SU The *-construction} also hold for $\Psi_m$, we adapting Lemma~\ref{LE Ksi}, and denoting $D=\langle d_1, d_{2},\ldots\,\rangle$ have:   
\begin{equation}
\label{EQ D obtained in two ways}
\langle d,e \rangle \cap \langle d_1, x_1, x'_1\rangle = D
\quad\text{and also}\quad
\mathscr{A} \cap \langle d_1, x_1, x'_1\rangle = D,
\end{equation}
i.e., $D$ is benign in $\mathscr{A}$ for the finitely presented overgroup
$K_D=\mathscr{A}^*$ and for the finitely generated 
$L_D=\langle d_1, x_1, x'_1\rangle$. 

As in Section~\ref{SU The proof for zeta}, we may suppose none of $t_1, t'_1, u_1, u_2, d,e$ and \textit{also} none of $x_1, x'_1$ is involved in  $K_{\mathcal X}$, so 
$K_{\mathcal X}$ and $\mathscr{A}^*$ intersect in $F$, and we can define $\mathcal K = K_{\mathcal X} *_F \mathscr{A}^*$ (the group $\mathcal K$ of Section~\ref{SU Construction of mathcal K and K zeta mathcal X} is the analog of the current $\mathcal K$, it plays a similar role, and we prefer to denote them by the same symbol).
For any $f\in \mathcal X$, for any $n\in \Z$, and for any positive index $j=1,2,\ldots$ we, using Lemma~\ref{LE action of d_m on f} in analogy with \eqref{EQ apply d^n on a_f},  compute in $\mathscr{A}^*$:
\begin{equation}
\label{EQ next apply d_j^n on a_f}
a_f^{d_j^n}=\big(a_f^{d_j}\big)^{d_j^{n-1}}=
a_{f_j^+}^{d_j^{n-1}}
=\cdots = 
a_g
\end{equation}
where $g(j)=f(j)+n$, and $g(i)=f(i)$ for all $i\neq j$. Here we applied a positive $n$, while the negative case could be covered using $f_j^-$ instead.
Then taking another positive value for $j$ we could repeatedly apply the analog of \eqref{EQ next apply d_j^n on a_f} on $a_g$. This way we could after finitely many steps  construct arbitrary element from 
$A_{\pi \mathcal X}$. Since each $a_g$ is also in $F_3$,  we arrive to the inclusion
$
A_{\pi \mathcal X} \subseteq
F_3 \cap \langle\, A_{\mathcal X}, D\rangle$.

To achieve the reverse inclusion apply the ``conjugates collecting'' process \eqref{EQ elements <X,Y>} 
for the sets $\X=\big\{a_f \;|\; f\in \mathcal X\big\}$
and $\Y=\big\{ d_1, d_{2},\ldots\,\big\}$.
Write every word $w\in \langle\, A_{\mathcal X}, D\rangle= \langle \X,\Y \rangle$ as:
\begin{equation*}
w=u\cdot v
= 
a_{f_1}^{\pm v_1}
a_{f_2}^{\pm v_2}
\cdots
a_{f_k}^{\pm v_k}
\cdot
v
\end{equation*}
where all $f_i$ are in $\mathcal X$ (and hence, all $a_{f_i}$ are in $\X$), and the   $v_1,v_2,\ldots,v_k,\; v\in\langle \Y \rangle$
are some words on $d_1, d_{2},\ldots$, i.e., some words on the letters $d,e$.
As we have shown by repeated application of \eqref{EQ next apply d_j^n on a_f} above, each $a_{f_i}^{\pm v_i}$ is in $A_{\pi \mathcal X}$. So $u$ is also in $A_{\mathcal Y}=A_{\pi \mathcal X}$ and, hence, is in $u \in F_3$. 
Thus, whenever $w\in F_3$, then  $v\in F_3$ also. But from construction of $\mathscr{A}$ 
(as an HNN-extension) in Section~\ref{SU Construction of the group A}
it is clear that a word on $d_1, d_{2},\ldots$ (i.e., a word on stable letters $d,e$) is in $F_3$ only if it is trivial, and so we have the equality
$
A_{\pi \mathcal X} =
F \cap 
\langle\, A_{\mathcal X}, D\rangle$.

$D$ is benign in $\mathcal K$ because it is benign in its subgroups $\mathscr{A}$ (or even in $\langle d,e \rangle$), see \eqref{EQ D obtained in two ways}.
$A_{\mathcal X}$ is benign not only in $F_3$ but also in $\mathcal K$ for the finitely presented $\mathcal K$ and for the same finitely generated $L_{\mathcal X}$ supposed above. 
Hence, 
by Corollary~\ref{CO intersection and join are benign multi-dimensional}\;\eqref{PO 2 CO intersection and join are benign multi-dimensional} the join $\langle\, A_{\mathcal X}, D \rangle$ is also benign in $\mathcal K$ for the finitely presented group: 
$$
K_{\langle\, A_{\mathcal X}, d\rangle}
= \big(\mathcal K *_{L_{\mathcal X}} v_1 \big) *_{\mathcal K}
\big(\mathcal K *_{\langle d_1, x_1, x'_1\rangle} v_2 \big),
$$
and for its finitely generated subgroup $L_{\langle\, A_{\mathcal X},\, D\rangle}
= \langle\, 
\mathcal K^{\,v_1}, \;
\mathcal K^{\,v_2}
\rangle$.

As $F_3$ is also benign in $\mathcal K$, by Corollary~\ref{CO intersection and join are benign multi-dimensional}\;\eqref{PO 1 CO intersection and join are benign multi-dimensional} the intersection $A_{\pi \mathcal X} =
F_3 \cap 
\langle\, A_{\mathcal X}, D\rangle$ is benign in $\mathcal K$ for the finitely presented overgroup:
$$
K_{\pi \mathcal X}
=K_{\mathcal Y}
= 
\big(K_{\langle\, A_{\mathcal X},\, D\rangle} *_{L_{\langle\, A_{\mathcal X},\; D\rangle}} \!v_3\big)
*_{\mathcal K} 
\big(\mathcal K *_{F_3} v_4 \big) 
$$
and for its finitely generated subgroup 
$L_{\pi \mathcal X}\!
=L_{\mathcal Y}\!
=\mathcal K^{\,v_3 v_4}$.
Since $A_{\pi \mathcal X}$ entirely is inside $F_3$, then it is also benign in $F_3$ for the same choice of $K_{\mathcal Y}$ and $L_{\mathcal Y}$ made above.

\subsubsection{Writing $K_{\pi \mathcal X}$ explicitly}
\label{SU Writing K pi X  by its generators and defining relations} 

It remains to explicitly write:
\begin{equation}
\label{EQ K pi X}
\begin{split}
K_{\pi \mathcal X}
=
\Big\langle
X_{\! \mathscr{A}};\; 
Z
\backslash \{a, b, c \};\;
x_1,x_2 ;\;
v_1,\ldots, v_4 
\;\mathrel{|} \;\;  
R_{\mathscr{A}};\;
S; \\
&\hskip-72mm 
\text{$x_1$ sends  $d,e$ to $d, e^2$};
\;\;\;\;
\text{$x'_1$ sends  $d,e$ to $d^{e^{-1}}, e^2$};\; 
\\
&\hskip-72mm 
\text{$v_1$ fixes the generators of  $L_{\!\mathcal X}$};
\;\;\;
\text{$v_2$ fixes $d^e, x_1, x'_1$};\; 
\\
&\hskip-72mm \text{$v_3$ fixes $X_{\! \mathscr{A}}^{\! v_1}$, $X^{\! v_1}$, 
$\{x_1, x'_1\}^{\! v_1}$,\;
$X_{\! \mathscr{A}}^{\! v_2}$, $X^{\! v_2}$,
$\{x_1, x'_1\}^{\! v_2}$
;}\\
&\hskip-72mm 
\text{$v_4$ fixes $a,b,c$}
\Big\rangle\,.
\end{split}
\end{equation}
We exclude $a, b, c$ from the generators $Z$ because they already are included in $X_{\! \mathscr{A}}$.
As a finitely generated subgroup  in $K_{\pi \mathcal X}$ take $L_{\pi \mathcal X}=
\mathcal K^{v_3 v_4}
=
\big\langle
X_{\! \mathscr{A}}^{v_3 v_4}
\cup
Z^{v_3 v_4}
\cup
\{x_1, x'_1\}^{v_3 v_4}
\big\rangle$.

If $K_{\mathcal X}$ has $m$ generators (which we may assume include $a,b,c$, see Point~\ref{RE abc can be added}) and $n$ defining relations, and if 
$L_{\!\mathcal X}$ has $k$ generators, then the group $K_{\pi \mathcal X}$ in \eqref{EQ K pi X}
has 
$9+ (m-3)+2+4 = m+12$ generators and 
$20+n+2+2+k+3+2\cdot(9+m+2)+3 = n+k+2m+52
$ defining relations.

\subsection{The proof for the operation $\theta$}
\label{SU The proof for theta}

Assume $A_{\mathcal X}$ is benign in $F_3$ for the explicitly given fini\-tely presented group $K_{\mathcal X}=\langle
\, Z
\;|\;
S
\rangle$ and for the finitely generated $L_{\mathcal X}\le K_{\mathcal X}$. 
Denote $\mathcal Y = \theta \mathcal X$. 
If, say, 
$f=(2,5,3)$ or $f=(2,5,3,8)$ are in $\mathcal X$, then $\mathcal Y$ contains the couple $\theta f = (2,3)$.

We are free to use the copy $\bar F = \langle
\bar a, \bar b, \bar c
\rangle$ of $F$, and suppose the copies 
$\bar K_{\mathcal X}=\langle
\, \bar Z
\;|\;
\bar S
\rangle$ and
$\bar L_{\mathcal X}$ are given for the copy $\bar A_{\mathcal X}$ of $A_{\mathcal X}$, see Point~\ref{NOR SU Auxiliary copy of A built here}.

\subsubsection{Finding the benign subgroup $O$}
\label{SU Finding the benign subgroup O} 

We reuse the copy
$\bar{\mathscr{A}}$ of $\mathscr{A}$
with the generating set $X_{\! \bar {\mathscr{A}}}$ from \eqref{EQ generators of X bar A}.
In analogy with Point~\ref{SU Construciton of mathcal K* and K pi mathcal X}, in $\bar {\mathscr{A}}$ the letters $\bar d, \bar e$ generate a free subgroup $\langle \bar d,\bar e\rangle$, and we can define an isomorphism $\varepsilon$ on it sending $\bar d,\bar e$ to $\bar d_2, \bar e$.
Using it construct the HNN-extension
$\bar {\mathscr{A}}^{*} = \bar {\mathscr{A}} *_{\varepsilon} y$ inside which $\bar d, \bar e, y$ clearly generate the HNN-extension:
$$
\langle \bar d, \bar e,y\rangle
= \langle \bar d, \bar e\rangle
*_{\varepsilon} y.
$$
Set $O=\langle \bar d_{j} \;|\; j\!=\!2k+\!1,\; k\in \Z\rangle$ to be the subgroup generated by $\bar d_{j}$ for all odd $j$, and show that $O$ is benign in $\bar {\mathscr{A}}$. It is easy to show that:
\begin{equation}
\label{EQ O obtained in two ways}
\langle \bar d,\bar e \rangle \cap \langle \bar d_1, y\rangle = O
\quad\text{and}\quad
\bar {\mathscr{A}} \cap \langle \bar d_1, y\rangle = O.
\end{equation}
Clearly, only the first of equalities \eqref{EQ O obtained in two ways} needs a verification.
For any $j=2k\!+\!1$ we have 
$\bar d_{j}
=\bar d^{\bar e^{\;2k+1}}
=\bar d^{y^{k}\cdot\, \bar e}
=\bar d_1^{y^k}$\!\!\!,\; and so\; $O \subseteq \langle \bar d,\bar e \rangle \cap \langle \bar d_1, y\rangle$.
On the other hand, using the process \eqref{EQ elements from <x,y>} applied to  $x=\bar d_1$ and to the current $y$ we can rewrite any $w\in \langle \bar d_1, y\rangle$ as:
\begin{equation}
\label{EQ elements from O}
w=\bar d_1^{\;\pm y^{k_1}}
\bar d_1^{\;\pm y^{k_2}}\!\!\cdots\, \bar d_1^{\;\pm y^{ k_s}}\!\!\!\cdot y^{l}=\!u\cdot v\,.
\end{equation}
As we just saw, all the $\bar d_1^{\;\pm y^{k_i}}$\!\! are equal to some $\bar d_j$ with \textit{odd} $j$ and, thus,  $u\in O$. If also $w\in \langle \bar d,\bar e \rangle$, then $v\!=\!y^l$ must be trivial, as it is a power of the stable letter $y$ of our HNN-extension.
\eqref{EQ O obtained in two ways} is proven, and so $O$ is benign  in $\bar {\mathscr{A}}$ for the finitely presented overgroup
$K_O=\bar {\mathscr{A}}^{*}$ and for the finitely generated 
$L_O=\langle \bar d_1, y\rangle$.

\subsubsection{Finding the benign intersection $\bar A_{\vartheta \mathcal X}$}
\label{SU Finding the benign intersection A} 

Denote by $\vartheta \mathcal X$ the auxiliary set of all functions $f\!\in \E$ for which there is a $g\in \mathcal X$ such that $f(j) = g(j)$ on all \textit{even} indices $j=2k$, and the value $f(j)$ is arbitrary for all \textit{odd} indices $j=2k\!+\!1$. In this step we show that $\bar A_{\vartheta \mathcal X}$ is benign in $\bar F$ in order to use this for the required set $\theta \mathcal X$ later.

We may suppose none of the letters $\bar t_1, \bar t'_1, \bar u_1, \bar u_2, \bar d,\, \bar e$ and $y$ was involved in construction of $\bar K_{\mathcal X}$, i.e., 
$\bar K_{\mathcal X} \cap \bar {\mathscr{A}}^{*} = \bar F$, and we can set $
\mathcal K 
= \bar K_{\mathcal X} *_{\bar F}  K_{O}
= \bar K_{\mathcal X} *_{\bar F} \bar {\mathscr{A}}^*$ 
(the current group $\mathcal K$ of course is different from $\mathcal K$ used in points~\ref{SU Construction of mathcal K and K zeta mathcal X} and \ref{SU Construciton of mathcal K* and K pi mathcal X}, but it plays a similar role):
$$
\mathcal K
=
\Big\langle
X_{\! \bar{\mathscr{A}}};\; 
\bar Z
\backslash \{\bar a, \bar b, \bar c \};\;
y ;\;
a,b,c 
\;\mathrel{|} \; 
R_{\!\bar{\mathscr{A}}};\;
\bar S;\;  
\text{$y$ sends $\bar d,\bar e$ \,to\, $\bar d_2, \bar e$}
\Big\rangle.
$$
For below references denote the set of $9+m-3+1+3=m+10$ generators of $\mathcal K$ by:
\begin{equation}
\label{EQ Generators of mathcal K}
X_{\!\mathcal K}
=
X_{\! \bar{\mathscr{A}}}
\cup\,
\bar Z
\backslash \{\bar a, \bar b, \bar c \}
\,\cup 
\{
y ;\;
a,b,c
\}.
\end{equation}

For any $f\!\in \mathcal X$,
for any positive $n$, 
and for any \textit{odd} integer index $j=2k\!+\!1$ we again by Lemma~\ref{LE action of d_m on f} 
compute in $\bar {\mathscr{A}}^*$:
\begin{equation}
\label{EQ next apply odd d_j on a_f}
\bar a_f^{{\bar d}_j^{\,n}}
=\big(\bar a_f^{{\bar d}_{2k+1}}\big)^{{\bar d}_{2k+1}^{\;n\,-1}}=
\bar a_{f_{2k+1}^+}^{\; {\bar d}_{2k+1}^{\;n\,-1}}
=\cdots = 
\bar a_g
\end{equation}
where $g(j)=f(j)+n$, and $g(i)=f(i)$ for all $i\neq j$. Here we used a positive $n$, while the negative case could be covered by $f_j^-$.
Choosing yet another \textit{odd} $j$ we could apply the analog of \eqref{EQ next apply odd d_j on a_f} on $\bar a_g$. After finitely many such steps we construct all the elements $\bar a_g$ from 
$\bar A_{\vartheta \mathcal X}$. Since these $\bar a_g$ also are in $\bar F_3$,  we arrive to the inclusion
$
\bar A_{\vartheta \mathcal X} \subseteq
\bar F_3 \cap \langle\, \bar A_{\mathcal X}, O\rangle$.

For the reverse inclusion apply \eqref{EQ elements <X,Y>} 
for the sets $\X=\big\{\bar a_f \;|\; f\in \mathcal X\big\}$
and $\Y=\big\{
\bar d_{j} \;|\; j=2k\!+\!1,\; k \in \Z
\big\}$.
Write every word $w\in \langle\, \bar A_{\mathcal X}, O\rangle= \langle \X,\Y \rangle$ as:
\begin{equation*}
w=u\cdot v
= 
\bar a_{f_1}^{\;\pm v_1}
\bar a_{f_2}^{\;\pm v_2}
\cdots
\bar a_{f_k}^{\;\pm v_k}
\cdot
v
\end{equation*}
where all $f_i$ are in $\mathcal X$, and    $v_1,v_2,\ldots,v_k,\; v\in\langle \Y \rangle$
are words on some letters $\bar d_j$ with various \textit{odd} indices $j$.
As we have seen in \eqref{EQ next apply odd d_j on a_f}, each $\bar a_{f_i}^{\;\pm v_i}$ is in $\bar A_{\vartheta \mathcal X}$. So $u$ is also in $\bar A_{\vartheta \mathcal X} \subseteq \bar F_3$.
Thus, whenever $w\in \bar F_3$, then  $v\in \bar F_3$ also. But $v$ is a word on stable letters $\bar d,\bar e$ and it can be  in $\bar F_3$ only if it is trivial. 

We proved that 
$
\bar A_{\vartheta \mathcal X} =
\bar F_3 \cap 
\langle\, \bar A_{\mathcal X}, O\rangle$, i.e., $\bar A_{\vartheta \mathcal X}$ is constructed from three benign subgroups $\bar A_{\mathcal X},\, O,\, \bar F_3$ by a join and an intersection. 
By Corollary~\ref{CO intersection and join are benign multi-dimensional}\;\eqref{PO 2 CO intersection and join are benign multi-dimensional} the join $\langle\, \bar A_{\mathcal X}, O\rangle$ is benign in $\mathcal K$ for the finitely presented group: 
$$
K_{\langle\, \bar A_{\mathcal X}, \, O\rangle}
= \big(\mathcal K *_{\bar L_{\mathcal X}} v_1 \big) *_{\mathcal K}
\big(\mathcal K *_{\langle \bar d_1,  \,y \rangle} v_2 \big),
$$
and for its finitely generated subgroup $L_{\langle\, \bar A_{\mathcal X}, O\rangle}
= \langle\, 
\mathcal K^{v_1}, \;
\mathcal K^{v_2}
\rangle$.
Then by Corollary~\ref{CO intersection and join are benign multi-dimensional}\;\eqref{PO 1 CO intersection and join are benign multi-dimensional} the intersection $\bar A_{\vartheta \mathcal X}$ is benign in $\mathcal K$ for the finitely presented:
$$
K_{\vartheta \mathcal X}
= 
\big(K_{\langle\, \bar A_{\mathcal X}, O\rangle} *_{
L_{\langle\, \bar A_{\mathcal X},\; O\rangle}
} \!v_3\big)
*_{\mathcal K} 
\big(\mathcal K *_{\bar F_3} v_4 \big) 
$$
and for its finitely generated subgroup 
$L_{\vartheta \mathcal X}\!
=\mathcal K^{\,v_3 v_4}$.
Since $\bar A_{\vartheta \mathcal X}$ is in $\bar F_3$, then it is also benign in $\bar F_3$ for the same choice of $K_{\vartheta \mathcal X}$ and $L_{\vartheta \mathcal X}$ just made above.

\subsubsection{Obtaining the benign subgroup $Q$ for $\theta$}
\label{SU Obtaining the benign subgroup for theta Q}

By our construction 
$\vartheta \mathcal X$ is the set of all functions $f\!\in \E$ which coincide with some $g\!\in \mathcal X$ on all \textit{even} indices, and which may have arbitrary coordinates on \textit{odd} indices. 
In particular, $\vartheta \mathcal X$ contains all those  $f\!\in \E$ which coincide with some $g\!\in \mathcal X$ on all even indices, and are \textit{zero} on all odd indices. 

If
$T=\big\langle 
(\bar a,\; a),\;
(\bar b,\; b),\;
(\bar c^{\;2}\!,\; c)
\big\rangle$, and
$Q= T \cap(\bar A_{\vartheta \mathcal X} \!\times F_3) $, then the combinatorial meaning of this intersection is uncomplicated to understand. 
If, say, $f\!=(2,5,3,8)$ is in $\mathcal X$, then $\vartheta \mathcal X$ contains \textit{all} functions of type $(2,x,3,y)$ with $x,y\in \Z$ and, in particular, it contains the function $f_0$ obtained from $f$ by changing to zero all its coordinates for odd indices, in this case $f_0=(2,0,3,0)$.
Since $\theta$ for any $f$ simply ignores all coordinates of $f$ for odd indices, then
$\theta(\vartheta \mathcal X)= \theta \mathcal X$. 
In particular, this means that $\theta f_0 = \theta f$, such as $\theta(2,0,3,0)=\theta(2,5,3,8)=(2,3)$.
Since $T$ for any $i$ contains the couple $
(\bar b_{2\, \cdot \, i},\, b_i) =(\bar b^{(\bar c^{2})^i}\!\!\!,\;\; b^{c^i})
= (\bar b, b)^{(\bar c^{\;2}\!\!,\; c)^{\,i}}$\!\!,\;  it also contains the couples of type $(\bar a_{f_0} ,\; a_{\theta f_0})$, such as:
$$
(\bar a,\, a)^{\;
(\bar b_{2\,\cdot\, 0},\; b_0)^{\,2}
\,\cdot\;
(\bar b_{2\,\cdot\, 1},\; b_1)^{\,3}
}
=
\big(
\bar a^{\;
\bar b_0^2 \,
\bar b_2^3 \,
}
,\;\;
a^{\,
b_0^2 \,
b_1^3 \,
}
\big)
=
\big(
\bar a^{\;
\bar b_0^2 \;
\boldsymbol{\bar b_1^0} \,
\bar b_2^3 \;
\boldsymbol{\bar b_3^0} 
}
,\;\;
a^{\,
b_0^2 \,
b_1^3 \,
}
\big)
=
(\bar a_{f_0} ,\; a_{\theta f_0})
$$
(we added the trivial factors ${\bar b_1^0}$ and ${\bar b_3^0}$ to ``reconstruct'' $f_0$). 
Clearly, $Q$ also contains the couples $(\bar a_{f_0} ,\; a_{\theta f_0})$ for all such functions $f_0$.

On the other hand, if a couple from $\bar A_{\vartheta \mathcal X} \times F_3$ is in $T$, then it must be generated by the couples 
$(\bar a,\; a),\;
(\bar b,\; b),\;
(\bar c^{\;2}\!,\; c)$ and, hence, its first coordinate has to involve $\bar c$ in even degrees \textit{only}, e.g., it may never contain subwords like 
$\bar b_3 = \bar b^{\, \bar c^{\,3}}$ or $\bar a^{\,\bar b_5}$.
But $\bar A_{\vartheta \mathcal X}$ is \textit{freely} generated by elements $\bar a_f$ for $f\!\in \vartheta \mathcal X$, such as $f\!=(2,5,3,8)$ or $(2,0,3,0)$. Hence, an element 
from $\bar A_{\vartheta \mathcal X}$ will be in $\langle \bar a, \bar b, \bar c^{\;2} \rangle$ only if all the coordinates for odd indices in $f$ are zero, i.e., $f=f_0$. 

The equality $Q=\big\langle 
(\bar a_{f_0} ,\; a_{\theta f_0})
\;|\;
f\in \mathcal X
\big
\rangle
=
\big\langle
(\bar a_{f_0} ,\; a_{\theta f})
\;|\;
f\in \mathcal X
\big
\rangle
$
has been proved, and for any $f\in \mathcal X$ we have the element $a_{\theta f}\in A_{\theta \mathcal X}$ standing as the \textit{second} coordinate in one of the couples $(\bar a_{f_0} ,\; a_{\theta f})$ above. This does \textit{not} mean that $\bar a_{f}$ also occurs as a \textit{first} coordinate in one of those couples, but we are not in need of that fact either.

\medskip 
Since $\bar A_{\vartheta \mathcal X}$ is benign in $\bar F_3$ for the earlier mentioned $K_{\vartheta \mathcal X}$ and $L_{\vartheta \mathcal X}$, then the direct product 
$\bar A_{\vartheta \mathcal X} \times F_3$ is benign in $\bar F_3 \times F_3$ for $K_{\vartheta \mathcal X} \times F_3$ and $L_{\vartheta \mathcal X} \times F_3$.
The subgroup $T=\big\langle 
(\bar a,\; a),\;
(\bar b,\; b),\;
(\bar c^{\;2}\!,\; c)
\big\rangle$
is benign in $\bar F_3 \times F_3$ by Remark~\ref{RE finite generated is benign} for $K_T=\bar F_3 \!\times\! F_3$ and $L_T=T$.
Hence, by Corollary~\ref{CO intersection and join are benign multi-dimensional}\;\eqref{PO 1 CO intersection and join are benign multi-dimensional} their intersection 
$Q$ is benign in $\bar F_3 \!\times\! F_3$ for the finitely presented:
$$
K_{Q}
= 
\big( 
(K_{\vartheta \mathcal X} \!\times\! F_3)
*_{
L_{\vartheta \mathcal X} \times \,F_3
} \; v_5\big)
*_{\bar F_3 \times F_3} 
\big((\bar F_3 \!\times\! F_3) *_T v_6 \big) 
$$
and for its finitely generated subgroup 
$L_{Q}
=(\bar F_3 \!\times\! F_3)^{\,v_5 v_6}$.

\subsubsection{``Extracting'' $A_{\theta \mathcal X}$ from $Q$}
\label{SU Extracting A theta X from Q} 

To ``extract'' the $A_{\theta \mathcal X}$ from $Q$ notice that by Corollary~\ref{CO intersection and join are benign multi-dimensional}\;\eqref{PO 2 CO intersection and join are benign multi-dimensional} the join 
$Q_1 = \big\langle 
\bar F_3\times  \1,\;
Q 
\big\rangle$ is benign in 
$\bar F_3 \!\times\! F_3$ for the finitely presented:
$$
K_{Q_1}
= 
\big((\bar F_3 \!\times\! F_3) 
*_{\bar F_3 \times \1} w_1 \big) 
*_{\bar F_3 \times F_3} 
\big( 
K_{Q}
*_{L_Q}  w_2\big)
$$
and for its finitely generated subgroup: 
$$
L_{Q_1} = \big\langle(\bar F_3 \!\times\! F_3)^{w_1}\!,\;(\bar F_3 \!\times\! F_3)^{w_2} \big\rangle.
$$
Then by Corollary~\ref{CO intersection and join are benign multi-dimensional}\;\eqref{PO 1 CO intersection and join are benign multi-dimensional} the intersection
$A_{\theta \mathcal X}\!=
\big(\1 \!\times\!  F_3 \big) \cap Q_1$
is benign in $\bar F_3 \!\times\! F_3$ for the finitely presented:
$$
K_{\theta \mathcal X}
=
\big((\bar F_3 \!\times\! F_3)\,*_{\1 \times  F_3} w_3\big) 
\,*_{\bar F_3 \times F_3}
\big(K_{Q_1} *_{L_{Q_1}}\! w_{4}\big)
$$
and for the finitely presented
$L_{\theta \mathcal X}
=(\bar F_3 \!\times\! F_3)^{w_3 w_4}$.
Since $A_{\theta \mathcal X}$ is inside $F_3$, it is benign in $F_3$ also, with the same choice for $K_{\theta \mathcal X}$, 
$L_{\theta \mathcal X}$.

\subsubsection{Writing $K_{\theta \mathcal X}$ explicitly}
\label{SU Writing K theta X  by its generators and defining relations} 

Now taking into account the above notation we can write:
\begin{equation}
\label{EQ K theta X}
\begin{split}
K_{\theta \mathcal X}
=
\Big\langle
X_{\! \bar{\mathscr{A}}};\; 
\bar Z 
\backslash \{\bar a, \bar b, \bar c \};\;
y ;\;
a,b,c ;\;
v_1,\ldots, v_{6}; 
w_1,\ldots, w_{4}
\;\mathrel{|} \; 
R_{\bar{\mathscr{A}}};\;
\bar S; \\
&\hskip-93mm 
\text{$y$ sends $\bar d,\bar e$ \,to\, $\bar d_2, \bar e$};
\\
&\hskip-93mm 
\text{$v_1$ fixes the generators of  $\bar L_{\!\mathcal X}$};
\;\;\;
\text{$v_2$ fixes $\bar d_1,  \,y$};\; 
\\
&\hskip-93mm \text{$v_3$ fixes $X_{\!\mathcal K}^{v_1}\! \cup X_{\!\mathcal K}^{v_2}$;} 
\hskip3mm 
\text{$v_4, w_1$ fix $\bar a, \bar b, \bar c$};\\
&\hskip-93mm 
\text{$a,b,c$ commute with $X_{\! \bar{\mathscr{A}}};\; 
\bar Z;\;
y ;\;
v_1,\ldots, v_{4}$};\\
&\hskip-93mm 
\text{$v_5$ fixes $X_{\!\mathcal K}^{v_3 v_4}$};
\hskip3mm
\text{$v_5, w_3$ fix $a, b,  c$};\\
&\hskip-93mm 
\text{$v_6$ fixes $
\bar a a,\,
\bar b b,\,
\bar c^{\;2} c
$};
\hskip3mm
\text{$w_2$ fixes $\big\{
a,b,c, \bar a, \bar b, \bar c
\big\}^{\! v_5 v_6}$};\\
&\hskip-93mm 
\text{$w_4$ fixes $\big\{
a,b,c, \bar a, \bar b, \bar c
\big\}^{w_1}
\!\cup
\big\{
a,b,c, \bar a, \bar b, \bar c
\big\}^{w_2}$}
\Big\rangle\,.
\end{split}
\end{equation}
By Point~\ref{RE abc can be added} we exclude $a, b, c$ from the generators $Z$.
As a finitely generated subgroup  in $K_{\theta \mathcal X}$ take $L_{\theta \mathcal X}=
(\bar F_3 \!\times\! F_3)^{w_3 w_4}
=
\big\langle
a,b,c, \bar a, \bar b, \bar c
\big\rangle^{w_3 w_4}$.
If $K_{\mathcal X}$ has $m$ generators and $n$ defining relations, and if 
$L_{\!\mathcal X}$ has $k$ generators, then the group $K_{\theta \mathcal X}$ in \eqref{EQ K theta X}
has 
$9+ (m-3)+1+3+6+4=m+20$ generators and 
$20+n+2+k+2 + 2\cdot (m + 10)+
2 \cdot 3+
3 \cdot (9 + m-3 + 1 + 4)
+ (m + 10) 
+ 2 \cdot 3
+ 3 + 6 + 2\cdot 6
= n + 6m + k + 111
$ defining relations.

\medskip 

\subsection{The proof for the operation $\tau$}
\label{SU The proof for tau}

Let $\mathcal Y = \tau \mathcal X$, i.e., when $\mathcal X$ contains, say, $f\!=(2,5,3)$, then $\mathcal Y$ contains $\tau f\!=(5,2,3)$. 
Assume the hypothesis of Theorem~\ref{TH Theorem A} holds for $\mathcal X$: the subgroup $A_{\mathcal X}$ is benign in $F_3$ for an explicitly given finitely presented  $K_{\mathcal X}=\langle
\, Z
\;|\;
S
\rangle$ and for its finitely generated subgroup $L_{\mathcal X}\le K_{\mathcal X}$.

\subsubsection{Writing $\langle d_i \;|\; i\in \Z \rangle$ as a product of three benign factors}
\label{SU Writing d_i as a product of three} 

Following the construction in Point~\ref{SU Construciton of mathcal K* and K pi mathcal X}, 
reuse the isomorphisms 
$\psi_m, \psi_m'$ 
on $\langle d,e\rangle$. 
For $m=0$ we have $\psi_0(d)=d_{1},\;
\psi_0'(d)=d_{0}=d$,\;
$\psi_0(e)=\psi_0'(e)=e^2$, using which we can define 
$\mathscr{B}_0 = \mathscr{A} *_{\psi_0, \psi_0'} (x_0, x'_0)$. 
Inside this group we clearly have:
$$
\langle d, e, x_0, x'_0\rangle
= \langle d, e\rangle
*_{\psi_0, \psi_0'} (x_0, x'_0)
\;\cong\; \Xi_{\,0}.
$$
Using an analog of Lemma~\ref{LE Ksi} we for the subgroup $D_0=\langle d_{\,-1}, d_{\,-2},\ldots\,\rangle$ have:   
\begin{equation}
\label{EQ D_0 obtained in two ways!!!}
\langle d,e \rangle \cap \langle d_{\,-1}, x_0, x'_0\rangle = D_0
\quad\text{and also}\quad
\mathscr{A} \cap \langle d_{\,-1}, x_0, x'_0\rangle = D_0,
\end{equation}
from where $D_0$ is benign in $\mathscr{A}$ for the finitely presented overgroup
$K_{D_0}=\mathscr{B}_0$ and for its finitely generated subgroup
$L_{D_0}=\langle d_{\,-1}, x_0, x'_0\rangle$.

\medskip
Similarly, for $m=2$ we have $\psi_2(d)=d_{\,-1},\;
\psi_2'(d)=d_{\,-2}$,\;
$\psi_2(e)=\psi_2'(e)=e^2$ by which we can define
$\mathscr{B}_2 = \mathscr{A} *_{\psi_2, \psi_2'} (x_2, x'_2)$ to discover inside it:
$$
\langle d, e, x_2, x'_2\rangle
= \langle d, e\rangle
*_{\psi_2, \psi_2'} (x_2, x'_2)
\;\cong\; \Xi_2.
$$
Then we for the subgroup $D_2=\langle d_2, d_{3},\ldots\,\rangle$ have:   
\begin{equation}
\label{EQ D_2 obtained in two ways!!!}
\langle d,e \rangle \cap \langle d_2, x_2, x'_2\rangle = D_2
\quad\text{and also}\quad
\mathscr{A} \cap \langle d_2, x_2, x'_2\rangle = D_2,
\end{equation}
that is, $D_2$ is benign in $\mathscr{A}$ for 
$K_{D_2}=\mathscr{B}_2$ and for
$L_{D_2}=\langle d_2, x_2, x'_2\rangle$.

Notice that neither $D_0$ nor $D_2$ involved the elements $d_0$ and $d_1$. The subgroup $D_1=\langle d_0, d_1 \rangle$ they generate is benign in $\mathscr{A}$ for 
$K_{D_1}=\mathscr{A}$ and for
$L_{D_1}=D_1$ by Remark~\ref{RE finite generated is benign}.

\medskip
The mentioned three finitely presented groups 
$K_{D_0}, K_{D_1} ,K_{D_2}$ intersect strictly in $\mathscr{A}$, and so we can form the $\bigast$-construction:
\begin{equation}
\label{EQ D_0 D_1 D_2}
\begin{split}
\mathscr{B}&
=
(K_{D_0}*_{L_{D_0}} y_0) 
\,*_{\!\mathscr{A}}
(K_{D_1} *_{L_{D_1}} y_1) 
\,*_{\!\mathscr{A}}
(K_{D_2} *_{L_{D_2}} y_2)\\
& 
= 
(\mathscr{B}_0 *_{\langle d_{\,-1}, x_0, x'_0\rangle} y_0) 
\,*_{\!\mathscr{A}}
(\mathscr{A} *_{ \langle d_0,\; d_1 \rangle } y_1)   
\,*_{\!\mathscr{A}}
(\mathscr{B}_2 *_{\langle d_2, x_2, x'_2\rangle} y_2)\,.
\end{split}
\end{equation}
Inside $\langle d,e \rangle$
the subgroups 
$D_0, D_1, D_2$ generate their \textit{free} product: 
\begin{equation}
\label{EQ D_0 D_1 D_2 details}
D_0 * D_1 *\, D_2 = 
\langle \ldots d_{\,-2}, d_{\,-1} \rangle
*
\langle d_0,\; d_1 \rangle 
* 
\langle d_2, d_3, \ldots \rangle
\end{equation} 
which is just the free group $\langle d_i \;|\; i\in \Z \rangle$ of countable rank.

\subsubsection{Obtaining the benign subgroup $Q$ for $\tau$}
\label{SU Obtaining the benign subgroup for tau Q}

Since the above subgroups $D_0, D_1, D_2$ also are inside $\mathscr{B}$, then by Corollary~\ref{CO smaller free product to the larger free product}
the sub\-groups:
$\mathscr{A}^{y_0}$\!,
$\mathscr{A}^{y_1}$\!,
$\mathscr{A}^{y_2}$
together generate in 
$\mathscr{B}$
their \textit{free} product
$\mathscr{A}^{y_0}* 
\mathscr{A}^{y_1} * 
\mathscr{A}^{y_2}$.
Hence any three isomorphisms on the above free factors $\mathscr{A}^{y_0}$\!,
$\mathscr{A}^{y_1}$\!,
$\mathscr{A}^{y_2}$ (or on arbitrary subgroups inside them) have a common continuation.
On 
$\mathscr{A}^{y_0}$ and $\mathscr{A}^{y_2}$ choose the identity isomorphisms, 
and on the subgroup $D_1^{y_1}=\langle d_0^{y_1}\!,\, d_1^{y_1} \rangle$ of the  factor $\mathscr{A}^{y_1}$
choose the swapping isomorphism sending 
$d_0^{y_1}\!,\, d_1^{y_1}$ to 
$d_1^{y_1}\!,\, d_0^{y_1}$ respectively. 
Denote their common continuation on $\mathscr{A}^{y_0} * D_1^{y_1} * \mathscr{A}^{y_2}$ by $\gamma$. 
Clearly, $\gamma$ can be well defined by its values on $9+2+9=20$ generators in $X_{\! \mathscr{A}}^{y_0} 
\cup \big\{ 
d_0^{y_1}\!,\, d_1^{y_1}
\big\}\cup 
X_{\! \mathscr{A}}^{y_2}$.

\medskip
As a generating set for $\mathscr{B}$ one may choose: 
\begin{equation}
\label{EQ generators of B}
X_{\! \mathscr{B}} = X_{\! \mathscr{A}} \cup \big\{
x_0, x'_0, x_2, x'_2,\,
y_0, y_1, y_2
\big\}.
\end{equation}
Then in analogy with 
Point~\ref{NOR SU Auxiliary copy of A built here} 
we can choose its copy: 
\begin{equation}
\label{EQ generators of bar B}
X_{\! \bar{\mathscr{B}}} = X_{\! \bar{\mathscr{A}}} \cup \big\{
\bar x_0, \bar x'_0, \bar x_2, \bar x'_2,\,
\bar y_0, \bar y_1, \bar y_2
\big\},
\end{equation}
and using it build the copy 
$\bar {\mathscr{B}}$ of 
${\mathscr{B}}$ via a  procedure similar to that above. 
Inside the direct product $\bar {\mathscr{B}} \!\times\! \mathscr{B}$ choose the set of $20$ couples:
\begin{equation}
\label{EQ Couples in BxB}
\Big\{ \big(\bar x^{\bar y_0},x^{y_0}\big)
\mathrel{\;|\,} x\!\in X_{\!{\mathscr{A}}}
\Big\} 
\cup 
\big\{ (\bar d_0^{\bar y_1},d_1^{y_1}),\; 
(\bar d_1^{\bar y_1},d_0^{y_1})
\big\}
\cup 
\Big\{ \big(\bar x^{\bar y_2},x^{y_2}\big)
\mathrel{\;|\,} x\! \in X_{\!{\mathscr{A}}}
\Big\}.
\end{equation}
Correlation of this set  with the above function $\gamma$ is easy to notice:
since $\gamma$ is an identical map over $\mathscr{A}^{y_0}$, then the couples $(\bar x^{\bar y_0},x^{y_0})$ with $x\!\in X_{\! {\mathscr{A}}}$ are the couples
$\big(\bar a^{\bar y_0}, \gamma(a^{y_0})\big)$, $\big(\bar b^{\bar y_0} , \gamma(b^{y_0})\big), \ldots , \big(\bar e^{\bar y_0}, \gamma(e^{y_0})\big)$. Similarly for $\mathscr{A}^{y_2}$ we have the couples $\big(\bar a^{\bar y_2}, \gamma(a^{y_2})\big)$, $\big(\bar b^{\bar y_2}, $ $ \gamma(b^{y_2})\big)$, $ \ldots , \big(\bar e^{\bar y_2}, \gamma(e^{y_2})\big)$.
Lastly, since $\gamma$ just swaps $d_0^{y_1}$ and $d_1^{y_1}$, the central two couples in \eqref{EQ Couples in BxB} are equal to 
$\big(\bar d_0^{\bar y_1},\gamma(d_0^{y_1})\big)$,\;
$\big(\bar d_1^{\bar y_1},\gamma(d_1^{y_1})\big)$.
That is, in \eqref{EQ Couples in BxB} the second coordinate of each couple is the image of the first coordinate under $\gamma$, with just ``the bar removed''.

Recall that $\mathscr{B}_0$ and $\mathscr{B}_2$ were constructed so that \eqref{EQ D_0 obtained in two ways!!!} and \eqref{EQ D_2 obtained in two ways!!!} hold. Hence for any of 
$d_{\,-1}, d_{\,-2},\ldots$
from $D_0$
we have $d_{i}^{y_0}=d_{\,i}$; and for any of 
$d_{\,2}, d_{3},\ldots$
from $D_2$
we have $d_{i}^{y_2}=d_{\,i}$.
Similarly, $d_{0}^{y_1}=d_{0}$ and $d_{1}^{y_1}=d_{1}$, since 
$L_{D_1}=\langle d_0, d_1 \rangle$.
Thus, the subgroup $T$ generated by $20$ couples
\eqref{EQ Couples in BxB} contains the set of all the \textit{infinitely many} couples:
\begin{equation}
\label{EQ Couples with dxd}
\big\{ (\bar a,\; a)
\big\} 
\cup
\big\{ (\bar d_i,d_i)
\mathrel{\;|\;} i \in \Z \backslash \{0,1\} 
\big\} 
\cup  
\big\{ (\bar d_0,d_1),\; 
(\bar d_1,d_0)
\big\}
\end{equation}
which is nothing but $\big\{ \big(\bar a,\; \gamma(a)\big)
\big\} 
\cup
\big\{ \big(\bar d_i,\gamma(d_{i})\big)
\mathrel{\;|\;} i \in \Z 
\big\}$.

\medskip
Now we are ready to again adapt the idea from Point~\ref{SU Obtaining the benign subgroup for rho Q} for the operation $\tau$, i.e., to denote $P=\bar A_{\mathcal X} \times \langle a, d, e \rangle$
in  $\bar F_3 \!\times F_3$, and set
$Q=T \cap P$ to establish that this intersection has the simple structure 
$Q=\big\langle 
(\bar a_f,\; a_{\tau f}\!)
\;|\;  f\in \mathcal X
\big\rangle$.

$T$ contains $Q$, for, from \eqref{EQ Couples with dxd} we can deduce that  $T$ for every $f\!\!\in \mathcal E$ contains a specific element
$\lambda_f=\big(\bar d_f,\;  \tilde d_{f}  \big)$,
where $\tilde d_{f}$ is obtained from $d_{f}$ by replacing its two factors 
$d_0^{f(0)}$\!\!,\; $d_1^{f(1)}$ by the factors $d_1^{f(0)}$\!\!,\; $d_0^{f(1)}$ respectively.
Say, for $f=(2,5,3)$ we have 
$d_{f}=d_0^2 d_1^5 d_2^3$
and
$\tilde d_{f}=d_1^2 d_0^5 d_2^3$.
Then $\lambda_f$ indeed is in $T$ because the couples
$(\bar d_0,\; d_1)$,
$(\bar d_1,\; d_0)$,
$(\bar d_2,\; d_2)$
are in $T$, and so the product: 
$$
(\bar d_0,\; d_1)^2 
(\bar d_1,\; d_0)^5
(\bar d_2,\; d_2)^3
=
(\bar d_0^2 \bar d_1^5 \bar d_2^3,\; d_1^2 d_0^5 d_2^3)
=
(\bar d_f,\; \tilde d_{f})
=
\lambda_f
$$
is also in $T$.
Notice that  
$\tilde d_{f} = d_1^2 d_0^5 d_2^3$
differs from 
$d_{\tau f} = d_{(5,2,3)}= d_0^5 d_1^2  d_2^3$
in the \textit{order} of factors $d_0$ and $d_1$ only.
But since by Remark~\ref{RE order of d_i does not matter} the order of $d_i$ does not matter in action of $d_i$ on $a_f$, we get $a^{d_1^2 d_0^5 d_2^3}=a^{d_0^5 d_1^2  d_2^3}=a^{b_0^5 b_1^2  b_2^3}=a_{(5,2,3)}=a_{\tau f}$, and so
$(\bar a_f,\; a_f)^{\lambda_f}
=
(\bar a_f,\; a_{\tau f})
$
is in $T$.
Since also $Q \subseteq P$, we thus have  $Q \subseteq T \cap P$.

To get the reverse inclusion notice that any couple from $P$ has its first coordinate inside 
$\bar A_{\mathcal X}$, i.e., that coordinate is generated by some elements $\bar a_f 
= \bar a^{\bar b_f}
=\bar a^{\bar d_f}$.
On the other hand, our couple is also in $T$, and from \eqref{EQ Couples with dxd} it follows that  if its first coordinate is rewritten as a word on 
$\bar a,\bar d_i$, then the second coordinate in the same couple can be obtained by replacing 
$\bar a, \bar d_0, \bar d_1$ by $a, d_1, d_0$, and then $\bar d_i$ 
by $d_i$ for all $i\neq 0,1$. 
But this just transforms $\bar a^{\bar d_f}$ to $a^{\tilde d_{f}}=a^{b_{\tau f}}=a_{\tau f}$. Say, for $f=(2,5,3)$  the first coordinate is:
\begin{equation*}
\bar a_f 
= \bar a^{\bar b_f}
=\bar a^{\bar d_f}
= \bar a^{\bar d_{0}^{2} \bar d_{1}^{5} \bar d_{2}^{3}}
= \bar d_{2}^{\;-3}\bar d_{1}^{\;-5}  \bar d_{0}^{\;-2}\, \cdot \,
\bar a \,
\cdot\, \bar d_{0}^{2}\, \bar d_{1}^{5}\, \bar d_{2}^{3}
\end{equation*}
and then the second coordinate  has to be: 
\begin{equation*}
d_{2}^{\;-3}  d_{0}^{\;-5}    d_{1}^{\;-2}\, \cdot \,
a \,
\cdot\,   d_{1}^{2}\,   d_{0}^{5}\,   d_{2}^{3}
= a^{d_{1}^{2}\,  d_{0}^{5}\,  d_{2}^{3}}
=a^{b_{\tau f}}
=a_{\tau f}.
\end{equation*}
Thus, $Q$ has a simple description:
$$
Q = \big\langle 
\big(\bar a^{\bar b_f}\!,\; a^{ b_{\tau f}}\!\big)
\;|\;  f\in \mathcal X
\big\rangle
=\big\langle 
(\bar a_f,\; a_{\tau f}\!)
\;|\;  f\in \mathcal X
\big\rangle.
$$
Compare the above used elements $\tilde d_{f}$ and $\lambda_f$ to their similarly denoted, but yet slightly different, analogs in Point~\ref{SU Obtaining the benign subgroup for rho Q}.

\medskip
Slightly adapting the construction of Point~\ref{SU Construction of the direct product K x A} notice that $K_{\!\mathcal X}$ could be built to intersect with $\mathscr{B}$ in $F$ strictly. This allows us to define 
$\mathcal K = K_{\mathcal X} *_F \mathscr{B}$
(compare this with the group $\mathcal K$ in Point~\ref{SU Construction of the direct product K x A}), such that 
$\mathscr{B} \cap L_{\mathcal X} = A_{\mathcal X}$, that is, 
$A_{\mathcal X}$ is also benign in $\mathscr{B}$ for the finitely presented overgroup $\mathcal K$ and for the same finitely generated subgroup $L_{\!\mathcal X}$ mentioned above. 
The analog $\bar{\mathcal K}$ of $\mathcal K$ can be constructed for $\bar{\mathscr{B}}$ so that  
$P=\bar A_{\mathcal X} \times \langle a, d, e \rangle$ is benign in $\bar{\mathscr{B}} \!\times\! {\mathscr{B}}$ for 
$K_P=\bar {\mathcal K} \times \mathscr{B}$ and for $L_P=\bar L_{\!\mathcal X} \!\times \langle a, d, e \rangle \le K_P$.

The $20$-generator group $T$
is benign in $\bar{\mathscr{B}} \!\times\! {\mathscr{B}}$  for  $K_{T} = \bar{\mathscr{B}} \!\times\! {\mathscr{B}}$ and for $L_{T} = T$.
Hence, the intersection $Q$ is also benign in $\bar{\mathscr{B}} \!\times\! {\mathscr{B}}$ for the finitely presented overgroup:
\begin{equation}
\label{EQ K_Q defined for tau}
\begin{split}
K_Q &
= \big(K_T *_{L_T} v_1\big) 
*_{\bar{\mathscr{B}} \times {\mathscr{B}}}
\big(K_P *_{L_P} v_2\big)\\
&= \Big(
\big(\bar{\mathscr{B}} \times {\mathscr{B}}\big) 
*_{T} 
v_1 
\Big)
\, *_{\bar{\mathscr{B}} \times {\mathscr{B}}}
\Big(
\big(\bar {\mathcal K} \!\times\! \mathscr{B}\big)
*_{\bar L_{\!\mathcal X} \times\, \langle a, d, e \rangle} 
v_2 
\Big)
\end{split}
\end{equation}
and for its $32$-generator subgroup $L_Q\!=\big(\!\bar{\mathscr{B}} \!\times\! {\mathscr{B}}\big)^{v_1v_2}$\!.

But since $Q=\big\langle 
(\bar a_f,\; a_{\tau f}\!)
\;|\;  f\in \mathcal X
\big\rangle$ lies inside $\bar F_3 \!\times\! F_3$, then $Q$ is benign in $\bar F_3 \!\times\! F_3$ also for the same choice of $K_Q$ and $L_Q$.

\subsubsection{``Extracting'' $A_{\tau \mathcal X}$ from $Q$}
\label{SU Extracting A tau X from Q SHORT}

In analogy with Point~\ref{SU Extracting A rho X from Q} we ``extract'' the benign subgroup 
$A_{\tau \mathcal X}=
\big\langle a_{\tau f}
\;|\;  f\in \mathcal X
\big\rangle
$ from $Q$, skipping some explanation details below.

The join 
$
Q_1=\big\langle \bar F_3 \!\times\! \1 ,\, Q \big\rangle
=\bar F_3 \!\times\! \langle
a_{\tau f} \mathrel{|} 
f \in \mathcal X \rangle
$ is benign in $\bar F_3 \!\times\! F_3$ for the finitely presented:
$$
K_{Q_1}= 
\big((\bar F_3 \!\times\! F_3)\,*_{\bar F_3 \times \1} w_1\big) 
\,*_{\bar F_3 \times F_3}
\big(K_{Q} *_{L_{Q}} w_2\big)
$$
and for its $12$-generator subgroup:
$$
L_{Q_1} = \big\langle(\bar F_3 \!\times\! F_3)^{w_1}\!,\;(\bar F_3 \!\times\! F_3)^{w_2} \big\rangle.
$$
The intersection:
$$
A_{\tau \mathcal X}=A_{\mathcal Y}=
\big(\1 \!\times\!  F_3 \big) \cap Q_1 
=
\langle
a_{\tau f} \mathrel{|} 
f \in \mathcal X \rangle
$$ 
is benign in $\bar F_3 \!\times\! F_3$ for the finitely presented overgroup:
$$
K_{\tau \mathcal X}
=
K_{\mathcal Y}
=
\big((\bar F_3 \!\times\! F_3)\,*_{\1 \times  F_3} w_3\big) 
\,*_{\bar F_3 \times F_3}
(K_{Q_1} *_{L_{Q_1}} w_4)
$$
and for its $6$-generator subgroup
$L_{\tau \mathcal X}
=L_{\mathcal Y}
=(\bar F_3 \!\times\! F_3)^{w_3 w_4}$.
But since $A_{\tau \mathcal X}$ is inside $F_3$, it is benign in $F_3$ also, for the same choice of $K_{\tau \mathcal X}$ and $L_{\tau \mathcal X}$ made above.  

\subsubsection{Writing $K_{\tau \mathcal X}$ explicitly}
\label{SU Writing K tau X  by its generators and defining relations} 

Recalling the above constructions, in particular, the generating sets  
$X_{\! \mathscr{B}}$ and $X_{\! \bar{\mathscr{B}}}$
in
\eqref{EQ generators of B} and
\eqref{EQ generators of bar B}
we explicitly have:
\begin{equation}
\label{EQ K tau X}
\begin{split}
K_{\tau \mathcal X}=
\Big\langle
X_{\! \mathscr{B}},\;  X_{\! \bar{\mathscr{B}}};\;\;
\bar Z 
\backslash \{ \bar a, \bar b, \bar c \};\;
y_0, y_1, y_2;\;
v_1, v_2;\;\; w_1,\ldots , w_4 
\;\mathrel{|} \\
&\hskip-87mm 
R_{\!\mathscr{A}};\;
R_{\!\mathscr{\bar A}};\;
\bar S;\; \\
&\hskip-87mm 
d^{x_0}= d_{\,1},\;
d^{x'_0} = {d},\;\;
e^{x_0}=e^{x'_0}=e^2;\; \\
&\hskip-87mm 
d^{x_2}= d_{\,-1},\;
d^{x'_2} = {d_{\,-2}},\;\;
e^{x_2}=e^{x'_2}=e^2;\; \\
&\hskip-87mm 
\text{$y_0$ fixes $d_{\,-1}, x_0, x'_0$;\, 
$y_1$ fixes $d_0, d_1$;\,
$y_2$ fixes $d_2, x_2, x'_2$}; 
\\
&\hskip-87mm 
\bar d^{\bar x_0}= \bar d_{\,1},\;
\bar d^{\bar x'_0} = {\bar d},\;\;
\bar e^{\bar x_0}=\bar e^{\bar x'_0}=\bar e^{\;2};\; \\
&\hskip-87mm 
\bar d^{\bar x_2}= \bar d_{\,-1},\;
\bar d^{\bar x'_2} = {\bar d_{\,-2}},\;\;
\bar e^{\bar x_2}=\bar e^{\bar x'_2}=\bar e^{\;2};\; \\
&\hskip-87mm 
\text{$y_0$ fixes $d_{\,-1}, x_0, x'_0$;\, 
$y_1$ fixes $d_0, d_1$;\,
$y_2$ fixes $d_2, x_2, x'_2$}; 
\\
&\hskip-87mm 
\text{$X_{\! \mathscr{B}}$ commutes with  $X_{\! \bar{\mathscr{B}}}$ and $\bar Z 
\backslash \{ \bar a, \bar b, \bar c \}$};\; \\
&\hskip-87mm 
\text{$v_1$ fixes $20$ couples \eqref{EQ Couples in BxB};\; 
$v_2$ fixes  
$\bar L_{\!\mathcal X}$ and $a,d,e$;
}
\\
&\hskip-87mm 
\text{$w_1$ fixes $\bar a,\bar b,\bar c$;\; $w_2$ fixes  
$X_{\! \mathscr{B}}^{\! v_1 v_2} \cup  X_{\! \bar{\mathscr{B}}}^{\! v_1 v_2} $};\;\;
\text{$w_3$ fixes $a,b,c$};\\
&\hskip-87mm 
\text{$w_4$ fixes 
$\big\{ 
a,b,c, \bar a, \bar b, \bar c
a,b,c, \bar a, \bar b, \bar c
\big\}^{\! w_1}$}
\Big\rangle\,.
\end{split}
\end{equation}
As the finitely generated subgroup $L_{\tau \mathcal X}$ in $K_{\tau \mathcal X}$ we can explicitly take 
$\big\langle
a,b,c, \bar a, \bar b, \bar c
\big\rangle^{w_3 w_4}$\!\!.
In \eqref{EQ K tau X} the notation 
$X_{\! \mathscr{B}}^{\! v_1 v_2}$
stands for the set of conjugates of all generators from $X_{\! \mathscr{B}}$ by $v_1 v_2$; and  
$X_{\! \bar{\mathscr{B}}}^{\! v_1 v_2}$ is defined analogously.
In the $1$'st and $9$'th lines of  \eqref{EQ K tau X} we exclude $\bar a, \bar b, \bar c$ from $\bar Z$ because they were already included in $X_{\! \bar{\mathscr{B}}}$.
If $K_{\mathcal X}$ has $m$ generators and $n$ defining relations, and if 
$L_{\!\mathcal X}$ has $k$ generators, then the group $K_{\tau \mathcal X}$ in \eqref{EQ K tau X}
has 
$16+ 16+ (m-3)+3+2+4=m+38$ generators and 
$20+20+n+2\cdot (4+4+3+2+3)
+ 16\cdot (16+m-3)+20+(k+3)
+3 + 2\cdot 16 +3 +
2\cdot 6
= n + 16m + k + 345
$ defining relations.

\medskip

\subsection{The proof for the operation $\omega_m$}
\label{SU The proof for omega_m}

Assume the hypothesis of Theorem~\ref{TH Theorem A} holds for $\mathcal X$, 
the group $K_{\mathcal X}=\langle
\, Z \;|\;  S \rangle$ with its subgroup $L_{\mathcal X}\le K_{\mathcal X}$ are given explicitly, and denote $\mathcal Y = \omega_m \mathcal X$ for some $m=1,2,\ldots$\;,
see a simple example for $m=3$  in \eqref{EQ the f promised}.

Two agreements are going to simplify the proofs below.  
Firstly, since the names of free generators do not actually matter,  later we are going to suppose that $A_{\mathcal X}$ is benign in a free group of rank $3$ on some differently named generators $g,h,k$ which will be introduced below, compare with Section~\ref{SU Defining subgroups by integer sequences}.
Secondly, for the set $\mathcal X_m  \!=\! \mathcal X \cap \E_m$ it is trivial to notice that $\omega_m(\mathcal X)=\omega_m(\mathcal X_m)$, and so without loss of generality we may reduce our consideration to the case $\mathcal X \subseteq \E_m$. 
This, in particular, allows us to write all  $f$ in $\mathcal X = \mathcal X_m$  as sequences $f=(j_0,\ldots,j_{m-1})$. If a \textit{shorter} sequence contains less than $m$ integers, we can without loss of generality extend its length to $m$ by appending some extra $0$'s at the end, see Section~\ref{SU Integer functions f}.

\subsubsection{The groups $\Gamma$ and $\mathscr{G}$}
\label{SU The groups Xi Gamma}  
In the groups $\Xi_m$ and $\Xi_0$ we by Lemma~\ref{LE Ksi} have:
\begin{equation*} 
\begin{split}
\langle b,c \rangle \cap \langle b_m, t_m, t'_m\rangle = \langle b_m, b_{m+1},\ldots\rangle 
& \text{ in } \Xi_m,
\\
\langle b,c \rangle \cap \langle b_{-1}, t_0, t'_0\rangle = \langle b_{-1}, b_{-2},\ldots\rangle 
& \text{ in } \Xi_0.\\
\end{split}
\end{equation*} 
In analogy with the group $\mathscr{C}$ given in \eqref{EQ Defining C}, 
build the finitely presented $\textstyle{\bigast}$-construction:         
$$
\mathscr{Z}
= \left(\Xi_m *_{\langle b_m, t_m, t'_m\rangle} r_1 \right)
\, *_{\langle b,c\rangle}
\left(\Xi_0 *_{\langle b_{-1}, t_0, t'_0\rangle} r_2 \right).
$$
In $\langle b,c \rangle$ the subgroup 
$B_m = \langle \ldots b_{-2}, b_{-1};\;\;
b_m, b_{m+1},\ldots\rangle$ is the join of the above $\langle b_m, b_{m+1},\ldots\rangle$ and $\langle b_{-1}, b_{-2},\ldots\rangle$, and so by Corollary~\ref{CO intersection and join are benign multi-dimensional}  it is benign in $\langle b,c \rangle$. As a finitely presented overgroup of $\langle b,c \rangle$ one can take $K_{B_m}\!=\mathscr{Z}$, and as its finitely generated subgroup one can pick $L_{B_m}\!= P_m\!=\big\langle\langle b,c\rangle^{r_1} \!,\, \langle b,c\rangle^{r_2} \big\rangle$, see Corollary~\ref{CO intersection and join are benign multi-dimensional}\;\eqref{PO 2 CO intersection and join are benign multi-dimensional}. Also, check Figure~8 in 
\cite{Auxiliary free constructions for explicit embeddings} illustrating this construction.

The letters $g,h,k$ were not so far used, and we  may now involve them to build:
\begin{equation}
\label{EQ second Gamma}
\Gamma = \langle b,c \rangle *_{B_m} \!(g,h,k)
\quad {\rm and} \quad
\mathscr{G} = \mathscr{Z} *_{P_m}\!(g,h,k)
\end{equation}
with three stable letters $g,h,k$ all fixing the subgroups $B_m$ and $P_m$\! respectively. 
The second one of the groups \eqref{EQ second Gamma} clearly is finitely presented, because $\mathscr{Z}$ is finitely presented, and $P_m$ is finitely generated. 
Taking into account $\langle b,c\rangle
\cap\,
P_m\!
=B_m$,
and using $P_m$\! and $B_m$ as the groups $A$ and $A'$ 
of Corollary~3.5\;(1), and Remark~3.6
from 
\cite{Auxiliary free constructions for explicit embeddings}, we see that $\Gamma$ is a subgroup of $\mathscr{G}$, given as an intersection:   
\begin{equation}
\label{<EQ b,c,g,h,k> = Gamma}
\langle b,c, g,h,k\rangle
= \langle b,c \rangle *_{\langle b,c\rangle \;\cap\; P_m} \!(g,h,k)
=\Gamma.
\end{equation}

For this group, in analogy with the elements $b_i, b_f$ and $a_f$, define the  elements
$h_i=h^{k^i}$\!\!,\,
$h_f 
= \cdots
h_{-1}^{f(-1)}
h_{0}^{f(0)} 
h_{1}^{f(1)}\cdots$,\,
and
$g_f 
= g^{h_f}$\!, see Section~\ref{SU Defining subgroups by integer sequences}.   

\medskip
Since the particular names of free generators of $F_3$ do not actually matter, we may suppose the analog $G_{\mathcal X} = \langle g^{h_f}
\;|\;
f\in \mathcal X
\rangle$ of $A_{\mathcal X}$ is benign in the free group $F'_3=\langle g,h,k \rangle$ of rank $3$, and the respective finitely presented overgroup $K_{\mathcal X}=\langle
\, Z
\;|\;
S
\rangle$ 
of 
$F'_3$, 
and the finitely generated $L_{\mathcal X}\le K_{\mathcal X}$, with $F'_3 \cap L_{\mathcal X} = G_{\mathcal X}$, are explicitly given.
In the beginning of Section~\ref{SU The proof for omega_m} we introduced $K_{\mathcal X}, L_{\mathcal X}$ for $F_3$, but using the same symbols for $F'_3$ should cause no confusion.

\subsubsection{Construction of $\Delta$}
\label{SU Construction of Delta}

The free group $\langle b,c \rangle$ contains a free subgroup $\langle b_i \mathrel{|} i\in \Z\rangle$  of countable rank, which is a free product $B_m * \,\tilde B_{m}$ with
$B_m$ mentioned above, and with 
its $m$-generator ``complement''\;
$\tilde B_{m}\!=\!\langle b_0,\ldots b_{m-1}\rangle$.  
In $\Gamma$ pick the subgroup $R=\langle g_f b_f^{-1} \mathrel{|} f\in \mathcal E_m\rangle$, and since the letter $a$ was \textit{not} involved in construction of $\Gamma$ or of $\mathscr{G}$, we can consider it a  new stable letter to build the HNN-extension $\Gamma *_R a$ (shortly we will see that in $\Gamma *_R a$ three elements $a,b,c$ are free generators for $\langle a,b,c\rangle$, and so we have no conflict with the above usage of $F_3=\langle a,b,c\rangle$ as a free group of rank $3$).

The intersection $\langle b,c \rangle \cap R$ is trivial because the non-trivial words of type $g_f b_f^{-1}$ generate $R$ freely, and so any non-trivial word they generate must involve in its normal form at least one stable letter $g$, and hence it need to be outside $\langle b,c \rangle$.
Therefore by Corollary 3.5\;(1) in \cite{Auxiliary free constructions for explicit embeddings} the subgroup generated in $\Gamma *_R a$ by $\langle b,c \rangle$ together with  $a$, is equal to: 
$$
\langle b,c \rangle *_{\langle b,c \rangle \,\cap \,R} a = \langle b,c \rangle *_{\{1\}} a
= \langle b,c \rangle *  a=\langle a,b,c \rangle=F_3.
$$
Hence, $a,b,c$ are \textit{free} generators, and the map sending $a,b,c$ to $a,b^{c^m}\!\!\!\!\!,\;c$ can be continued to an isomorphism  $\rho$ from $F_3$ to its subgroup.  Identifying $\rho$ to a  stable letter $r$ we arrive to:
\begin{equation}
\label{EQ nested form}
\Delta = \big(\Gamma *_R a \big) *_\rho r
= \Big(\big(\langle b,c \rangle *_{B_m} (g,h,k)\big) *_R a \Big) *_\rho r.
\end{equation}

\subsubsection{Discovering the subgroup $A_{\omega_m \mathcal X}$\! inside $\Delta$}
\label{SU Obtaining G cup} 

Introduce the subgroup $W_{\mathcal X}=\langle g_f\!,\, a,\, r \mathrel{|} f\in \mathcal X\rangle$ in $\Delta$. The objective of this section is to prove:

\begin{Lemma} In the above notation the following equality holds:
\label{LE intersection F and W}
\begin{equation}
\label{EQ intersection}
F_3\cap W_{\mathcal X}= A_{\omega_m \mathcal X}.
\end{equation}
\end{Lemma} 

Uncomplicated routine of the proof of the lemma will follow from a series of simple observations and examples below. 
Firstly, by the agreement in the beginning of Section~\ref{SU The proof for omega_m} we may suppose $\mathcal X \subseteq\E_m$, and all functions in $\mathcal X$ are of the form $f\!=(j_0,\ldots,j_{m-1})$.

For arbitrary sequence  $h\!\in \omega_m \mathcal X$ the element $a_h\!=a^{b_h}$ is inside $W_{\mathcal X}$. 
Let us display this almost trivial fact by a routine step-by-step construction example. 
Let $m=3$, and let  $(7,2,4)$, $(2,5,3)$ be in $\mathcal X$.
Then the set $\omega_3 \mathcal X$ contains the sequence, say, 
\begin{equation}
\label{EQ the f promised}
h=(
0,0,0,\;\;
7,2,4,\;\;
0,0,0,\;\;
0,0,0,\;\;
0,0,0,\;\;
2,5,3,\;\;
7,2,4).
\end{equation}
To show that $a_h = a^{b_h}$ is in $W_{\mathcal X}$ we start by the initial sequences 
$l_1=(7,2,4)$ and 
$l_2=(2,5,3)$ 
in $\mathcal X$, and then use them via a few steps to arrive to the sequence $h$ above. 

In these steps we are going to use the evident fact that the relation $(g_f^{\vphantom8} b_f^{-1})^{\,a}=g_f^{\vphantom8} b_f^{-1}$  is equivalent to $a^{g_f\vphantom{j}}\!=a^{b_f}$.

\vskip-1mm
\textit{Step 1.} Since  $l_1=(7,2,4)\in \mathcal X$, then $g_{l_1}\!\! \in W_{\mathcal X}$, and so 
$
a^{g_{l_1}}\!\!=
a^{b_{l_1}}=
a^{b_0^{7}\,b_1^{2}b_2^{4}}
\in W_{\mathcal X}$, see Lemma~\ref{LE action of d_m on f} and Remark~\ref{RE order of d_i does not matter}. 

\textit{Step 2.} 
Since
$b_i^r=b_i^\rho\!
=(b^\rho)^{(c^i)^{\,\rho}}\!
=(b^{c^3})^{c^i}\!
=b^{c^{i+3}}\!
=b_{i+3}$, then conjugating the previously constructed element $a^{b_{l_1}}$ by $r$ we get:
$$
\big(a^{b_{l_1}}\!\big)^{r}\!\!=
\big(a^r\big)^{(b_0^{7}\,b_1^{2}b_2^{4})^{\,r}}\!\!=
a^{b_3^{7}\,b_4^{2}b_5^{4}}
=
a^{b_0^{0}\,b_1^{0}b_2^{0} \;\cdot\; b_3^{7}\,b_4^{2}b_5^{4}}=
a^{b_{l_3}}\in W_{\mathcal X}
$$
for the sequence $l_3=(0,0,0,\;
7,2,4)$.

Next, conjugating $a^{b_{l_3}}$ by 
$g_{l_2}$ we have:
$$
\big(\! a^{b_{l_3}}\big)^{g_{l_2}}
\!=\, 
a^{b_{l_3}\cdot \,g_{l_2}}
= a^{b_3^{7}\,b_4^{2}b_5^{4} \;\cdot \,g_{l_2}}
.$$

\textit{Step 3.} Each of the stable letters $g,h,k$ commutes with any 
$b_i$ for $i <0$ or $i\ge m=3$, 
so $g_{l_2}$ commutes with $b_3^{7}\,b_4^{2}b_5^{4}$, and then:
$$
a^{b_3^{7}\,b_4^{2}b_5^{4} \;\cdot \,g_{l_2}}
=
a^{\,g_{l_2} \cdot \; b_3^{7}\,b_4^{2}b_5^{4} }.
$$
Then for one more time applying Step 1 to $a^{g_{l_2}} $ we transform the above to:
\vskip-3mm
$$
a^{\,g_{l_2} \cdot \; b_3^{7}\,b_4^{2}b_5^{4} }
=
a^{b_0^{2}\,b_1^{5}b_2^{3} 
\;\cdot\; 
b_3^{7}\,b_4^{2}b_5^{4}}=
a^{b_{l_4}}
$$
for the sequence $l_4=(2,5,3,\;\;
7,2,4)$.
Next, we repeat Step 2 for the above $a^{b_{l_4}}$ for \textit{four times} i.e., conjugate the above by $r^4$ to get 
the element $a^{b_{l_5}}$ for the sequence:
$$
l_5=(0,0,0,\;\;
0,0,0,\;\;
0,0,0,\;\;
0,0,0,\;\;
2,5,3,\;\;
7,2,4)
.$$
Next apply Step 3 and Step 1 again to conjugate 
$a^{b_{l_5}}$ by 
$g_{l_1}$. We get the element 
$a^{b_{l_6}}$ for the sequence:
$$
l_6=(7,2,4,\;\;
0,0,0,\;\;
0,0,0,\;\;
0,0,0,\;\;
2,5,3,\;\;
7,2,4)
.$$
Then we again apply Step 2, i.e., conjugate $a^{b_{l_6}}$  by $r$ to construct in $ W_{\mathcal X}$ the element $a^{b_{h}}=a_h$
with the sequence $h$ promised in \eqref{EQ the f promised} above. 

Since such a procedure can easily be performed for random $m$ and for an \textit{arbitrary} $h \in \omega_m \mathcal X$, we get that $A_{\omega_m \mathcal X} \le W_{\mathcal X}$.
Since also $A_{\omega_m \mathcal X} \le F_3$, we then have $A_{\omega_m \mathcal X} \le F_3 \cap W_{\mathcal X}$.

\bigskip 
Next assume some word $w$ from $W_{\mathcal X}=\langle g_f\!,\, a,\, r \mathrel{|} f\in \mathcal X \rangle$ is in $F_3$, and deduce from \eqref{EQ nested form} that it is in $A_{\omega_m \mathcal X}$ necessarily.

Since $w$ is also in $\Delta$, it can be brought to its normal form involving stable letter $r$ and some elements from $\Gamma *_R a$.
The latter elements, in turn, can be brought to normal forms involving stable letter $a$ and some elements from $\Gamma$. 
Then the latters can further be brought to normal forms involving stable letters  $g, h,k$ and some elements from $\langle b,c\rangle$. 
That is, $w$ can be brought to a \textit{``nested'' normal form} reflecting three
``nested'' HNN-extensions in the right-hand side of \eqref{EQ nested form}.
Let us detect the cases when it involves nothing but the letters $a,b,c$.
The only relations of $\Gamma$ involve $g,h,k$, and they are equivalent to  $a^{g_f}\!=a^{b_f}$. Thus, the only way by which $g,h,k$ may be eliminated in the above normal form  is to have in $w$ subwords of type $g_f^{\!-1} a\, g_f^{\vphantom8} = a^{g_f}$\! which can be replaced by respective subwords $a^{b_f}\!\in F_3$. 
If after this procedure some subwords $g_f$ still remain, then three scenario cases are possible: 

\textit{Case 1}. 
The word $w$ may contain a subword of type $w'=g_f^{-1}  a^{b_l}  g_f^{\vphantom8}$ for such an $l$ that $l(i)=0$ for $i=0,\ldots,m\!-\!1$.
Check the example of Step 1, when this is achieved for $l=l_3=(0,0,0,\;
7,2,4)$ and $f=l_2=(2,5,3)$. 
Then just replace $w'$ by $a^{b_{l'}}$ for an $l'\in \omega_m \mathcal X$ (such as $l'=l_4=(2,5,3,\;\;
7,2,4)$ in our example). 

\textit{Case 2}. If  $w'=g_f^{-1}  a^{b_l}  g_f$,\, but the condition $l(i)=0$ fails for an $i=0,\ldots,m-1$, then $g_f$ does \textit{not} commute with $b_l$, so we cannot apply the relation $a^{g_f}\!=a^{b_f}$\!, and so $w \notin G $.
Turning to  example in steps 1--3, notice that for, say, $f=(7, 2, 4) \in \mathcal X$ we may \textit{never} get something like
$a^{(g_f)^{\,2} }\!\!
=\big(a^{b_0^{7}\,b_1^{2}\,b_2^{4}}\big)^{g_f}\!\!
=a^{(b_0^{7}\,b_1^{2}\,b_2^{4})^2}
$ because $g_f$ does not commute with $b_0,  b_1,  b_2$. That is, all the \textit{new} functions $l$ we get \textit{exclusively} are from $\omega_m \mathcal X$.

\textit{Case 3}. If $g_f$ is in $w$, but is not in a subword $g_f^{-1}  a^{b_l}  g_f$, we again have $w \notin F$, unless all such $g_f$ trivially cancel each other.

This means, if $w \in F_3$, then  elimination of  $g,h,k$  turns $w$ to a product of elements from  $\langle r\rangle$ and of some  $a^{b_f}$ for some $f\in \omega_m \mathcal X$ ($a$ is also of that type, as $(0)\in\mathcal X$).
Now apply    \ref{SU The conjugates collecting process}  for 
$\X=\{a^{b_f} \mathrel{|} f\in \omega_m \mathcal X \}$
and $\Y=\{r\}$ to state that
$w$ is a product of some power $r^i$ and of some elements each of which is an $a^{b_f}$ conjugated by a power $r^{n_i}$ of $r$.
These conjugates certainly are in $\omega_m \mathcal X$ (see Step 2 above), and so $w\in F_3$ holds if and only if $i=0$, i.e., if 
$w\in A_{\omega_m \mathcal X}$.

Therefore, $F_3\cap W_{\mathcal X} \le A_{\omega_m \mathcal X}$ holds, and equality \eqref{EQ intersection} has been proved.

\medskip
The equality \eqref{EQ intersection} does not \textit{yet} mean that $A_{\omega_m \mathcal X}$ is benign in $F_3$ because the group $\Delta$
of \eqref{EQ nested form} may \textit{not} necessarily be finitely presented,
and $W_{\mathcal X}$ may \textit{not} necessarily be finitely generated. Our near objective is to replace $\Delta$ by a  finitely presented alternative $\mathscr{D}$ in which these two ``defects'' are corrected, see Point~\ref{SU Construction of finitely presented DD} below.

\subsubsection{Presenting $R$ as a join}
\label{SU Presenting as a join} 

Let us present $R$ as a join of $m+1$ subgroups in $\Gamma$, each benign in  $\mathscr{G}$.  
Denote $\Phi_m=\langle b_0,\ldots,b_{m-1}, g, h_0,\ldots,h_{m-1} \rangle$, and notice that:

\begin{Lemma}
\label{LE new free subgroup}
$\Phi_m$ is freely generated by  $2m\!+\!1$ elements $b_0,\ldots,b_{m-1}, g, h_0,\ldots,h_{m-1}$\! in $\Gamma$, and hence in $\mathscr{G}$.
\end{Lemma} 

\begin{proof}  

Firstly, $\tilde B_{m}=\langle b_0,\ldots, b_{m-1}\rangle$ has trivial intersection with $P_m$ because $\langle b,c\rangle
\cap\,
P_m\! =B_m$ implies  
$\tilde B_{m} \cap\,
P_m 
\le \big(\tilde B_{m} \cap\langle b,c\rangle\big) \cap
P_m 
=\tilde B_{m} \cap\, \big(\langle b,c\rangle \cap P_m\big)
=\tilde B_{m} \cap\,B_{m}
= \1$.
Therefore, in $\mathscr{G}$ we 
by Corollary~3.5\;(1) 
and by Remark~3.6 
in \cite{Auxiliary free constructions for explicit embeddings} have:
$$
\langle b_0,\ldots,b_{m-1},\; g,h,k\rangle
=\tilde B_{m} *_{\tilde B_{m}   \,\cap\;   P_m} (g,h,k) = \tilde B_{m} *_{{\1}} (g,h,k)
=\tilde B_{m} * \langle g,h,k \rangle
$$
which simply is a free group of (rank $m+3$).
Since $h_0,\ldots,h_{m\,-1}$ generate a free subgroup (of rank $m$) inside $\langle g,h,k \rangle$, they together with $g$ and with $b_0,\ldots,b_{m\,-1}$ generate a free subgroup (of rank $m+1+m=2m+1$)
inside $\langle b_0,\ldots,b_{m\,-1},\; g,h,k\rangle
\le \Gamma
\le \mathscr{G}
$.
\end{proof}   

Next we need a series of auxiliary benign subgroups in $\mathscr{G}$. Namely, for an integer $s=1,\ldots,m$ and for a sequence $f=(j_0, \ldots, j_{s-2}, j_{s-1})\in \mathcal E_s$ following the notation in Section~\ref{SU Integer functions f} write $f^+=(j_0, \ldots, j_{s-2},\, j_{s-1}\!+\!1)$.
In this notation for any $f$ the group $\mathscr{G}$ contains the elements 
$g_{f^+}^{\vphantom{1}} \!
\cdot 
b_{s-1}^{-1} 
\cdot 
g_f^{-1}$, such as, 
$g^{h_0^{2} h_1^{5} h_2^{3} h_3^{\boldsymbol{8}} }
\cdot
b_{3}^{-1} 
\cdot
g^{-\,h_0^{2} h_1^{5} h_2^{3} h_3^{7} }
$  
for the tuple $f\!=\!(2,5,3,7)$ of the length $s=4$.
Denote:
\begin{equation*}
\label{EQ Ksi 0 and m}
\begin{split}
V_{\mathcal E_s} &= \Big\langle 
g_{f^+}^{\vphantom{1}} \!
\cdot 
b_{s-1}^{-1} 
\cdot 
g_f^{-1} \mathrel{|} f\in \mathcal E_s
\Big\rangle
\\
&= \Big\langle 
g^{h_0^{i_0} \cdots\, h_{s-2}^{i_{s-2}}h_{s-1}^{(i_{s-1}\boldsymbol{+1})}} 
\!
\cdot 
b_{s-1}^{-1} 
\cdot 
g^{-h_0^{i_0} \cdots\, h_{s-2}^{i_{s-2}}h_{s-1}^{i_{s-1}}} \mathrel{|} i_{0}\ldots,i_{s-2},i_{s-1} \in \Z
\Big\rangle.
\end{split}
\end{equation*}

\begin{Lemma}
\label{LE VEm is bening}
In the above notation each
$V_{\mathcal E_s}$, $s=1,\ldots,m$, is a benign subgroup in $\mathscr{G}$ for some explicitly given finitely presented group and its finitely generated subgroup.
\end{Lemma}

\begin{proof}  
By Lemma~\ref{LE new free subgroup} 
the elements $b_{s-1}, g, h_0,\ldots, h_{s-1}$
are \textit{free} generators for the $(s+2)$-generator subgroup
$\langle b_{s-1}, g, h_0,\ldots, h_{s-1} \rangle$
of $\Phi_m$. Thus, any of the following maps $\lambda_{i,\,j}$ can be continued to an isomorphism on the free group $\langle b_{s-1}, g, h_0,\ldots, h_{s-1} \rangle$:
\begin{equation}
\label{EQ lambda definitions}
\begin{aligned}
& \lambda_{s-1,\,0} 
& \!\!\!\!{\rm sends} \hskip3mm
& b_{s-1}, g, h_0,\ldots, h_{s-2}, h_{s-1}
& {\rm to} \hskip3mm 
& b_{s-1}, g^{h_0}, h_0,\ldots, h_{s-2}, h_{s-1};\\
& \lambda_{s-1,\,1}
& {\rm sends} \hskip3mm
& b_{s-1}, g, h_0,\ldots, h_{s-2}, h_{s-1}
& {\rm to} \hskip3mm 
& b_{s-1}, g^{h_1}, h_0^{h_1},\ldots, h_{s-2}, h_{s-1};\\
& \hskip5mm\vdots & & \hskip19mm\vdots &  & \hskip22mm\vdots \\
& \lambda_{s-1,\, s-1} 
& {\rm sends} \hskip3mm
& b_{s-1}, g, h_0,\ldots, h_{s-2}, h_{s-1}
& {\rm to} \hskip3mm 
& b_{s-1}, g^{h_{s-1}}, h_0^{h_{s-1}},\ldots, h_{s-2}^{h_{s-1}}, \; h_{s-1}.
\end{aligned}
\end{equation}  
Say, for $m\!=\!1$ the map $\lambda_{0,0}$ sends  
$b_0, g, h_0$ to $b_0, g^{h_0}, h_0$;\,
while for $m\!=\!2$ the map
$\lambda_{1,0}$ sends  
$b_1, g, h_0, h_1$ to $b_1, g^{h_0}, h_0, h_1$ and 
$\lambda_{1,1}$ sends  
$b_1, g, h_0, h_1$ to $b_1, g^{h_1}, h_0^{h_1}, h_1$, 
etc... 

For these isomorphisms $\lambda_{i,\,j}$ respectively pick certain stable letters $l_{i,\,j}$ to construct:
$$
\Lambda_s=
\mathscr{G}
*_{\,\lambda_{s-1,\; 0},\;\ldots\;,\,\lambda_{s-1,\; s-1}} (l_{s-1,\; 0},\ldots,l_{s-1,\; s-1})
$$
for all the values $s=1,\ldots,m$.

\medskip

The effects of conjugation by elements $l_{s-1,\; 0},\ldots,l_{s-1,\; s-1}$ on the products $g_{f^+}^{\vphantom{1}} \!
\cdot 
b_{s-1}^{-1} 
\cdot 
g_f^{-1}$
is very easy to understand: $l_{s-1,\; i}$ just adds $1$ to the $i$'th coordinate of $f$, say, for $s=4$, $f=(2,5,3,7)$ and $l_{3,\,2}=l_{4-1,\,3-1}$ we have: 
\begin{equation}
\label{EQ example with l}
\begin{aligned}
\left(
g_{f^+}^{\vphantom{1}} \!
\cdot 
b_{3}^{-1} 
\cdot 
g_f^{-1}
\right)^{l_{3,\,2}}
&=\left(g^{h_{2}}\right)^{
\left(h_0^{h_{2}}\right)^{2}
\left(h_1^{h_{2}}\right)^{5} h_2^{3} \,h_3^{\boldsymbol{8}} }
\cdot
b_{3}^{-1} 
\cdot
\left(g^{h_{2}}\right)^{-\,
\left(h_0^{h_{2}}\right)^{2}
\left(h_1^{h_{2}}\right)^{5}
h_2^{3}\, h_3^{7} } \\
& =  g^{h_2 \cdot\, h_2^{-1}  h_0^2 h_2 
h_2^{-1} h_1^5 h_2
h_2^{3} h_3^{\boldsymbol{8}} }
\cdot
b_{3}^{-1} 
\cdot
g^{-\,h_2 \cdot\, h_2^{-1}  h_0^2 h_2 
h_2^{-1} h_1^5 h_2
h_2^{3} h_3^{7} }\\
& = g^{h_0^{2} h_1^{5} h_2^{\boldsymbol{4}} h_3^{\boldsymbol{8}} }
\cdot
b_{3}^{-1} 
\cdot
g^{-\,h_0^{2} h_1^{5} h_2^{\boldsymbol{4}} h_3^{7} }
=g_{f'^{\;+}}^{\vphantom{1}} \!
\cdot 
b_{3}^{-1} 
\cdot 
g_{f'}^{-1}
\in V_{\mathcal E_4}
\end{aligned}
\end{equation}
where $f'=(2,5,3+1,7)=(2,5,\boldsymbol{4},7)$.
In particular, actions of the above letters $l_{i,\,j}$ keep the elements from $V_{\mathcal E_s}$ inside $V_{\mathcal E_s}$.

\medskip

We in \ref{SU Integer functions f} agreed that for our purposes we may concatenate zero entries to any sequence $f$ without changing the respective elements $b_f, a_f, h_f, g_f$. Hence, interpret the zero sequence as $f_{0}=(0,\ldots, 0)\in \E_s$, and rewrite the product $g^{h_{1}} 
\!
\cdot 
b_{1}^{-1} 
\cdot 
g^{-1}$ as
$g_{f_0^+}^{\vphantom{1}} \!
\cdot 
b_{1}^{-1} 
\cdot 
g_{f_0}^{-1}
$.

Applying \eqref{EQ elements <X,Y>} for the sets $\X=\{
g^{h_{s-1}}  
\!
\cdot 
b_{s-1}^{-1} 
\cdot 
g^{-1}\}
$ and  
$\Y=\{l_{s-1,\; 0},\ldots,l_{s-1,\; s-1}\}$
we see that any element $w$ from $\langle \X
, \Y \rangle \le \Lambda_s$ is a product of elements of $g_{f^+}^{\vphantom{1}} \!
\cdot 
b_{s-1}^{-1} 
\cdot 
g_f^{-1}$ (for certain sequences $f\in \mathcal E_s$) and of certain powers of the stable letters $l_{s-1,\; 0},\ldots,l_{s-1,\; s-1}$. The word $w$ is inside $\mathscr{G}$ if and only if all those powers are cancelled out in the normal form, and $w$ in fact is in $V_{\mathcal E_s}$, that is, denoting $L_s = \langle
g^{h_{s-1}} 
\!
\cdot 
b_{s-1}^{-1} 
\cdot 
g^{-1}\!{\boldsymbol,}\;\;\; l_{s-1,\; 0},\ldots,l_{s-1,\; s-1}
\rangle$ we have 
$\mathscr{G}\, \cap \, L_s
\subseteq V_{\mathcal E_s}$.

On the other hand, for \textit{any} $f \in \E_s$ it is very easy to obtain $g_{f^+}^{\vphantom{1}} \!
\cdot 
b_{s-1}^{-1}  
\cdot 
g_f^{-1}$ via conjugations of  $g^{h_{s-1}} 
\!
\cdot 
b_{s-1}^{-1} 
\cdot 
g^{-1}$ by the free letters $l_{s-1,0},\ldots,l_{s-1,\;s-1}$, see \eqref{EQ example with l}. For instance, for $f=(2,5,3,7)$ we compute:
$$
g_{f^+}^{\vphantom{1}} \!
\cdot 
b_{3}^{-1}  
\cdot 
g_f^{-1} 
= 
\left(
g^{h_{3}} 
\!
\cdot 
b_{3}^{-1} 
\cdot 
g^{-1}
\right)^{\;
l_{3,0}^2 \, \cdot \;
l_{3,1}^5 \, \cdot \;
l_{3,2}^3 \, \cdot \;
l_{3,3}^7
}\!.
$$
Therefore, $\mathscr{G}\, \cap \, L_s
= V_{\mathcal E_s}$, 
i.e., $V_{\mathcal E_s}$ is benign in $\mathscr{G}$ for the above finitely presented group $\Lambda_s$ and for its $(s+1)$-generator subgroup $L_s$.
\end{proof}   

In addition to the groups given in the above proof set the auxiliary groups  
$V_{\mathcal E_0}=L_0=\langle g \rangle$ and $\Lambda_0=\mathscr{G}$. Since this $V_{\mathcal E_0}$ already is finitely generated, it trivially is benign in finitely generated $\mathscr{G}$ for the stated finitely presented $\Lambda_0$ and for the finitely generated $L_0$, see Remark~\ref{RE finite generated is benign}. 

\medskip
The collected information outputs:

\begin{Lemma}
\label{LE represent L}
$R=\langle g_f b_f^{-1} \mathrel{|} f\in \mathcal E_m\rangle$  is a benign subgroup in $\mathscr{G}$ for  some explicitly given finitely presented group and its finitely generated subgroup.
\end{Lemma}

\begin{proof}

First show that
$R$
is generated by its $m\!+\!1$ subgroups $V_{\mathcal E_0}, V_{\mathcal E_1}, \ldots, V_{\mathcal E_m}$.
For each $s=1,\ldots,m$
denote  
$Z_{\mathcal E_s}=\langle g_f b_f^{-1} \mathrel{|} f\in \mathcal E_s\rangle$, and also set $Z_{\mathcal E_0}=\langle g \rangle$.
In this notation $R$ is nothing but $Z_{\mathcal E_m}$ for $s=m$.
It is easy to see that 
$\langle Z_{\E_{s-1}}, V_{\E_s} \rangle = Z_{\E_s}$ for each $s$, see details in \cite{The Higman operations and  embeddings} using an original idea from \cite{Higman Subgroups of fP groups}.
Then:
$$
Z_{\E_{m}} 
= \langle Z_{\E_{m-1}}, V_{\E_{m}} \rangle 
= \langle Z_{\E_{m-2}}, 
V_{\E_{m-1}}, V_{\E_{m}} \rangle
= \cdots =
\big\langle V_{\E_{0}}, V_{\E_{1}}, \ldots, V_{\E_{m}} \big\rangle.
$$ 

\smallskip
By Lemma~\ref{LE VEm is bening} each
$V_{\mathcal E_s}$, $s=1,\ldots,m$, is benign in $\mathscr{G}$ for an explicitly given finitely presented group $\Lambda_s$ and its finitely generated subgroup $L_s$.
For $m=0$ the subgroup $V_{\mathcal E_0}=\langle g \rangle$ is benign in $\mathscr{G}$, as remarked above.

It remains to load these components into 
the $\bigast$-construction
\eqref{EQ star construction short form}, and to apply Corollary~\ref{CO intersection and join are benign multi-dimensional} to get the following finitely presented overgroup holding $\mathscr{G}$:
\begin{equation}
\label{EQ adapted construction with Theta} 
\mathscr{F}=
\Big(\cdots
\Big( \big( (\Lambda_0 *_{L_0} p_0) *_{\mathscr{G}} (\Lambda_1 *_{L_1} p_1) \big) *_{\mathscr{G}}  (\Lambda_2 *_{L_2} t_2)\Big)\cdots 
\Big) *_{\mathscr{G}}  (\Lambda_m *_{L_m} p_m),
\end{equation}
and its finitely generated subgroup 
$
\mathscr{H}
=\big\langle \mathscr{G}^{\,p_0}
,\ldots,
\mathscr{G}^{\,p_m}
\big\rangle$ for which $\mathscr{G} \cap \mathscr{H} = R$ holds.
\end{proof}

\subsubsection{Construction of finitely presented $\mathscr{D}$ and $K_{\omega_m \mathcal X}$}
\label{SU Construction of finitely presented DD}    

Observe that in construction of $\mathscr{F}$ we never used the letter $a\in G$. Hence, in analogy with construction of $\Gamma *_R a$ in  Point~\ref{SU Construction of Delta}, we can build the HNN-extension $\mathscr{F} *_\mathscr{H} a$ using $a$ as a stable letter fixing $\mathscr{H}$. 
Since $\mathscr{F}$ of \eqref{EQ adapted construction with Theta} is finitely presented, and $\mathscr{H}$ is finitely generated, $\mathscr{F} *_\mathscr{H} a$ is finitely presented.    

Inside $\mathscr{F} *_\mathscr{H} a$ the elements $a,b,c$ generate the same \textit{free} subgroup discussed in Point~\ref{SU Construction of Delta}, and so we can again define an isomorphism $\rho$ sending $a,b,c$ to $a,b^{c^m}\!\!\!\!\!,\;c$ together with the 
\textit{finitely presented} analog $\mathscr{D}$ of $\Delta$ from    
\eqref{EQ nested form}:
\begin{equation}
\label{EQ nested form analog}
\mathscr{D}
= \big(\mathscr{F} *_\mathscr{H} a \big) *_\rho r
= \Big(
\big(\bigast_{i=0}^{m}(\Lambda_i, L_i, p_i)_{\mathscr{G}}\big)
*_\mathscr{H} a \Big) *_\rho r.
\end{equation}
For any $\mathcal X \subseteq \E_m$ we in analogy with Point~\ref{SU Obtaining G cup} introduce $W_{\mathcal X}=\langle g_f\!,\, a,\, r \mathrel{|} f\in \mathcal X \rangle$ in $\mathscr{D}$. But since $W_{\mathcal X}$ is in the subgroup $\Delta$ of $\mathscr{D}$ already, we in $\mathscr{D}$ have the literal analog of \eqref{EQ intersection}:
\begin{equation*}
\label{EQ intersection analog}
F\cap W_{\mathcal X}= A_{\omega_m \mathcal X}.
\end{equation*}

\medskip
$\mathscr{D}$ was built via some free constructions 
by adjoining to  $\langle 
b,c
\rangle$ the new letters: 
\begin{equation} 
\label{EQ letters defined for Delta}
\begin{split}
& \hskip4mm t_m, t'_m, t_0, t'_0, r_1, r_2;\;
g,h,k;\; \\
& l_{s-1,\; 0},\ldots,l_{s-1,\; s-1}\;
(s\!=\!1,\ldots,m)
;\; \\
& \hskip11mm  p_0, p_1,\ldots, p_m;\; a, r.
\end{split} 
\end{equation}

For the benign subgroup $G_{\mathcal X} = \langle g^{h_f}
\;|\;
f\in \mathcal X
\rangle$ of $F'_3=\langle g,h,k\rangle$ we at the end of
Point~\ref{SU The groups Xi Gamma}
assumed to   explicitly have a finitely presented overgroup $K_{\mathcal X}=\langle
\, Z
\;|\;
S
\rangle$ of $F'_3$ with a finitely generated $L_{\mathcal X}$
such that $F'_3 \cap L_{\mathcal X} = G_{\mathcal X}$.\;
Since in construction of $\mathscr{D}$ we had the freedom to chose the new letters \eqref{EQ letters defined for Delta}, we may suppose all of them, except $g,h,k$, are \textit{not} used in $Z$ to define $K_{\mathcal X}=\langle
\, Z
\;|\;
S
\rangle$. This means
$K_{\mathcal X}$
and $\mathscr{D}$  intersect in $F'_3$ strictly, and hence, the finitely presented amalgamated product $K_{\mathcal X} *_{F'_3} \mathscr{D}$ can be defined.

\medskip
The subgroup 
$G_{\mathcal X}$ is benign in $\mathscr{D}$ also. Indeed,  since $F'_3 \cap L_{\mathcal X} = G_{\mathcal X}$
and
$F'_3 \cap \,G_{\mathcal X} = G_{\mathcal X}$, we can apply Corollary~3.2\;(3) in \cite{Auxiliary free constructions for explicit embeddings}  
to the subgroup $\Gamma' = \langle\, L_{\mathcal X} ,\, G_{\mathcal X} \rangle = L_{\mathcal X}$ of $\Gamma = K_{\mathcal X} *_{F'} \mathscr{D}$  
to get that  $\mathscr{D} \cap L_{\mathcal X} = G_{\mathcal X}$.
Being finitely generated $\langle a,r \rangle$ is benign in $\mathscr{D}$ for the finitely presented $\mathscr{D}$ and for the finitely generated $\langle a,r \rangle$, see Remark~\ref{RE finite generated is benign}.
Hence by Corollary~\ref{CO intersection and join are benign multi-dimensional}  
the join $ \big\langle G_{\mathcal X}, \langle a,r \rangle \big\rangle
=
\langle g_f\!,\, a,\, r \mathrel{|} f\in \mathcal X \rangle
=
W_{\mathcal X} 
$ is benign in $\mathscr{D}$.
As its finitely presented overgroup one may
by Lemma~\ref{LE join in HNN extension multi-dimensional} chose:
\begin{equation}
\label{EQ definition of L}
\mathscr{L}
=\big((K_{\mathcal X} *_{F'_3} \mathscr{D})*_{L_{\mathcal X}} q_1 \big) 
\;*_{\mathscr{D}}\,
\big(\mathscr{D} *_{\langle a,\, r \rangle} q_2\big),
\end{equation}
and as a finitely generated subgroup we may take $L' = \big\langle \mathscr{D}^{q_1}\!,\; \mathscr{D}^{q_2} \big\rangle$.

\medskip
$F_3$ is benign in $\mathscr{D}$ for the finitely presented $\mathscr{D}$ and for the finitely generated $F_3$.
Hence by
Corollary~\ref{CO intersection and join are benign multi-dimensional} 
and by \eqref{EQ intersection}
the intersection 
$F\cap W_{\mathcal X}= A_{\omega_m \mathcal X}$ is benign in $\mathscr{D}$ for the finitely presented $\bigast$-construction:
\begin{equation}
\label{EQ explicite KomegaB}
K_{\omega_m \mathcal X}=(\mathscr{L} *_{L'} q_3) \;*_{\mathscr{D}} \,
(\mathscr{D} *_{F_3} q_4)
\end{equation}
(generated by $Z$, by $b,c$, by the adjoined letters \eqref{EQ letters defined for Delta}
and by four new letters $q_1, q_2, q_3, q_4$),
and for the finitely generated subgroup:
$$
L_{\omega_m \mathcal X} = \mathscr{D}^{\,q_3 q_4}
$$ 
in the above $K_{\omega_m \mathcal X}$.
But since $F_3 \le \mathscr{D}$ and $A_{\omega_m \mathcal X} \le F_3$, we conclude that equality $F_3 \cap L_{\omega_m \mathcal X} = A_{\omega_m \mathcal X}$ also holds in $K_{\omega_m \mathcal X}$.
This concludes the proof of the promised fact that $A_{\omega_m \mathcal X}$ is benign in $F_3$. Below we have the respective groups $K_{\omega_m \mathcal X}$ and $L_{\omega_m \mathcal X}$
written vie free constructions.

\subsubsection{Explicitly writing $K_{\omega_m \mathcal X}$ by generators and defining relations}
\label{SU Writing K_omega B by generators and defining relations}
Using the definitions of  $\mathscr{Z}$ and $\mathscr{G}$ in Point~\ref{SU The groups Xi Gamma} we have:
\begin{equation}
\label{EQ relations Z}
\begin{split}
\mathscr{Z}
& = \big\langle
b, c, t_m, t'_m, t_0, t'_0, r_1, r_2 
\mathrel{\;\;|\;\;}
b^{t_m}= b^{c^{-m+1}}\!\!,\;
b^{t'_m}= b^{c^{-m}}\!\!,\;
\\
& \hskip9mm
b^{t_0}= b^{c}\!,\;
b^{t'_0}= b,\;\;
c^{t_m}=c^{t'_m}=c^{t_0}=c^{t'_0}=c^2;\\
& \hskip9mm
\text{$r_1$ fixes $b_m, t_m, t'_m$};\;\;\;
\text{$r_2$ fixes $b_{-1}, t_0, t'_0$}
\big\rangle.
\end{split}
\end{equation}
%
\begin{equation}
\label{EQ relations G}
\begin{split}
\mathscr{G}
& = \big\langle
b, c, t_m, t'_m, t_0, t'_0, r_1, r_2;\;
g,h,k
\mathrel{\,|\,}
\text{$14$ relations of $\mathscr{Z}$ from \eqref{EQ relations Z}};\\
& \hskip16mm
\text{$g,h,k$ fix 
$b^{r_1}\!,\; c^{r_1}\!,\; b^{r_2}\!,\; c^{r_2}$}
\big\rangle.
\end{split}
\end{equation}

\noindent 
Using 
$\Lambda_0, L_0, \;
\Lambda_1, L_1, \ldots, \Lambda_m, L_m$ and 
$\mathscr{F}$
in Point~\ref{SU Presenting as a join}, in particular, \eqref{EQ adapted construction with Theta}, write:
\begin{equation}
\label{EQ relations F}
\begin{split}
\mathscr{F}
& = \big\langle
b, c, t_m, t'_m, t_0, t'_0, r_1, r_2;\;\;
g,h,k;\; 
\\
& \hskip21mm
l_{s-1,\; 0},\ldots,l_{s-1,\; s-1}\;(s=1,\ldots,m)
;\;\;
p_0, p_1,\ldots, p_m
\mathrel{\;|\;}\\
& \hskip11mm
\text{$14$ relations of $\mathscr{Z}$ from \eqref{EQ relations Z}};\\
& \hskip11mm\text{$g,h,k$ fix 
$b^{r_1}\!,\; c^{r_1}\!,\; b^{r_2}\!,\; c^{r_2}$};\\
& \hskip11mm
\text{$l_{s-1,\; 0},\ldots,l_{s-1,\; s-1}$ act as $\lambda_{s-1,\; 0},\ldots,\lambda_{s-1,\; s-1}$  
in \eqref{EQ lambda definitions}},\\     
& \hskip11mm
\text{$p_0$ fixes $g$};\;
\text{$p_s$ fixes $g^{h_{s-1}} 
\!
\cdot 
b_{s-1}^{-1} 
\cdot 
g^{-1}\!{\boldsymbol,}\;\; l_{s-1,\; 0},\ldots,l_{s-1,\; s-1}$}\\
& \hskip59mm
\text{for each $s=1,\ldots,m$}
\big\rangle.
\end{split}
\end{equation}
In total, we got $2+6+3+ (1+\cdots+m) + (m\!+\!1) =11+{1 \over 2}(m\!+\!1)(m\!+\!2)$ generators and $14 + 3\cdot4 + \sum_{\,s=1}^m \!s(s\!+\!2) \;+ 1  
+ \sum_{\,s=1}^m (s\!+\!1)
=27+
\frac{1}{3}
(m^3+6m^2+8m)
$ relations in \eqref{EQ relations F}.
We used the clear fact that for each $s=1,\ldots,m$
the table \eqref{EQ lambda definitions}
produces $s(s\!+\!2)$ relations.

Using definition of $\mathscr{H}$ in Point~\ref{SU Presenting as a join} 
and
definition of $\mathscr{D}$ in Point~\ref{SU Construction of finitely presented DD}, we have:
\begin{equation}
\label{EQ relations D}
\begin{split}
\mathscr{D}
& = \big\langle
b, c, t_m, t'_m, t_0, t'_0, r_1, r_2;\;\;
g,h,k;\; 
\\
& \hskip18mm
l_{s-1,\; 0},\ldots,l_{s-1,\; s-1}\;(s=1,\ldots,m)
;\;\;
p_0, p_1,\ldots, p_m;\;\; a,r
\mathrel{\,|}\\
& \hskip11mm
\text{$\textstyle 27+
\frac{1}{3}
(m^3+6m^2+8m)$ relations of $\mathscr{F}$ from \eqref{EQ relations F}};\\
& \hskip11mm
\text{$a$ fixes the conjugates of 
$b, c, t_m, t'_m, t_0, t'_0, r_1, r_2;\;
g,h,k$}\\
& \hskip70mm
\text{by each of $p_0, \ldots, p_m$};\\
& \hskip11mm
\text{$r$ sends $a,b,c$ to $a,b^{c^m}\!\!\!\!,\;c$}
\big\rangle.
\end{split}
\end{equation}
Hence, $\mathscr{D}$ has $
11+{1 \over 2}m (m\!+\!1)+ (m\!+\!1) +2
= 13+{1 \over 2}(m\!+\!1)^2$ generators and 
$27+
\frac{1}{3}
(m^3+6m^2+8m) + 11(m+1) +3
=41+\frac{1}{3}(6m^2+m^3+41m)$ relations.

Then from definition of $\mathscr{L}$ in \eqref{EQ definition of L} and from definition of $L_{\mathcal X}$ in Point~\ref{SU Construction of finitely presented DD}:
\begin{equation}
\label{EQ relations L}
\begin{split}
\mathscr{L}
&  = \big\langle
b, c, t_m, t'_m, t_0, t'_0, r_1, r_2;\;\;
g,h,k;\; 
\\
& \hskip12mm
l_{s-1,\; 0},\ldots,l_{s-1,\; s-1}\;(s\!=\!1,\ldots,m)
;\;\\
& \hskip12mm
p_0, p_1,\ldots, p_m;\; a,r;\,Z;\, q_1, q_2
\mathrel{\,|}
\\
& \hskip18mm
\text{$\textstyle 41+\frac{1}{3}(6m^2+m^3+41m)$} \\
& \hskip25mm \text{relations of $\mathscr{D}$ in \eqref{EQ relations D}};\\
& \hskip18mm \text{relations $S$;\; $q_1$\! fixes} \\
& \hskip25mm \text{the generators of $L_{\mathcal X}$};\;
\text{$q_2$ fixes $a,r$}
\big\rangle.
\end{split}
\end{equation}
Thus, $\mathscr{L}$ has $15+{1 \over 2}(m\!+\!1)^2+|\,X|$ generators and $43+\frac{1}{3}(6m^2+m^3+41m)
+|R|+l_{\mathcal X}$ relations, where $l_{\mathcal X}$ is the (finite) number of generators of $L_{\mathcal X}$.

\medskip
Finally, from construction of 
$K_{\omega_m \mathcal X}$ in \eqref{EQ explicite KomegaB} and of $L' = \big\langle \mathscr{D}^{q_1}\!,\; \mathscr{D}^{q_2} \big\rangle$
in Point~\ref{SU Construction of finitely presented DD}:
\begin{equation}
\label{EQ relations K omega2 B}
\begin{split}
K_{\omega_m \mathcal X}
& = \big\langle
b, c, t_m, t'_m, t_0, t'_0, r_1, r_2;\;\;
g,h,k;\; 
\\
& \hskip17mm
l_{s-1,\; 0},\ldots,l_{s-1,\; s-1}\;(s\!=\!1,\ldots,m)
;\;\;
p_0, p_1,\ldots, p_m;\\
& \hskip17mm
\text{$a,r$;\; $Z$;\;\; $q_1,\, q_2$;\;\; $q_3,\, q_4$ }
\mathrel{\;\;|\;\;}
\\
& \hskip11mm
\text{$\textstyle 41+\frac{1}{3}(6m^2+m^3+41m)$ relations of $\mathscr{D}$ in \eqref{EQ relations D}};\\
& \hskip11mm
\text{relations $S$};\;
\text{$q_1$\! fixes generators of $L_{\mathcal X}$;\; $q_2$ fixes $a,r$};\\
& \hskip11mm
\text{$q_3$ fixes conjugates of all
$\textstyle 13+{1 \over 2}(m\!+\!1)^2$ generators}\\
& \hskip29mm 
\text{of $\mathscr{D}$ by each of $q_1,\, q_2$}; \;\;\;\;
\text{$q_4$ fixes $a,b,c$}
\big\rangle.
\end{split}
\end{equation}
$K_{\omega_m \mathcal X}$ has $17+{1 \over 2}(m\!+\!1)^2+|\,X|$ generators and
$41+\frac{1}{3}(6m^2+m^3+41m)
+|R|+l_{\mathcal X} +2
+ 2 \cdot \big( 13+{1 \over 2}(m\!+\!1)^2 \big) +3 
=73+ \frac{1}{3}(m^3+9m^2+47m)+|R|+l_{\mathcal X}$ relations, where $l_{\mathcal X}$ is the (finite) number of generators of $L_{\mathcal X}$. \;
Lastly, as $L_{\omega_m \mathcal X}$ pick the conjugate $\mathscr{D}^{\,q_3 q_4}$ of $\mathscr{D}$. It can be generated by its $13+{1 \over 2}(m\!+\!1)^2$ elements.

\medskip
The proof of Theorem~\ref{TH Theorem A}
hereby is finished.

\bigskip  
\section{Theorem~\ref{TH Theorem B} and the final embedding}
\label{SE Theorem B and the final embedding}

\subsection{Theorem~\ref{TH Theorem B} on embedding of $T_{\! \mathcal X}$}
\label{SU Theorem on embedding of TX} 

For any function $f\! \in \E$ define the word 
$w_f(x,y) = \cdots
x^{f(-1)}
y^{f(0)} 
x^{f(1)}\cdots$ in the free group $\langle x,y \rangle$ of rank $2$  
\;Say, for $f\!=(3,5,2)=(3,5,2,0)$ we have $w_f(x,y) = x^3 y^5 x^2 = x^3 y^5 x^2 y^0$\!. 
Notice that this notation is rather similar to that of Section~\ref{SU Defining subgroups by integer sequences}, but is yet different from the latter. 
For a subset $\mathcal X$ of $\E$ such words $w_f(x,y)$, $f\! \in \mathcal X$, generate a subgroup in $\langle x,y \rangle$, and the factor group of the latter by the normal closure of that subgroup defines the group:
\begin{equation}
\label{EQ introducing T_X}
\begin{split}
T_{\!\mathcal X} & = \langle x,y \rangle / \langle w_f(x,y) \;|\; f\! \in \mathcal X \rangle^{\langle x,y \rangle} \\
& =
\langle x,y 
\;|\; 
w_f(x,y)=1,\;\; f \in \mathcal X\rangle
,
\end{split}
\end{equation} 
i.e., the $2$-generator group defined by the relations $w_f(x,y)$ for all $f\! \in \mathcal X$.
We intentionally made this notation similar to $T_{\! G}$ used earlier, see Section~\ref{SU The embedding alpha into a 2-generator group}, because in an important particular case these groups are going to \textit{coincide} below, see Section~\ref{SU Equality and the final embedding}. 

Notice that for the same subset $\mathcal X$ we at the moment have both the subgroup $A_{\!\mathcal X}$ in $F_3=\langle a,b,c \rangle$, see Section~\ref{SU Defining subgroups by integer sequences}, and also the factor group $T_{\!\mathcal X}$ of $\langle x,y \rangle$. 
They are connected via:

\begin{TheoremABC}{B}
\label{TH Theorem B}
Let $\mathcal X$ be a subset of $\E$ for which $A_{\!\mathcal X}$ is benign in $F_3$.
Then $T_{\! \mathcal X}$ can be embedded into a finitely presented group $\mathcal G$.
Moreover, if the finitely presented overgroup $K_{\! \mathcal X}$ and its finitely generated subgroup $L_{\! \mathcal X}$ are given for $A_{\!\mathcal X}$ explicitly, then $\mathcal G$ can also be given explicitly.
\end{TheoremABC}

Check Section~\ref{SU Recursive enumeration and recursive groups} to  recall what we understand under \textit{explicitly given} groups.
The proof of Theorem~\ref{TH Theorem B} will be given in sections 
\ref{SU If AX is benign then}\,--\,\ref{SU Writing mathcal G by generators and defining relations}, and the promised group $\mathcal G$ can be found in Section~\ref{SU Writing mathcal G by generators and defining relations}.
See also Corollary~\ref{CO G can be embedded into 2-generator TG} in Section~\ref{SU Embedding G into the 2-generator group TG} stating that the above finitely presented group can even be $2$-\textit{generator}.

\subsection{If $A_{\!\mathcal X}$ is benign in $F_3\!=\!\langle a, b, c\rangle$, then $Z_{\!\mathcal X}$ is benign in $\langle z, r,\, s\rangle$}
\label{SU If AX is benign then}

For each function $f\!\in \E$, in analogy with the word $w_f(x,y)$ used in Section~\ref{SU Theorem on embedding of TX}, define a new word 
$w_f(r,\,s) = \cdots
s^{f(-1)}
r^{f(0)} 
s^{f(1)}\cdots$ in the free group $\langle z,r,\,s \rangle$ of rank $3$.
Then for a given subset $\mathcal X$ of $\E$ define the subgroup 
$Z_{\!\mathcal X}= \big\langle z^{w_f(r,\,s)} \mathrel{|} f \in \mathcal X \big\rangle$ in $\langle z,r,\,s \rangle$, see Figure~\ref{Figure_16_If_AX_is_benign_ZX_is_benign}.

It turns out that
if $A_{\!\mathcal X}$ is benign in $F_3$, then  $Z_{\!\mathcal X}$ is benign in $\langle z,r,\,s \rangle$.
To show this notice that the free group $\langle z, r,\, s\rangle$ has an isomorphism $\lambda$ sending $z, r,\, s$ to $z,\, s, r$  ($\lambda$ just swaps $r$ with $s$). Using it we can build the HNN-extension $\mathcal L\!= \langle z, r,\, s\rangle *_{\lambda} l$. 
Since $F_3=\langle a, b, c\rangle$ is free, there is an injection $\phi\! :F_3 \to \mathcal L$ sending $a, b, c$ to $z, r, l$. It is trivial to check that for each $f \!\in \E$ we have $\phi(a_f) = z^{w_f(r,\,s)}$\!, say, for $f=(3,5,4,7)$ we have:
$$
a_f\!
=
a^{
\big(b^{(c^0)}\big)^{\! 3}
\big(b^{(c^1)}\big)^{\! 5}
\big(b^{(c^2)}\big)^{\! 4}
\big(b^{(c^3)}\big)^{\! 7}
}
\xrightarrow{\phantom{-} \phi \phantom{-}}
\;
z^{
\big(r^{(l^{\,0})}\big)^{\! 3}
\big(r^{(l^{\,1})}\big)^{\! 5}
\big(r^{(l^{\,2})}\big)^{\! 4}
\big(r^{(l^{\,3})}\big)^{\! 7}
}
\!\!
=z^{
r^3 s^5 r^4 s^7}
\!
=z^{w_f(r,\,s)}
$$

\vskip-2mm
\noindent
because
$\big(r^{(l^{\,0})}\big)^{\! 3}
\!\!=\!\big(\lambda^0(r) \big)^{\! 3}\!\!=\!r^3$\!,\,
$\big(r^{(l^{\,1})}\big)^{\! 7}
\!\!=\!\big(\lambda(r)\big)^{\! 7}\!\!=\!s^7$\!,\,
$\big(r^{(l^{\,2})}\big)^{\! 4}
\!\!=\!\big(\lambda^2(r) \big)^{\! 4}\!\!=\!r^4$
and
$\big(r^{(l^{\,3})}\big)^{\! 7}
\!\!=\!\big(\lambda^3(r)\big)^{\! 7}\!\!=\!s^7$\!.
Hence, $Z_{\!\mathcal X}$ is the image of 
$A_{\mathcal X}$ under $\phi$, and we can use  tricks with direct product similar to those used in points 
\ref{SU Obtaining the benign subgroup for rho Q}, \ref{SU Extracting A rho X from Q} and elsewhere.

Assume the finitely presented overgroup $K_{\mathcal X}=\langle
\,Z
\;|\;
S
\rangle$ and its finitely generated subgroup $L_{\!\mathcal X}$ are explicitly known for $A_{\!\mathcal X}$.
Since we had the freedom to chose the letters $z,r,\,s,l$ in construction of $\mathcal L$, we may assume they are disjoint from $K_{\!\mathcal X}$.

\begin{figure}[h]
\includegraphics[width=390px]{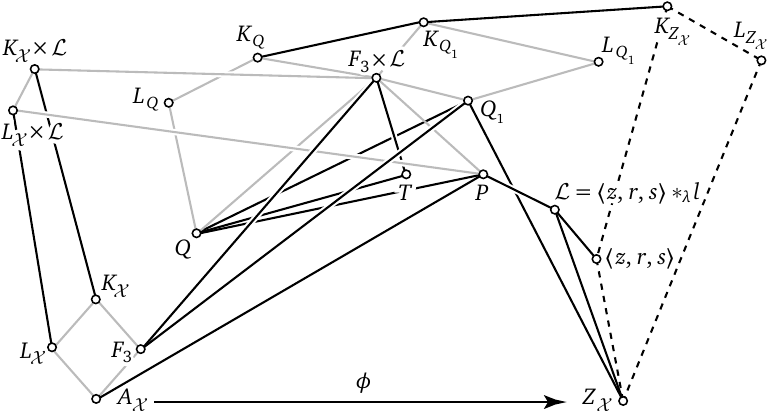}
\caption{If $A_{\!\mathcal X}$ is benign in $F_3$, then  $Z_{\!\mathcal X}$ is benign in $\langle z,r,\,s \rangle$.}
\label{Figure_16_If_AX_is_benign_ZX_is_benign}
\end{figure}

The $3$-generator subgroup 
$T\!=\big\langle 
(a,\, z),\;
(b,\, r),\;
(c,\, l)
\big\rangle$
is benign in the direct product $F_3 \times \mathcal L$,
see Remark~\ref{RE finite generated is benign}, and the  direct product $P=A_{\!\mathcal X} \times \mathcal L$ is benign in $F_3 \times \mathcal L$ for the finitely presented $K_{\!\mathcal X} \times \mathcal L$ and finitely generated $L_{\!\mathcal X} \times \mathcal L$.
Hence by Corollary~\ref{CO intersection and join are benign multi-dimensional}\,\eqref{PO 1 CO intersection and join are benign multi-dimensional} the intersection $Q=
T \cap P$ is benign in $F_3 \times \mathcal L$ for the finitely presented $\bigast$-construction:
$$
K_{Q} =
\big((F_3 \times \mathcal L)\;*_{T} v_1\big)
\,
*_{F_3 \times \mathcal L}
\big((K_{\mathcal X} \times \mathcal L)\;*_{L_{\mathcal X} \times \mathcal L} v_2\big)
$$ 
and for the finitely generated subgroup $L_{Q} =(F_3 \times \mathcal L)^{v_1 v_2}$ of the latter.

Since $T$ is also equal to  $\big\langle 
\big(a,\, \phi(a)\big),\;
\big(b,\, \phi(b)\big),\;
\big(c,\, \phi(c)\big)
\big\rangle$ and $\phi$ is an isomorphism, then 
$T=\big\langle
\big(w,\,\phi(w)\big) \mathrel{|} 
w \in  F_3\big\rangle$.
As the first coordinate of any couple from $P$ is in $A_{\!\mathcal X}$, we get the simple description 
$Q =
\big\langle
\big(w,\, \phi(w)\big) \mathrel{|} 
w \in   A_{\mathcal X}\big\rangle$, i.e., the second coordinates of couples of $Q$ in fact form the image $\phi(A_{\mathcal X})
=Z_{\mathcal X}$. 

\medskip
Modify $Q$ via a few steps to arrive to $Z_{\!\mathcal X}$ wanted.
$F_3 \cong F_3 \times \1$ is benign in 
$F_3 \times \mathcal L$ for the finitely presented $F_3 \times \mathcal L$ and finitely generated $F_3 \times \1$.  Hence the join 
$
Q_1=\big\langle F_3 \times \1 ,\, Q \big\rangle
=F_3 \times \langle
\phi(w) \mathrel{|} 
w \in   A_{\mathcal X}\rangle
$ is benign in $F_3 \times \mathcal L$ for the finitely presented $\bigast$-construction:
$$
K_{Q_1}=
\big((F_3 \times \mathcal L)\,*_{F_3 \times \1} w_1\big) 
\,*_{F_3 \times \mathcal L}
(K_{Q} *_{(F_3 \times \mathcal L)^{v_1 v_2}} \;w_2),
$$
and for the finitely generated subgroup 
$L_{Q_1} =\big\langle(F_3 \times \mathcal L)^{w_1},\;(F_3 \times \mathcal L)^{w_2} \big\rangle$ in $K_{Q_1}$.

Further,  
$\mathcal L = \1 \times  \mathcal L$ is benign in $F_3 \times \mathcal L$ for the finitely presented $F_3 \times \mathcal L$ and finitely generated $\1 \times  \mathcal L$.  Hence, the intersection
$$
\big(\1 \times  \mathcal L\big) \cap Q_1 
= \langle
\phi(w) \mathrel{|} 
w \in   A_{\mathcal X}\rangle
=\phi(A_{\mathcal X})
=Z_{\mathcal X}
$$ 
is benign in $F_3 \times \mathcal L$ for the finitely presented overgroup:
$$
K_{Z_{\mathcal X}}  =
\big((F_3 \times \mathcal L)\,*_{\1 \times  \mathcal L}\; w_3\big) 
\,*_{F_3 \times \mathcal L}
(K_{Q_1} *_{L_{Q_1}} w_4),
$$
and for its finitely generated subgroup
$L_{Z_{\mathcal X}} 
=(F_3 \times \mathcal L)^{w_3 w_4}$.
Since $Z_{\mathcal X}$ is entirely inside $\langle z,r,\,s \rangle$, 
we have $\langle z,r,\,s \rangle \cap\, L_{Z_{\mathcal X}} \!= Z_{\mathcal X}$, that is, $Z_{\mathcal X}$ is also benign in  $\langle z,r,\,s \rangle$ for the same groups $K_{Z_{\mathcal X}}$ and $L_{Z_{\mathcal X}}$ constructed above.

\subsection{Writing $K_{Z_{\mathcal X}}$  by generators and defining relations}
\label{SU Writing KZX by generators and defining relations} 

If $K_{\mathcal X}=\langle
\,Z
\;|\;
S
\rangle$ 
is given, following Point~\ref{RE abc can be added} we may assume $Z$ contains $a,b,c$. 
Also, the finitely many generators of $L_{\!\mathcal X}$ can effectively be computed inside $K_{\!\mathcal X}$.
Using definitions of $\mathcal L$, $Q$, $\phi$, $K_{Q}$ above we have:
\begin{equation}
\label{EQ relations KQ}
\begin{split}
K_{Q} 
& = \big\langle
\,Z;\;\; z, r,\, s, l; \;\;  v_1, v_2
\mathrel{\;\;|\;\;} S;\\
& \hskip11mm 
\text{$l$ sends $z, r,\, s$ to $z,\, s, r$};\\[-2pt]
& \hskip11mm 
\text{$z, r,\, s, l$ commute with generators $Z$};\\
& \hskip11mm 
\text{$v_1$ fixes $az,\; br,\; cl$;}\\
& \hskip11mm 
\text{$v_2$ fixes $z, r, s, l$ and the generators of  $L_{\!\mathcal X}$}
\big\rangle.
\end{split}
\end{equation} 
Notice that we do \textit{not} include the relations telling that $a,b,c$ commute with $z, r,\, s, l$ (to reflect the direct product $F_3 \times \mathcal L$) because 
$K_{\!\mathcal X}$ already includes $F_3$, and so the third line of \eqref{EQ relations KQ} already is enough. 
If $K_{\!\mathcal X}$ has $m$ generators (which we may assume include $a,b,c$) and $n$ defining relations, and if 
$L_{\mathcal X}$ has $k$ generators, then the group $K_{Q}$ in \eqref{EQ relations KQ} has 
$m+4+2=m+6$ generators and 
$n+3+4\cdot m + 3 + 4 + k = n+4m+k + 10$ 
defining relations.

\medskip
Next, using the definition of $K_{Q_1}$ in Section~\ref{SU If AX is benign then} we write:
\begin{equation}
\label{EQ relations KQ1}
\begin{split}
K_{Q_1}
& = \big\langle
\text{$m\!+\!6$ generators of $K_{Q_1}$ from \eqref{EQ relations KQ}};\;\;\;  w_1, w_2
\mathrel{\;\;|\;\;}\\
& \hskip11mm 
\text{$n+4m+k + 10$ relations of $K_{Q_1}$ from \eqref{EQ relations KQ}};\\
& \hskip11mm 
\text{$w_1$ fixes $a,b,c$};\\
& \hskip11mm 
\text{$w_2$ fixes conjugates of $a,b,c;\;  z, r,\, s, l$ by $v_1 \! v_2$}
\big\rangle. 
\end{split}
\end{equation}
The group $K_{Q_1}$ in \eqref{EQ relations KQ1} has 
$m+6+2=m+8$ generators and 
$n + 4m + k + 10 + 3 + 7 
= n + 4m + k + 20$ 
defining relations.

\medskip
Finally, by definition of $K_{Z_{\mathcal X}}$
\!above we have:
\begin{equation}
\label{EQ relations KZX}
\begin{split}
K_{Z_{\mathcal X}} \!\!
& = \big\langle
\text{$m\!+\!8$ generators of $K_{Q_1}$ from \eqref{EQ relations KQ1}};\;\;\;  w_3, w_4
\mathrel{\;\;|\;\;}\\
& \hskip11mm 
\text{$n + 4m + k + 20$ relations of $K_{Q_1}$ from \eqref{EQ relations KQ1}}; \\
& \hskip11mm 
\text{$w_3$ fixes $z, r, s, l$};\\
& \hskip11mm 
\text{$w_4$ fixes conj.\! of $a,b,c;\;  z, r,\, s, l$ by $w_1$ and $w_2$}
\big\rangle.
\end{split}
\end{equation}
$K_{Z_{\mathcal X}}$ has $m+8+2=m+10$ generators and $n + 4m + k + 20 + 4 + 2 \cdot 7 = 
n + 4m + k + 38$ relations.
The group  
$L_{Z_{\mathcal X}} 
=(F_3 \times \mathcal L)^{w_3 w_4}$
is a $7$-generator subgroup in $K_{Z_{\mathcal X}}$.

\subsection{If $Z_{\mathcal X}$ is benign in $\langle z, r,\, s\rangle$, then $Q_{\mathcal X}$ is benign in $\langle p,q \rangle$}
\label{SU If ZX is benign then Q(p,q) is also benign}

Take yet another free group $\langle p,q \rangle \cong F_2$,
and again in full analogy with the words $w_f(x,y)$ in Section~\ref{SU Theorem on embedding of TX}
introduce the words 
$w_f(p,q) = \cdots
q^{f(-1)}
p^{f(0)} 
q^{f(1)}\cdots$ for each $f \in \mathcal E$,\, and then also define the subgroup
$Q_{\mathcal X}  = 
\langle w_f(p,q) \mathrel{|} f \in \mathcal X \rangle$ in $\langle p,q \rangle$.

It turns out that $Q_{\mathcal X}$ is benign in $\langle p,q \rangle$ as soon as $Z_{\!\mathcal X}$ is benign in $\langle z,r,\,s \rangle$.
To show this we are going to ``connect'' $\langle z,r,\,s \rangle$ to $\langle p,q \rangle$ in a specific larger group, see Figure~\ref{Figure_17_If_ZX_is_benign_QX_is_benign}. The idea is based on Higman's original idea mentioned in Section~5 in \cite{The Higman operations and  embeddings}, and a variant of this we recently used   for the group $\Q$ in \cite{On explicit embeddings of Q}.
Namely, pick an infinite cycle $\langle u \rangle$ and notice that the free product $K_{Z_{\mathcal X}} * \langle u \rangle$ contains the free subgroup $\langle z,r,\,s \rangle * \langle u \rangle=\langle z,r,\,s, u \rangle \cong F_4$. From the normal form of elements in free products it is trivial that in $K_{Z_{\mathcal X}}\! * \langle u \rangle$ the intersection $\langle z,r,\,s, u \rangle \cap L_{Z_{\mathcal X}}$ is nothing but $\langle z,r,\,s \rangle \cap L_{Z_{\mathcal X}}=Z_{\!\mathcal X}$, i.e., $Z_{\!\mathcal X}$ is benign in $\langle z,r,\,s, u \rangle$ for the finitely presented $K_{Z_{\mathcal X}}\! * \langle u \rangle$ and for the finitely generated  $L_{Z_{\mathcal X}}$.
In $\langle z,r,\,s, u  \rangle$ the words of type $u^{z^{w(r,\,s)}}$\!\!, with $w(r,\,s)$ running through the set of \textit{all} possible words on $r,\,s$ (no restriction depending for $\mathcal X$ for now), \textit{freely} generate a subgroup $\mathcal{J}_1$ of countable rank.

\vskip-2mm
\begin{figure}[h]
\includegraphics[width=390px]{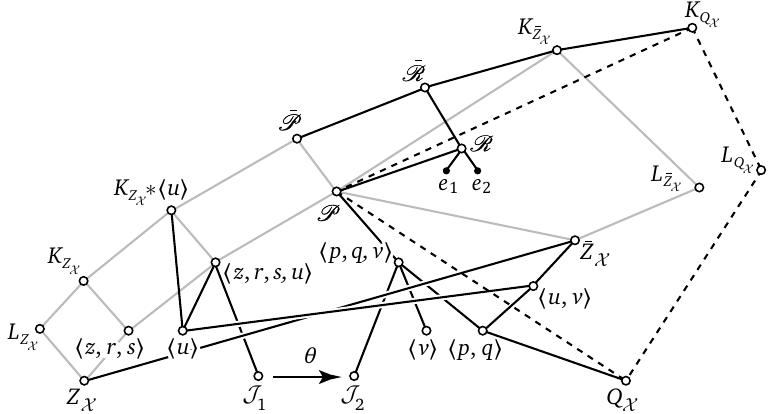}
\caption{If $Z_{\mathcal X}$ is benign in $\langle z, r,\, s\rangle$, then $Q_{\mathcal X}$ is benign in $\langle p,q \rangle$.}
\label{Figure_17_If_ZX_is_benign_QX_is_benign}
\end{figure} 

\medskip 
For another infinite cycle $\langle v \rangle$ in the free product $\langle p,q\rangle * \langle v \rangle=\langle p,q,v \rangle \cong F_3$ select the generators $v\cdot w(p,q)$ for yet another \textit{free} group $\mathcal{J}_2$ of countable rank. 
The map $\theta$ sending each $u^{z^{w(r,\,s)}}\!$ to $v\cdot w(p,q)$ can be continued to an isomorphism $\theta: \mathcal{J}_1 \to \mathcal{J}_2$ between two  subgroups in $K_{Z_{\mathcal X}}\! * \langle u \rangle$ and in $\langle p,q,v \rangle$. This lets us define the following two free products with amalgamation by $\theta$:
$$
\mathscr{P}
=
\langle z,r,\,s,u\rangle 
*_{\theta}
\langle p,q,v \rangle
,\quad\quad\quad
\bar{\mathscr{P}}
=
\big( K_{Z_{\mathcal X}}*\langle u \rangle \big)
*_{\theta}
\langle p,q,v \rangle
$$  
with $\mathscr{P} \le \bar{\mathscr{P}}$.
Since the words $v\cdot w(p,q)$ generate the whole $\langle p,q,v \rangle$, the latter entirely is inside $K_{Z_{\mathcal X}}*\langle u \rangle$, and we in fact have 
$\mathscr{P}
\cong
\langle z,r,\,s,u\rangle$
and
$\bar{\mathscr{P}}
\cong
K_{Z_{\mathcal X}}*\langle u \rangle$, that is, $\bar{\mathscr{P}}$ can be understood as the group $K_{Z_{\mathcal X}}*\langle u \rangle$ in which we denoted  
$u^z$ by $v$,
denoted each 
$u^{z^{w(r,\,s)}}$ by $v\cdot w(p,q)$, and then added the new relations 
$u^{z^{w(r,\,s)}}\!=\,v\cdot w(p,q)$ in order to mimic $\theta$.

\medskip
Further, $\langle z,r,\,s,u\rangle$ admits an isomorphism sending $z,r,\,s,u$ to $z^r\!,\, r,\,s,u$, and $\langle p,q,v \rangle$ admits an isomorphism sending $p,q,v$ to $p,q,v\cdot p$. It is trivial to verify that these two isomorphisms agree on words of type $u^{z^{w(r,\,s)}}$ and $v\cdot w(p,q)$, and so they have a common continuation $\eta_1$ on the whole $\mathscr{P}$. Similarly define an isomorphism $\eta_2$ on $\mathscr{P}$ sending 
$z,r,\,s,u$ to $z^s\!,\, r,\,s,u$
and 
$p,q,v$ to $p,q,v\cdot q$.
Using these isomorphisms define the finitely generated HNN-extensions:
$$
\mathscr{R} = \mathscr{P} *_{\eta_1,\eta_2} (e_1, e_2)
,\quad\quad\quad
\bar{\mathscr{R}} = \bar{\mathscr{P}} *_{\eta_1,\eta_2} (e_1, e_2)
$$
where $\mathscr{R} \le \bar{\mathscr{R}}$.
According to the above constructions, $\bar{\mathscr{R}}$ can be given by the relations:
\begin{enumerate}
\item 
\label{EN PO1 relations of bar R}
the \textit{finitely} many relations of $K_{Z_{\mathcal X}}$;

\item 
\label{EN PO2 relations of bar R}
the \textit{infinitely} many relations for $\theta$, i.e., those stating  $u^{z^{w(r,\,s)}}\!\!=v\cdot w(p,q)$ for all $w$;

\item 
\label{EN PO3 relations of bar R}
the \textit{finitely} many relations telling the images of $7$ generators $z,r,\,s,u;p,q,v$  under conjugation by $e_1$ and $e_2$.
\end{enumerate}

The relations of point \eqref{EN PO2 relations of bar R} above are mainly redundant, and they can  be replaced by the \textit{single} relation $u^z=v$ (which evidently is $u^{z^{w(r,\,s)}}\!\!=v\cdot w(p,q)$ for the \textit{trivial} word $w=1$). The routine of verification is simple, and we display the idea just by an example. For, say, the word $w(r,\,s) = w_f(r,\,s) 
= r^3 s^5 r^4 s^7$ from Section~\ref{SU If AX is benign then} we from $u^z=v$ can deduce:
\begin{equation}
\label{EQ 3 5 4 7}
u^{z^{r^3 s^5 r^4 s^7}}\!\! = v\cdot p^3 q^5 p^4 q^7
\end{equation}
in the following way. 
Putting the exponents in the above word $w$ in \textit{reverse order} write down the word $w'=w'(e_1, e_2) 
= e_2^7 e_1^4 e_2^5 e_1^3$ in stable letters $e_1, e_2$. Then:
\begin{equation}
\label{EQ kappa_2 kappa_1 first}
\begin{split}
\left(u^{z}\right)^{w'} \!&  
= \left(u^{z}\right)^{e_2^{\boldsymbol 7} e_1^4 e_2^5 e_1^3}
= \left((u^{e_2})^{z^{e_2}}\right)^{e_2^{\boldsymbol 6} e_1^4 e_2^5 e_1^3} 
= 
\left(u^{z^s}\right)^{e_2^{\boldsymbol 6} e_1^4 e_2^5 e_1^3}
\\
& 
= \big(u^{z^{s^7}}\big)^{e_1^4 e_2^5 e_1^3}
= \big(u^{z^{r^4 s^7}}\big)^{e_2^5 e_1^3}
= \big(u^{z^{s^5 r^4 s^7}}\big)^{e_1^3}
= u^{z^{r^3 s^5 r^4 s^7}};
\end{split}
\end{equation}
\begin{equation}
\label{EQ kappa_2 kappa_1 second}
\begin{split}
\hskip-1mm \left(v\right)^{w'} \!&  
= \left(v\right)^{e_2^{\boldsymbol 7} e_1^4 e_2^5 e_1^3}
= \left(v^{e_2}\right)^{e_2^{\boldsymbol 6} e_1^4 e_2^5 e_1^3} 
= \left(v \cdot q\right)^{e_2^{\boldsymbol 6} e_1^4 e_2^5 e_1^3}
= \left(v \cdot q^7\right)^{e_1^4 e_2^5 e_1^3}
\\
& 
= \left(v \cdot p^4 q^7\right)^{e_2^5 e_1^3}
= \left(v \cdot q^5 p^4 q^7\right)^{e_1^3}
= v \cdot p^3 q^5 p^4 q^7.
\end{split}
\end{equation}
So from $u^z = v$ with equalities \eqref{EQ kappa_2 kappa_1 first} and \eqref{EQ kappa_2 kappa_1 second} follows \eqref{EQ 3 5 4 7} for \textit{arbitrary} word $w$, that is, the above $\bar{\mathscr{R}}$ is finitely presented.

\medskip
Clearly, $Z_{\!\mathcal X}$ is benign also in a larger group $\mathscr{P}
=
\langle z,r,\,s,u\rangle 
*_{\theta}
\langle p,q,v \rangle$ for the finitely presented overgroup $\bar{\mathscr{R}}$, and for the earlier used finitely generated  $L_{Z_{\mathcal X}}$.
The finitely generated subgroup $\langle u,v \rangle$ trivially is benign in 
$\mathscr{P}$ for the finitely presented $\bar{\mathscr{R}}$, and for its finitely generated subgroup $\langle u,v \rangle$, see Remark~\ref{RE finite generated is benign}.
Hence by Corollary~\ref{CO intersection and join are benign multi-dimensional}\,\eqref{PO 2 CO intersection and join are benign multi-dimensional} the join $\bar Z_{\mathcal X}=\langle Z_{\mathcal X};\; u, v \rangle$ is benign in $\mathscr{P}$ for the finitely presented $\bigast$-construction: 
$$
K_{\bar Z_{\mathcal X}}
=
\big(\bar{\mathscr{R}} *_{L_{Z_{\mathcal X}}}\! h_1\big) *_{\bar{\mathscr{R}}}
\big(\bar{\mathscr{R}} *_{\langle u,v \rangle} h_2\big),
$$
and for its finitely generated subgroup
$
L_{\bar Z_{\mathcal X}}\!
=
\big\langle 
\mathscr{P}^{h_1}\!,\;
\mathscr{P}^{h_2}\!
\big\rangle
$.

Further, the finitely generated subgroup $\langle p,q \rangle$ is benign in 
$\mathscr{P}$ for the finitely presented $\bar{\mathscr{R}}$, and for its finitely generated subgroup $\langle p,q \rangle$.
Hence, if we also prove the equality:
\begin{equation}
\label{EQ important equality here}
\bar Z_{\mathcal X}
\cap 
\langle p,q \rangle
= Q_{\mathcal X}  = 
\big\langle w_f(p,q) \mathrel{|} f \in \mathcal X\, \big\rangle,
\end{equation}
then by Corollary~\ref{CO intersection and join are benign multi-dimensional}\,\eqref{PO 1 CO intersection and join are benign multi-dimensional} the intersection $Q_{\mathcal X}$ is benign in $\mathscr{P}$ for the finitely presented group:
$$
K_{Q_{\mathcal X}} 
=
\big(K_{\bar Z_{\mathcal X}} *_{L_{\bar Z_{\mathcal X}}} f_1\big) *_{\bar{\mathscr{R}}}
\big(\bar{\mathscr{R}} *_{\langle p,q \rangle} f_2\big),
$$
and for its finitely generated subgroup
$
L_{Q_{\mathcal X}}
=
\mathscr{P}^{f_1 f_2}
$.
But since $Q_{\mathcal X}$ entirely is inside $\langle p,q \rangle$, then from 
$\mathscr{P} \cap L_{Q_{\mathcal X}} = Q_{\mathcal X}$ it follows $\langle p,q \rangle \cap L_{Q_{\mathcal X}}\!\! =\, Q_{\mathcal X}$, that is, $Q_{\mathcal X}$ is benign in $\langle p,q \rangle$ for the \textit{same} groups
$K_{Q_{\mathcal X}}$ and $L_{Q_{\mathcal X}}$\! chosen above.

To conclude our proof it remains to argument the equality \eqref{EQ important equality here}.
For every $f\! \in \mathcal X$ and the respective word $w_f(p,q)$ we have:
$$
w_f(p,q)= v^{-1} \cdot  v\cdot w_f(p,q)
= v^{-1} \cdot u^{z^{w_f(r,\,s)}} \!
\in \bar Z_{\mathcal X}.
$$
To see that the left-hand side of \eqref{EQ important equality here} is in $W_{\!\mathcal X}$ it is enough to apply the conjugate collection process of Section~\ref{SU The conjugates collecting process} for $\X=\{u,v \}
$ and for 
$\Y=\{w_f(p,q) \mathrel{|} f \! \in \mathcal X \}$.

\subsection{Writing $K_{Q_{\mathcal X}}$  by generators and defining relations}
\label{SU Writing KQX by generators and defining relations} 

From construction of $\bar{\mathscr{P}}$ in previous section and from representation \eqref{EQ relations KZX} of $K_{Z_{\mathcal X}}$ we have:
\begin{equation}
\label{EQ relations bar P}
\begin{split}
\bar{\mathscr{P}} 
& = \big\langle
\text{$m\! +\! 10$ generators of $K_{Z_{\mathcal X}}$ from \eqref{EQ relations KZX}};\;\;\;  u, p, q, v
\mathrel{\;\;|\;\;}\\
& \hskip11mm 
\text{$n\! +\! 4m\! + \!k\! +\! 38$ relations of $K_{Z_{\mathcal X}}$ from \eqref{EQ relations KZX}}\\
& \hskip11mm 
\text{\textit{infinitely} many relations $u^{z^{w(r,\, s)}}\!\! = v\cdot w(p,q)$}\\
& \hskip11mm 
\text{for \textit{all} words $w(r,\, s)\in \langle r,\, s \rangle$}\;
\big\rangle.
\end{split}
\end{equation}
$\bar{\mathscr{P}}$ has $m\! +\! 10\!+\! 4=m\! +\! 14$ generators and \textit{infinitely} many relations, where the group $K_{\mathcal X}=\langle
\,Z
\;|\;
S
\rangle$ and the values $m, n, k$ are the same as those in Section~\ref{SU Writing KZX by generators and defining relations}. 
Next we reduce those infinitely many relations to \textit{finitely} many relations. Namely, from definition of $\bar{\mathscr{R}}$, $\eta_1$, $\eta_2$, $e_1$, $e_2$ in previous section we get: 
\begin{equation}
\label{EQ relations bar R}
\begin{split}
\bar{\mathscr{R}} 
& = \big\langle
\text{$m\! +\! 10$ generators of $K_{Z_{\mathcal X}}$ from \eqref{EQ relations KZX}};\;\; u, p, q, v;\; e_1, e_2
\mathrel{\;|}\\
& \hskip11mm 
\text{$n\! +\! 4m\! + \!k\! +\! 38$ relations of $K_{Z_{\mathcal X}}$ from \eqref{EQ relations KZX}};\\
& \hskip11mm 
\text{a\, \textit{single}  relation $u^z=v$};\\
& \hskip11mm 
\text{$e_1$ sends 
$z,r,\, s,u; p,q,v$ to $z^r\!,\, r,\, s,u, p,q,v\cdot p$};\\[-3pt]
& \hskip11mm 
\text{$e_2$ sends 
$z,r,\, s,u; p,q,v$ to $z^s\!,\, r,\, s,u, p,q,v\cdot q$}
\big\rangle.
\end{split}
\end{equation}
$\bar{\mathscr{R}}$ has $m\!+\!10\!+\!6=m\!+\!16$ generators and $n\! +\! 4m\! + \!k\! +\! 38\!+\!1\!+\!2\cdot 7
=n\! +\! 4m\! + \!k\! +\! 53$ relations.
Hence, using definition of $K_{\bar Z_{\mathcal X}}$\! in Section~\ref{SU If ZX is benign then Q(p,q) is also benign},  and the $7$-generator group $L_{Z_{\mathcal X}}$ at the end of  Section~\ref{SU Writing KZX by generators and defining relations}
we have:
\begin{equation}
\label{EQ relations KZT+}
\begin{split}
K_{\bar Z_{\mathcal X}}
& = \big\langle
\text{$m\!+\!16$ generators of $\bar{\mathscr{R}}$ from \eqref{EQ relations bar R}};\;  h_1,h_2
\mathrel{\;|\;} \\
& \hskip11mm 
\text{$n\! +\! 4m\! + \!k\! +\! 53$ relations of $\bar{\mathscr{R}}$ from \eqref{EQ relations bar R}};\\
& \hskip11mm 
\text{$h_1$ fixes conjugates of $a,b,c;\;  z, r,\, s, l$ by $w_3 w_4$};\\
& \hskip11mm
\text{$h_2$ fixes 
$u,v$}
\big\rangle.
\end{split}
\end{equation}
$K_{\bar Z_{\mathcal X}}$ has $m\!+\!16\!+\!2=m\!+\!18$ generators and $n\! +\! 4m\! + \!k\! +\! 53\! +\!7\! +\!2=n\! +\! 4m\! + \!k\! +\! 62$ relations, and $
L_{\bar Z_{\mathcal X}}
=
\big\langle 
\mathscr{P}^{h_1}\!,\;
\mathscr{P}^{h_2}\!
\big\rangle 
$ is a $14$-generator because $\mathscr{P}$ is $7$-generator.

Finally, using construction of $K_{Q_{\mathcal X}}$, $L_{Q_{\mathcal X}}$, $\bar{\mathscr{R}}$ at the end of~\ref{SU If ZX is benign then Q(p,q) is also benign}, we arrive to:
\begin{equation}
\label{EQ relations KQX}
\begin{split}
K_{Q_{\mathcal X}}\!\!
& = \big\langle
\text{$m\!+\!18$ generators of $K_{\bar Z_{\mathcal X}}$ from \eqref{EQ relations KZT+}};\;  f_1,f_2
\mathrel{\;\;|\;\;} \\
& \hskip9mm 
\text{$n\! +\! 4m\! + \!k\! +\! 62$ relations of $K_{\bar Z_{\mathcal X}}$ from \eqref{EQ relations KZT+}};\\
& \hskip9mm 
\text{$f_1$ fixes conjugates of $a,b,c;\;  z, r,\, s, l$ by $h_1$ and $h_2$};\\
& \hskip9mm 
\text{$f_2$ fixes 
$p,q$}
\big\rangle.
\end{split}
\end{equation}
Finally, the overgroup
$K_{Q_{\mathcal X}}$ has $m\!+\!18\!+\!2=m\!+\!20$ generators and $n\! +\! 4m\! + \!k\! +\! 62
\!+\! 2 \cdot 7 +2 =
n\! +\! 4m\! + \!k\! +\! 78
$ relations.
The finitely generated subgroup
$L_{Q_{\mathcal X}}$\! of $K_{Q_{\mathcal X}}$\! is just $7$-generator.

\subsection{The Higman Rope Trick}
\label{SU The Higman Rope Trick}


By hypothesis of Theorem~\ref{TH Theorem B} the subset $\mathcal X \subseteq \E$ determines a benign subgroup $A_{\! \mathcal X}$ in $F_3$.
Hence $\mathcal X$ is recursive, and by Theorem 3 in \cite{Higman Subgroups of fP groups} it can be const\-ructed via the operations \eqref{EQ Higman operations} from 
$\Zz$ and
$\S$. We in sections~\ref{SU The proof for the case of Z and S}\,--\,\ref{SU The proof for omega_m} were able to explicitly construct the finitely presented overgroup $K_{\! \mathcal X}$ with its finitely generated subgroup $L_{\! \mathcal X}$ such that 
$F_3\cap L_{\! \mathcal X}
= A_{\! \mathcal X}$ holds.
Those constructions avoided any usage of the letters $x,y$, and so the group $K_{\! \mathcal X}$ outputted by Theorem~\ref{TH Theorem A} at the end of Chapter~\ref{SE Theorem A and its proof steps} does not involve $x,y$, and hence the group $T_{\!\mathcal X}$ of Section~\ref{SU Theorem on embedding of TX} is disjoint from $K_{\! \mathcal X}$.
Further, $T_{\!\mathcal X}$ is disjoint from the group $K_{Q_{\mathcal X}}$\! built in sections \ref{SU If AX is benign then}\,--\,\ref{SU Writing KQX by generators and defining relations} by just adjoining some new letters $z, r, \,s, \ldots , v$ to $K_{\! \mathcal X}$.  
Hence no conflict arises if we use the group $T_{\!\mathcal X}$ in a construction together with $K_{Q_{\mathcal X}}$.  

It is evident that if we in a word $w_f(p,q)$ replace the letters $p,q$ by $x,y$, then we get nothing but the initial word $w_f(x,y)$ in $F_3$. This remark is going to play useful role later.
Denoting for simplicity the normal closure $\langle w_f(x,y) \;|\; f\! \in \mathcal X \rangle^{\langle x,y \rangle}$ mentioned in Section~\ref{SU Theorem on embedding of TX} by $\mathcal R$, we can rewrite 
$T_{\!\mathcal X} = \langle x,y \rangle / \mathcal R$.

\begin{figure}[h]
\includegraphics[width=390px]{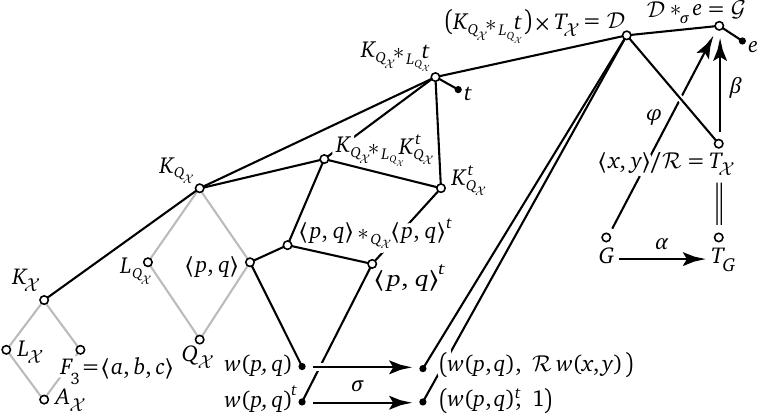}
\caption{Using the Higman Rope Trick.}
\label{Figure_18_Higman_rope_trick}
\end{figure}

By assumption of Theorem~\ref{TH Theorem B} and by Section~\ref{SU If ZX is benign then Q(p,q) is also benign}, the subgroup $Q_{\mathcal X}$ is benign in $\langle p,q \rangle$. 
Fix a new stable letter $t$, and build the finitely presented HNN-extension $K_{Q_{\mathcal X}} \!*_{L_{Q_{\mathcal X}}}\! t$ inside which the subgroup $K_{Q_{\mathcal X}}$\! and its conjugate $K_{Q_{\mathcal X}}^t$\! clearly generate their free product with amalgamation
$
K_{Q_{\mathcal X}} \!*_{L_{Q_{\mathcal X}}}\!
K_{Q_{\mathcal X}}^t
$.
Since $\langle p,q \rangle \cap L_{Q_{\mathcal X}} = Q_{\mathcal X}$, then by Corollary~3.2\,(1) in \cite{Auxiliary free constructions for explicit embeddings}
the subgroup $\langle p,q \rangle$  and its conjugate $\langle p,q \rangle^t$ generate their free product with amalgamation: 
\begin{equation}
\label{EQ small free product with amalgamation}
\langle p,q \rangle *_{Q_{\mathcal X}}
\! \langle p,q \rangle^t\!.
\end{equation}

For the direct product 
$\mathcal D = 
\big( K_{Q_{\mathcal X}} \!*_{L_{Q_{\mathcal X}}}\! t \big) \times T_{\!\mathcal X}$
we have an isomorphism from $\langle p,q\rangle \times \1$ into $\mathcal D$ defined on \textit{every} $w(p,q) \in \langle p,q \rangle$ (\textit{not} necessarily a relation of type $w_f(p,q)$), via:
$$
w(p,q) = \big(w(p,q) ,\;\; 1\big) 
\to 
\big(w(p,q) ,\;\; \mathcal R\, w(x,y) \big),
$$
where $\mathcal R\, v(x,y)$ is an element (coset) in $T_{\!\mathcal X}\! =\! \langle x,y \rangle / \mathcal R$, see Figure~\ref{Figure_18_Higman_rope_trick}.
We also have the identical injection $\langle p,q \rangle^t$ into $\mathcal D$ given via:
$$
w(p,q)^t = \big(w(p,q)^t ,\;\; 1\big) 
\to 
\big(w(p,q)^t ,\;\; 1\big).
$$
For every word $w(p,q)\in Q_{\mathcal X}$ (in particular, for each of the relations $w(p,q)=w_f(p,q)$), 
we have $w(p,q) = w(p,q)^t$, and the coset
$\mathcal R \, w(x,y)$ is trivial in $\langle x,y \rangle / \mathcal R$. Hence, we have 
$\big(w(p,q) ,\;\; \mathcal R\, w(x,y) \big)=\big(w(p,q)^t ,\; 1\big)$ for such $w(p,q)$, and the above two isomorphisms agree on the amalgamated subgroup $Q_{\mathcal X}$ of \eqref{EQ small free product with amalgamation}.
Hence these two isomorphisms have a common continuation $\sigma$ on the whole \eqref{EQ small free product with amalgamation}, and using it we define our last HNN-extension:
\begin{equation}
\label{EQ definition of G}
\mathcal G  
=
\mathcal D *_\sigma e
=
\Big(\big( K_{Q_{\mathcal X}} \!*_{L_{Q_{\mathcal X}}} \!t \big) \times T_{\!\mathcal X} \Big) *_\sigma e.
\end{equation}
The group $\mathcal G$ clearly contains $T_{\!\mathcal X}$, and for later use we denote that identical embedding via:
\begin{equation}
\label{EQ define beta}
\beta: T_{\!\mathcal X} \to \mathcal G.
\end{equation} 
$\mathcal G$ is finitely generated, and so the desired embedding for Theorem~\ref{TH Theorem B} will be achieved, if we show that $\mathcal G$ can be given by \textit{finitely} many relations.
The relations suggested by \eqref{EQ definition of G} are:
\begin{enumerate}
\item the \textit{finitely} many relations of $K_{Q_{\mathcal X}}$;
\label{ENUM PO1 relations of G}

\item the \textit{finitely} many relations stating that $t$ fixes the \textit{finitely} many generators of $L_{Q_{\mathcal X}}$;
\label{ENUM PO2 relations of G}

\item the \textit{finitely} many relations stating that both generators $x,y$ of $T_{\!\mathcal X}$ commute with $t$ and with the \textit{finitely} many generators of $K_{Q_{\mathcal X}}$;
\label{ENUM PO3 relations of G}

\item the \textit{infinitely} many relations $w_f(x,y)$, $f\!\in \mathcal X$, for $T_{\!\mathcal X}$.
\label{ENUM PO4 relations of G}

\item the \textit{infinitely} many relations defining the action of $e$ on \textit{all} words in $\langle p,q \rangle$ and in $\langle p,q \rangle^t$\!. 
\label{ENUM PO5 relations of G}
\end{enumerate}

\medskip
Since $\sigma$ is an isomorphism, infinitely many relations of point~\eqref{ENUM PO5 relations of G} can be replaced by just \textit{four} relations defining the images of generators: 
$$
\text{$p^e\!=\! p\; \mathcal R  x$, \quad
$q^e\!=\!q\; \mathcal R  y$, \quad
$(p^t)^e\!=\! p^t$\!, \quad
$(q^t)^e\!=\!q^t$\!.}
$$

If we also show that the relations $w_f(x,y)$ of the point~\eqref{ENUM PO4 relations of G} are redundant, then $\mathcal G$ will turn out to be a {\it finitely presented} group.
Indeed, for every $w_f(x,y)$ we have: 
\begin{equation}
\label{EQ two variants for w(p,q)}
\begin{split}
& w_f(p,q)^e = \big( w_f(p,q) ,\;\; \mathcal R\, w_f(x,y) \big),
\\
& \big(w_f(p,q)^t\big)^e = \big( w_f(p,q)^t ,\;\, 1 \big),
\end{split}
\end{equation}
and since $w_f(p,q)$ and $w_f(p,q)^t$ are in $Q_{\mathcal X}$, they are equal. 
Since the left-hand sides of the equalities \eqref{EQ two variants for w(p,q)} turn out to be equal, the right-hand sides also are equal, and $\mathcal R\, w_f(x,y) = 1$ holds. 
Since in $\langle x,y \rangle / \mathcal R$ the identity element is the coset $\mathcal R$, we get that $w_f(x,y)$ is in $\mathcal R$, that is, $w_f(x,y)$ is a relation for the factor group $\langle x,y \rangle / \mathcal R = T_{\! \mathcal X}$, and we deduced this fact from the \textit{finitely} many relations for the points \eqref{ENUM PO1 relations of G}, \eqref{ENUM PO2 relations of G}, \eqref{ENUM PO3 relations of G}, \eqref{ENUM PO5 relations of G} only.
This argument may be easier with a very simple example. Say, for the sequence $f=(3,5,2)\in \mathcal X$ we have 
$w_f(x,y)=x^3 y^5 x^2$, from where $w_f(p,q)=p^3 q^5 p^2$ and:
$$ 
(p^3 q^5 p^2)^e 
=\sigma(p^3 q^5 p^2)
= \big( p^3 q^5 p^2 ,\;\; \mathscr{R} x^3 y^5 x^2 \big),
$$ 
\vskip-5mm
$$
\big( (p^3 q^5 p^2)^t\big)^e 
=\sigma \big( (p^3 q^5 p^2)^t\big) 
= \big( (p^3 q^5 p^2)^t ,\;\; 1 \big). 
$$ 
Since $p^3 q^5 p^2$ and $(p^3 q^5 p^2)^t$ are in $Q_{\mathcal X}$, they are equal, and then  $\mathscr{R} x^3 y^5 x^2=1$, that is, 
$\mathscr{R} x^3 y^5 x^2$ is the identity  element $1=1_{T_{\!\mathcal X}}$\! in  $\langle x,y \rangle / \mathcal R$, and so $x^3 y^5 x^2 \in \mathcal R$ indeed is a relation for the factor group $T_{\!\mathcal X}$.

\medskip
What we applied was a variation of the ``Higman rope trick'' used in \cite{Higman Subgroups of fP groups} and adopted elsewhere. In particular, it is utilized by Valiev in \cite{Valiev}. 
Lindon and Schupp use it in Valiev's interpretation in Section IV\,\!.\,\!7 of \cite{Lyndon Schupp}, see more in the discussion 
\cite{Higman rope trick} or in \cite{Leary Southampton}.

\medskip
Meanwhile, the proof of Theorem~\ref{TH Theorem B} has been completed.

\subsection{Equality $T_{\!\mathcal X}\!=\!T_{\!G}$, and the final embedding}
\label{SU Equality and the final embedding}

What is the relationship of the above group $T_{\!\mathcal X}$ with the  very similarly denoted $2$-generator group $T_{\!G}$ introduced in Section~\ref{SU The embedding alpha into a 2-generator group}? 

In Section~\ref{SU Using the 2-generator group to get the set X} the subset $\mathcal X$ of $\E$ was built  using the relations on two letters $x,y$ of some $2$-generator recursive group $T_{\!G}$ into which our initial recursive group $G$ was explicitly embedded via $\alpha:G\to T_{\!G}$ in \eqref{EQ embedding alpha}. 

Recalling how in Section~\ref{SU Using the 2-generator group to get the set X} the sequence (function) $f$ was written down from the 
relation
$w(x,y)$ via \eqref{EQ f_i occurs first}, and comparing this with how 
the words $w_f(x,y)$ were produced from $w_f(p,q)$ in Section~\ref{SU Theorem on embedding of TX} to get the normal closure 
$$
\mathcal R = \langle w_f(x,y) \;|\; f\! \in \mathcal X \rangle^{\langle x,y \rangle}
$$ along with the factor group $T_{\!\mathcal X} = \langle x,y \rangle / \mathcal R$, it is very easy to notice that the groups 
$T_{\!\mathcal X}$ and $T_{\!G}$, in fact, \textit{coincide}.

Say, if $w(x,y) = x^3 y^5 x^2$ is some relation of $T_{\!G}$, we ``coded'' it in Section~\ref{SU Using the 2-generator group to get the set X} by the function
$f\!=(3,5,2)$.
Using this $f$ we defined the word $a_f = a^{b_f}=a^{
b_0^3
b_1^5
b_2^2}$
in $F_3$.

Next, in sections \ref{SU If AX is benign then}\,--\,\ref{SU Writing KQX by generators and defining relations} we went from 
$a^{b_f}$ 
to $z^{w_f(r,\,s)}
=\,z^{
r^3
s^5
r^2}$\!\!,\, and then  
to $w_f(p,q)=p^3 q^5 p^2$\,\!.

In Section~\ref{SU The Higman Rope Trick}, just 
replacing $p,q$ by $x,y$,
we obtained the word 
$w_f(x,y)= x^3 y^5 x^2$  \textit{identical} to what we started from.

\smallskip
Eventually, the equality $T_{\!\mathcal X}\!=\!T_{\!G}$ allows us to embed the initial group $G$ into the finitely presented group $\mathcal G$ via the composition:
\begin{equation}
\label{EQ define varphi}
\varphi: G \to \mathcal G.
\end{equation} 
of the embedding $\alpha: G \to  T_G$ from \eqref{EQ embedding alpha} with the embedding $\beta: T_{\!\mathcal X} \to  \mathcal G$ from \eqref{EQ define beta}.

\begin{Remark}
\label{RE why embedding into 2-generator was good}
Coincidence $T_{\!\mathcal X}\!=\!T_{\!G}$ is one of the key points for the sake of which our analog of Higman embedding is by far simpler to build  \textit{for $2$-generator} groups, and it justifies why we first built the embedding $\alpha: G \to  T_{\!G}$ of our initial recursive group $G$ into a specific $2$-generator recursive group $T_{\!G}$ in \eqref{EQ define beta}, and only after that continued the whole process of construction of $\mathcal G$ for $T_{\!G}$.
\end{Remark}

Compare the above remark with Remark~\ref{RE One could try to produce X from Q directly} in Section~\ref{SU Using the 2-generator group to get the set X}, in which we stressed yet another advantage of the embedding $\alpha: G \to  T_{\!G}$, namely, the fact that it allows to get a by far simpler set of sequences $\mathcal X$.

\subsection{Writing $\mathcal G$ by generators and defining relations}
\label{SU Writing mathcal G by generators and defining relations} 

From notation in Section~\ref{SU The Higman Rope Trick} including the construction  \eqref{EQ definition of G} for $\mathcal G$, and the choice 
$$L_{Q_{\mathcal X}}\!\!
=
\mathscr{P}^{f_1 f_2}$$ 
in Section~\ref{SU If ZX is benign then Q(p,q) is also benign} we have:
\begin{equation}
\label{EQ relations of G}
\begin{split}
\mathcal G
& = \big\langle
\text{$m\!+\!20$ generators of $K_{Q_{\mathcal X}}$\! from \eqref{EQ relations KQX}};\;\,  t,\, x,y,\,  e
\mathrel{\;\;|\;\;} \\
& \hskip11mm 
\text{$n\! +\! 4m\! + \!k\! +\! 78$ relations of $K_{Q_{\mathcal X}}$\! from \eqref{EQ relations KQX}};\\
& \hskip11mm 
\text{$t$ fixes conjugates of\, 
$z, r,\, s, u;\, p, q, v$ by $f_1 f_2$}; \\ 
& \hskip11mm \text{$x, y$ commute with $t$}  \\ 
& \hskip41mm \text{and with generators of $K_{Q_{\mathcal X}}$}; \\ 
& \hskip11mm  
\text{$e$ sends $p,\, q,\, p^t\!,\, q^t$ to $px, qy, p^t\!,\, q^t$}
\big\rangle 
\end{split}
\end{equation}
where the values $m, n, k$ are those from Section~\ref{SU Writing KZX by generators and defining relations}.
The final group $\mathcal G$ has 
$$m+20+4
=
m+24$$ 
generators, and 
$$
n+ 4m+k+ 78
+7 
+ 2 \cdot (1+m+20)
+ 4
=
n+ 6m+k+ 131
$$
relations.

\medskip
It is an easy task to substitute into 
\eqref{EQ relations of G}
the ``nested'' generators and relations from \eqref{EQ relations KQX}, then from 
\eqref{EQ relations KZT+}, 
\eqref{EQ relations bar R}, 
\eqref{EQ relations KZX}, 
\eqref{EQ relations KQ1}, 
\eqref{EQ relations KQ} to get the full presentation of \,  $\mathcal G$. 

\smallskip 
Since the main goal of this work is to write $\mathcal G$ explicitly, let us do those routine steps also.
Namely, suppose like above $K_{\mathcal X}=\langle
\,Z
\;|\;
S
\rangle$ 
is the explicitly known finitely presented group produced for the benign subgroup $A_{\mathcal X}$ of $F_3$ at the end of Chapter~\ref{SE Theorem A and its proof steps} after a series of Higman operations \eqref{EQ Higman operations}, and the subgroup $L_{\mathcal X} \le K_{\mathcal X}$ is known by its explicitly given finitely many generators. 

\smallskip 
As it was remarked in Point~\ref{RE abc can be added} by Tietze transformations, we may assume $Z$ contains $a,b,c$. Hence, in the presentation below the generators $a,b,c$ do not have to be listed, as soon as $Z$ already is in the generating set.

\smallskip
Then $\mathcal G$ can be explicitly rewritten as:

\begin{equation}
\label{EQ relations of G FULL}
\begin{split}
\mathcal G
& = \big\langle
\,Z;\;\; z, r,\, s, l; 
\;  v_1, v_2; \;
w_1, w_2; \;
w_3, w_4; \; 
u, p, q, v; \; e_1,e_2;\\
& \hskip21mm h_1,h_2;\; f_1,f_2;\;  t,\, x,y,\,  e
\mathrel{\;\;|\;\;} S;\\
& \hskip11mm 
\text{$l$ sends $z, r,\, s$ to $z,\, s, r$};\\[-2pt]
& \hskip11mm 
\text{$z, r,\, s, l$ commute with generators $Z$};\\
& \hskip11mm 
\text{$v_1$ fixes $az,\; br,\; cl$;}\\
& \hskip11mm 
\text{$v_2$ fixes $z, r, s, l$ and the generators of  $L_{\!\mathcal X}$};\\
& \hskip11mm 
\text{$w_1$ fixes $a,b,c$};\\
& \hskip11mm 
\text{$w_2$ fixes conjugates of $a,b,c;\;  z, r,\, s, l$ by $v_1 \! v_2$}\\
& \hskip11mm 
\text{$w_3$ fixes $z, r, s, l$};\\
& \hskip11mm 
\text{$w_4$ fixes conjugates of $a,b,c;  z, r,\, s, l$ by $w_1$\! and\! $w_2$};\\
& \hskip11mm 
\text{$u^z=v$};\\
& \hskip11mm 
\text{$e_1$ sends 
$z,r,\, s,u; p,q,v$ to $z^r\!,\, r,\, s,u, p,q,v\cdot p$};\\[-3pt]
& \hskip11mm 
\text{$e_2$ sends 
$z,r,\, s,u; p,q,v$ to $z^s\!,\, r,\, s,u, p,q,v\cdot q$};\\
& \hskip11mm 
\text{$h_1$ fixes conjugates of $a,b,c;\;  z, r,\, s, l$ by $w_3 w_4$};\\
& \hskip11mm
\text{$h_2$ fixes 
$u,v$};\\
& \hskip11mm 
\text{$f_1$ fixes conjugates of $a,b,c;\;  z, r,\, s, l$ by $h_1$ and $h_2$};\\
& \hskip11mm 
\text{$f_2$ fixes 
$p,q$};\\
& \hskip11mm 
\text{$t$ fixes conjugates of\, 
$z, r,\, s, u;\, p, q, v$ by $f_1 f_2$}; \\ 
& \hskip11mm \text{$x, y$ commute with $t$ and with generators of $K_{Q_{\mathcal X}}$}; \\ 
& \hskip11mm  
\text{$e$ sends $p,\, q,\, p^t\!,\, q^t$ to $px, qy, p^t\!,\, q^t$}
\big\rangle.
\end{split}
\end{equation} 

An implementation of \eqref{EQ relations of G FULL} is done for the group $\Q$, and the generators and relations of a finitely presented  overgroup $\mathcal Q$ of $\Q$ are written down in Section~9.1 of \cite{On explicit embeddings of Q}.

\subsection{Embedding $G$ into the $2$-generator group $T_{\!\mathcal G}$}
\label{SU Embedding G into the 2-generator group TG}

The group $\mathcal G$ given in \eqref{EQ relations of G FULL} by $m\!+\!24$ generators can be replaced by a just $2$-generator finitely presented overgroup of $G$.
We use the values $m,n,k$ from Section~\ref{SU Writing KZX by generators and defining relations}:
$m$ is the number of generators $Z$ of the group $K_{\!\mathcal X}$ outputted by Theorem~\ref{TH Theorem A},
$n$ is the number of its defining relations in $S$, and $k$ is the number of generators for $L_{\!\mathcal X}$.

We again apply the method of \cite{Embeddings using universal words} outlined in Section~\ref{SU The embedding alpha into a 2-generator group}. 
To stress the similarity of the constructions denote a new free group of rank $2$ by
$\langle \x, \y \rangle$ with bold $\x, \y$, and inside it again consider some ``universal words''. This time we define $m+24$ such words:
$$
a_i(\x, \y)=\y^{(\x \y^{\,i})^{\,2}\, \y^{-1}} \y^{- \x}
\in \langle \x, \y \rangle, \quad i=1,2, \ldots , m\!+\!24.
$$
We respectively map $m+24$ generators $a,b,\ldots,e$ of the group $\mathcal G$ listed in \eqref{EQ relations of G FULL} to the words 
$a_1(\x, \y),\, a_2(\x, \y), \ldots , a_{m+24}(\x, \y)$ in $\langle \x, \y \rangle$.

Next in each of $n\! +\! 6m\! + \!k\! +\! 131$ relations listed in \eqref{EQ relations of G FULL} replace the letters $a,b,\ldots,e$ by the respective words $a_1(\x, \y),\, a_2(\x, \y), \ldots , a_{m+24}(\x, \y)$. For example,   
letting $a,b,c$ be the first three generators of $\mathcal G$, and $w_1$ be its $(m+7)$'th generator (because $|X|=m$), we replace the line ``$w_1$ fixes $a,b,c$''
in \eqref{EQ relations of G FULL} by the following line depending on two letters $\x, \y$ only:
$$
\text{$\y^{(\x \y^{\,(m+7)})^{2} \,\y^{-1}} \y^{- \x}$ fixes the elements $\y^{(\x \y^{\,i})^{\,2} \y^{-1}} \, \y^{- \x}$ for $i=1,2,3$}.
$$
The set $\RR$ of all such new words $a_j(\x, \y)$, 
$j=1,\ldots, n\! +\! 6m\! + \!k\! +\! 131$,
has in $\langle
\x, \y
\rangle$ a normal closure $\langle \,\RR\, \rangle^{\!\langle
\x, \y
\rangle}$\!,\, the factor group by which is a $2$-generator group which we in analogy with Section~\ref{SU The embedding alpha into a 2-generator group}  denote $T_{\!\mathcal G} = \langle
\x, \y
\rangle / \langle \,\RR\, \rangle^{\!\langle
\x, \y
\rangle}
= \langle\, \x,\y \mathrel{|} \RR \,\rangle$.
The map $\gamma$ sending the $i$'th generator of $\mathcal G$ from \eqref{EQ relations of G FULL} to the $i$'th word $a_i(\x,\y)$, and then to the $i$'th coset 
$\langle \,\RR\, \rangle^{\!\langle
\x, \y
\rangle} \, a_i(\x,\y) \in T_{\!\mathcal G}$ can according to Theorem~\ref{TH universal embedding} be continued to an injective embedding $\gamma : \mathcal G \to T_{\!\mathcal G}$. 

\medskip
Adding $\gamma$ to the already constructed embedding $\alpha$ from Section~\ref{SU The embedding alpha into a 2-generator group} and $\beta$ from Section~\ref{SU Equality and the final embedding} we get an embedding $\psi:G\to T_{\!\mathcal G}$ of the initial group $G$ into $T_{\!\mathcal G}$ as the composition:
\begin{equation}
\label{EQ composition psi}
G 
\xrightarrow{\;\; \alpha \;\;}
T_{\!G}
\xrightarrow{\;\; \beta \;\;}
\mathcal G
\xrightarrow{\;\; \gamma \;\;}
T_{\!\mathcal G}.
\end{equation}
The overgroup $T_{\!\mathcal G}$ has just $2$ generators, and it can be defined by $n\! +\! 6m\! + \!k\! +\! 131$  relations:

\begin{Corollary}
\label{CO G can be embedded into 2-generator TG}
In the above notation the composition $\psi:G\to T_{\!\mathcal G}$ of three embeddings $\alpha, \beta, \gamma$ is an explicit embedding of the recursive group $G$ into a $2$-generator group $T_{\!\mathcal G}$.
\end{Corollary}

As an application of Corollary~\ref{CO G can be embedded into 2-generator TG}, we embedded the group $\Q$ into a $2$-generator group  $T_{\!\mathcal Q}$ with explicitly listed finitely many relation in Section~9.2 of \cite{On explicit embeddings of Q}.

\vskip24mm

\noindent 
E-mail:
\href{mailto:v.mikaelian@gmail.com}{v.mikaelian@gmail.com}
$\vphantom{b^{b^{b^{b^b}}}}$ \\
\noindent 
Web: 
\href{https://www.researchgate.net/profile/Vahagn-Mikaelian}{researchgate.net/profile/Vahagn-Mikaelian}

\end{document}